\documentclass[11pt]{article}       
\usepackage{geometry}               
\usepackage{graphicx}
\usepackage{geometry}
\usepackage{caption}
\usepackage{subcaption}
\usepackage{mathrsfs}
\usepackage{comment}
\usepackage{amsmath,esint} 
\usepackage{amssymb}  
\usepackage{amsthm}  
\usepackage{bm}
\usepackage{color}
\usepackage{array}
\usepackage{cite}
\usepackage{hyperref}
\usepackage{comment}

\newtheorem{proposition}{Proposition}
\newtheorem{defn}{Definition}
\newtheorem{corollary}{Corollary}
\newtheorem{lemma}{Lemma}
\newtheorem{theorem}{Theorem}
\newtheorem{remark}{Remark}

\numberwithin{equation}{section}

\newcommand{\N}{\mathbb{N}}

\newcommand{\R}{\mathbb{R}}

\title{Existence of weak solutions to volume-preserving mean curvature flow with obstacles} 

\newcommand{\footremember}[2]{%
    \footnote{#2}
    \newcounter{#1}
    \setcounter{#1}{\value{footnote}}%
}

\author{
  Jiwoong Jang \footremember{trailer}{Department of Mathematics, University of Maryland-College Park (jjang124@umd.edu)},
  }

\date{}
\providecommand{\keywords}[1]{\textbf{Key words.} #1}
\providecommand{\MSC}[1]{\textbf{MSC codes.} #1}

\begin{document}
\maketitle

\pagestyle{myheadings}
\thispagestyle{plain}

\begin{abstract}

We prove the existence of global-in-time weak solutions to volume-preserving mean curvature flow with in the presence of obstacles by the phase field method in all dimensions. Namely, we prove the convergence of solutions to the Allen-Cahn equation with a multiplier to a weak solution to the flow. The choice of the multiplier is motivated from \cite{MSS16,KK20,T23}, which enables us to complete the comparison between the multiplier and the forcing that stops the intrusion into the obstacle. We also prove the vanishing of the discrepancy measure by dealing with the forcing term that is now spatially dependent due to the obstacles.

\end{abstract}


\vspace{0.2cm}

\MSC{35K55, 53E10.}

\section{Introduction}\label{sec:introduction}
\subsection{Settings and goals}\label{subsec:setting-goal}
Let $d\geq2$ be an integer, $T>0$, and $\Omega:=\left(\mathbb{R}/\mathbb{Z}\right)^d$. Suppose that an open subset $U_t\subset\Omega$ has a smooth boundary $M_t=\partial U_t$ for each $t\in[0,T)$. The family $\left\{U_t\right\}_{t\in[0,T)}$ of open subsets is said to evolve by \emph{volume-preserving mean curvature flow} if
\begin{align}\label{eq:VPMCF}
v=h-\lambda \nu\qquad\text{on }M_t\quad\text{for }t\in(0,T).
\end{align}
Here, $v$ is the normal velocity vector on $M_t$, $h$ is the mean curvature vector of $M_t$, $\nu$ is the inward unit normal vector on $M_t$, and $\lambda$ is the multiplier responsible for the volume constraint given by
\begin{align*}
\lambda=\lambda(t)=\frac{1}{\mathcal{H}^{d-1}(M_t)}\int_{M_t}h\cdot\nu d\mathcal{H}^{d-1}.
\end{align*}
By the volume constraint we mean that with the $\lambda$ above, the total volume is preserved over time, that is, $\mathcal{L}^d(U_t)=\mathcal{L}^d(U_0)$ for $t\in[0,T)$, where $\mathcal{L}^d$ is the $d$-dimensional Lebesgue measure.

Let $O_+$ and $O_-$ be open subsets of $\Omega$ with disjoint closures and let $O:=O_+\cup O_-$. We say a family $\{U_t\}_{t\in[0,T)}$ of open subsets with smooth boundaries $M_t$ is a solution to \emph{the obstacle problem for volume-preserving mean curvature flow} if \eqref{eq:VPMCF} is satisfied on $\Omega\setminus\overline{O}$, $O_+\subset U_t$, $O_-\cap U_t=\emptyset$ for $t\in[0,T)$. The volume-preserving mean curvature flow with obstacles arises as a natural model for cell motility (see \cite{ESV12,MBRZ16} and the references therein). In this paper, we establish the global-in-time existence of weak solutions to mean curvature flow with volume constraint and obstacles in all dimensions $d\geq2$. To the best of the author's knowledge, this paper is the first work on the mean curvature flow in presence of \emph{both} volume constraint condition and obstacles.

\subsection{Prior work and background}\label{eq:background-literature}
We give a list (by no means complete) of literature on related problems. We first present prior works on the obstacle problem without volume constraint and next on the volume-preserving mean curvature flow without obstacles.

\medskip

The short-time existence of $C^{1,1}$ solutions to the obstacle (without volume constraint) when the obstacle has a compact $C^{1,1}$ boundary was established when $d=2$ in \cite{ACN12} and when $d\geq2$ in \cite{MN15}. Moreover, \cite{ACN12} showed the global-in-time existence of weak solutions by minimizing movement schemes, and \cite{MN15} showed the well-posedness of viscosity solutions in graph setting. In general case in the level-set framework, \cite{M14} showed the well-posedness of viscosity solutions to
\begin{align*}
u_t+\mathrm{tr}\left\{\left(I_d-\frac{Du}{|Du|}\otimes\frac{Du}{|Du|}\right)D^2u\right\}+A|Du|=0\qquad\text{with}\quad\psi^-\leq u\leq\psi^+.
\end{align*}
Here, $I_d$ is the $d\times d$ identity matrix, $\psi^{\pm}$ are given uniformly continuous functions that represent the obstacles, and $A$ is a given Lipschitz continuous function. See \cite{GTZ19} for the well-posedness of Lipschitz continuous viscosity solutions and the large-time behavior with a constant forcing. When $A\equiv0$, the large-time profile was studied in \cite{M24} with the two-person game theory \cite{KS06} in 2-d. 

The \emph{Allen-Cahn equation} has also been taken as a phase field for the obstacle problem. The Allen-Cahn equation, originally from \cite{AC79}, reads
\[
\partial_t\phi^{\varepsilon}=\Delta\phi^{\varepsilon}-\frac{1}{\varepsilon^2}W'(\phi^\varepsilon),
\]
where the standard double-well potential $W$ is given by $W(r):=\frac{1}{2}(1-r^2)^2$. Formal derivation that the transition layer of $\phi^{\varepsilon}$ converges to interfaces moving by mean curvature as $\varepsilon\to0$ was justified by \cite{I93} with the Brakke's motion by mean curvature \cite{B78}. We refer to the references in \cite{C96} for the extensive literature of the works for the Allen-Cahn equation.


For the obstacle problem, \cite{BP12} used a penalized double-well potential in the Allen-Cahn equation. Paper \cite{T21} took the standard potential $W$ to study the obstacle problem, which was seen as the mean curvature motion with (an approximation of) the forcing term
\begin{align*}
g(x):=\begin{cases}
0, \qquad\qquad &\text{for}\quad x\in\Omega\setminus\overline{O},\\
\frac{d}{R_0}, \qquad\qquad &\text{for}\quad x\in\overline{O_+},\\
-\frac{d}{R_0}, \qquad\qquad &\text{for}\quad x\in\overline{O_-}.
\end{cases}
\end{align*}
On the boundary $\partial O_{\pm}$ of the obstacle, the forcing $\pm\frac{d}{R_0}$ ``wins" the curvature effect $\kappa$ so that the surfaces moving by mean curvature do not intrude the obstacle. This is rigorously justified with the maximum principle. We note that the maximum principle serves as the base principle in the viscosity solution approach described above.

\medskip

The difficulty for us occurs when we consider the volume constraint as volume-preserving mean curvature flow does not enjoy the maximum principle. If $\lambda(t)$ is a multiplier responsible for the volume constraint, then the choice of the forcing should be (an approximation of)
\begin{align*}
g(x,t):=\begin{cases}
\lambda(t), \qquad\qquad &\text{for}\quad x\in\Omega\setminus\overline{O},\\
\frac{d}{R_0}, \qquad\qquad &\text{for}\quad x\in\overline{O_+},\\
-\frac{d}{R_0}, \qquad\qquad &\text{for}\quad x\in\overline{O_-}.
\end{cases}
\end{align*}
Unlike the obstacle problem without the volume constraint, it is unclear whether or not the forcing $\pm\frac{d}{R_0}$ ``wins" the multiplier $\lambda(t)$ (plus the curvature $\kappa$) since a regularity of $\lambda(t)$ is unknown in general. This is our main difficulty and will be explained in the context of recent studies on volume-preserving mean curvature flow.

\medskip

We briefly give prior works on volume-preserving mean curvature flow \eqref{eq:VPMCF}. The existence of global-in-time classical solutions to \eqref{eq:VPMCF} with convex smooth initial data was proved by \cite{G86} in 2-d and by \cite{H87} in all dimensions. The short-time existence of classical solutions to \eqref{eq:VPMCF} with smooth initial data (possibly nonconvex) was proved by \cite{ES98}. In general, classical solutions to \eqref{eq:VPMCF} can develop singularities in finite time (see \cite{EI05,MS00}).

\medskip

Paper \cite{RS92} proposed the Allen-Cahn equation with volume conservation that is given by
\begin{align}\label{eq:RS}
\partial_t\phi^{\varepsilon}=\Delta\phi^{\varepsilon}-\frac{1}{\varepsilon^2}W'(\phi^\varepsilon)+\frac{1}{\varepsilon}\lambda^{\varepsilon}(t),    
\end{align}
where $\lambda^{\varepsilon}(t)$ is a Lagrange multiplier that enforces $\int \phi^{\varepsilon}(t,x)dx=\int \phi^{\varepsilon}(0,x)dx$. Radially symmetric solutions to \eqref{eq:RS} were studied in \cite{BS97}, and the transition layer of the solution to \eqref{eq:RS} converges to the classical solution under the assumption that the classical solution exists to \eqref{eq:VPMCF}. Paper \cite{LS18} proved the convergence to a multiphase volume-preserving mean curvature flow of solutions to a multiphase mass-conserving Allen-Cahn equation under a natural energy assumption for the solutions to \eqref{eq:RS}.

Another model with a Lagrange multiplier concentrated on the transition layer was suggested in \cite{G97} (see also \cite{BB11}), which reads
\begin{align}\label{eq:golovaty}
\partial_t\phi^{\varepsilon}=\Delta\phi^{\varepsilon}-\frac{1}{\varepsilon^2}W'(\phi^\varepsilon)+\frac{1}{\varepsilon}\lambda^{\varepsilon}(t)\sqrt{2W(\phi^{\varepsilon})}.
\end{align}
Under the assumption that \eqref{eq:VPMCF} has a classical solution, \cite{AA14} proved the convergence of solutions to \eqref{eq:golovaty} as $\varepsilon\to0$ toward the solution to \eqref{eq:VPMCF}, and \cite{KL24} recently obtained a rate of the convergence under the same assumption. Without the energy assumption and the existence of classical solutions in the limit flow, \cite{T17} verified the convergence to a weak solution to \eqref{eq:VPMCF} of solutions to the model \eqref{eq:golovaty} by \cite{G97} when $d=2,3$, and then \cite{T23} generalized the conclusion for all dimensions. Here, the notion of weak solutions to \eqref{eq:VPMCF} is the called $L^2$-flow developed in \cite{MR08} and is similar to the Brakke's flow \cite{B78}. We adopt this notion in this paper and see Definition \ref{def:L2-flow}.

\medskip

Various studies on \eqref{eq:VPMCF} have been suggested. Under a natural assumption for approximate solutions in terms of energies, \cite{MSS16} proved the global existence of the weak solution in the sense of distributions in $d\leq7$ using minimizing movement schemes (see also \cite{LS95}). The convergence of thresholding schemes to distributional BV solutions to \eqref{eq:VPMCF} was proved in \cite{LS17} under a similar energy assumption. The weak-strong uniqueness for gradient-flow calibrations, which are also relevant that for $L^2$-flows, was proved in \cite{L24}. When the initial data satisfies a geometric condition called the $\rho$-reflection property, \eqref{eq:VPMCF} was studied with viscosity solutions in \cite{KK20}. Seeing a solution to \eqref{eq:VPMCF} as diffused interfaces in view of the Allen-Cahn equation, the exponential convergence toward single ``diffused balls" as time goes by was recently proved in \cite{BMR24}. See also the recent study \cite{MR24} on \eqref{eq:VPMCF} as a singular limit of the Patlak-Keller-Segel system.

\medskip

Paper \cite{T23}, the most relevant reference to this paper, considered a phase field model that is mainly motivated by the approximate problem of \eqref{eq:VPMCF} considered in \cite{MSS16,KK20}. The motivation was due to the reason from an additional regularity of the multiplier that is explained in the following. The approximate problem with a small parameter $\delta\in(0,1)$ is
\begin{align}\label{eq:approximate-VPMCF}
v=h-\lambda^{\delta} \nu\qquad\text{on }M^{\delta}_t\quad\text{for }t\in(0,T),
\end{align}
where $M_t^{\delta}$ is a smooth boundary of an open set $U_t^{\delta}$ and
\begin{align}\label{eq:approximate-multiplier}
    \lambda^{\delta}(t):=\frac{1}{\delta}\left(\mathcal{L}^d(U^{\delta}_0)-\mathcal{L}^d(U^{\delta}_t)\right).
\end{align}
The solution to \eqref{eq:approximate-VPMCF} is a $L^2$-gradient flow of the energy
\begin{align*}
E^{\delta}(t):=\mathcal{H}^{d-1}(M^{\delta}_t)+\frac{1}{2\delta}\left(\mathcal{L}^d(U^{\delta}_0)-\mathcal{L}^d(U^{\delta}_t)\right)^2.
\end{align*}
The choice $\lambda^{\delta}$ with a small number $\delta\in(0,1)$ penalizes the volume $\mathcal{L}^d(U^{\delta}_t)$ at time $t$ different from the initial volume $\mathcal{L}^d(U^{\delta}_0)$.

An advantage of considering \eqref{eq:approximate-VPMCF} is from the gradient flow structure which gives
\begin{align}\label{eq:approximate-VP}
\left(\mathcal{L}^d(U^{\delta}_0)-\mathcal{L}^d(U^{\delta}_t)\right)^2\leq2\delta E^{\delta}(t)\leq2\delta E^{\delta}(0)\leq C\delta,\quad\text{or equivalently,}\quad|\lambda^{\delta}(t)|\leq C\delta^{-1/2}.
\end{align}
Indeed, this approximate volume-preserving property \eqref{eq:approximate-VP} yields an additional regularity of $\lambda^{\delta}$ in the sense that if we know a natural $L^2$-estimate $\sup_{\delta\in(0,1)}\int_0^T|\lambda^{\delta}(t)|^2dt\leq C$, then for a.e. $t\in(0,T)$,
\begin{align}\label{eq:advantage-for-takasao}
\liminf_{\delta\to0}\frac{1}{2\delta}\left(\mathcal{L}^d(U^{\delta}_0)-\mathcal{L}^d(U^{\delta}_t)\right)^2=\liminf_{\delta\to0}\frac{\delta}{2}|\lambda^{\delta}(t)|^2=0.
\end{align}
Paper \cite{T23} obtained the monotonicity of the perimeter functional $\mathcal{H}^{d-1}(M_t)$ (as the limit of the energy dissipation $E^{\delta}(t)\leq E^{\delta}(0)$ as $\delta\to0$) by utilizing \eqref{eq:approximate-VP}, \eqref{eq:advantage-for-takasao} as above, and generalized the result of \cite{T17} to all dimensions.

\medskip

In this paper, we also utilize the estimate \eqref{eq:approximate-VP} in two ways. First, \eqref{eq:approximate-VP} plays a crucial role in the comparison between the forcing $\pm\frac{d}{R_0}$ on the obstacles $O_{\pm}$ that stops the intrusion and the multiplier $\lambda^{\varepsilon}(t)$ (given as in \eqref{eq:lambda}) that drives the surfaces on $\Omega\setminus O_{\pm}$, which is the main difficulty aforementioned. See Proposition \ref{prop:VP-and-multiplier} and Lemma \ref{lem:barrier-functions}. Second, we use the property \eqref{eq:approximate-VP} in a similar manner to \cite{T23} when we prove the vanishing of the discrepancy measure $|\xi^{\varepsilon}|\to0$ (see Sections \ref{sec:discrepancy-density-ratio-estimates} for the definition of $\xi^{\varepsilon}$ and Theorem \ref{thm:vanishing-xi}). However, we need further arguments since the nonpositivity of $\xi^{\varepsilon}$ does not directly hold from the maximum principle as in \cite{I93,T23}, which is another subtle issue from considering the obstacles. We need to control the term from $\xi^{\varepsilon}$ in the monotonicity formula (Proposition \ref{prop:monotonicity-formula}). This term is indeed controlled by establishing upper bounds of $\xi^{\varepsilon}$ and the density ratio (Propositions \ref{prop:upper-bound-discrepancy}, \ref{prop:upper-bound-density-ratio}), for which we critically use the estimate \eqref{eq:approximate-VP} (in the form of \eqref{eq:L-infty-g-ep-rescaled}). Regarding this technical issue when we prove $|\xi^{\varepsilon}|\to0$ (Theorem \ref{thm:vanishing-xi}), we refer to the more detailed account in Sections \ref{sec:discrepancy-density-ratio-estimates} and \ref{sec:rectifiability-integrality}.

\subsection{Notations and precise setup of phase field}\label{subsec:setup}

We arrange the notations including the ones already mentioned before we introduce our main equation.

\medskip

We denote by $\Omega:=\left(\mathbb{R}/\mathbb{Z}\right)^d$, $d\geq2$. An initial surface is denoted by $M_0=\partial U_0$ with interior $U_0$ and is assumed to be $C^1$ throughout the paper. We denote by $O_+$ an obstacle that is initially inside the boundary $M_0$ so that $\overline{O_+}\subset U_0$, and denote by $O_-$ an obstacle that is initially outside the boundary $M_0$ so that $\overline{O_-}\subset\Omega\setminus\overline{U_0}$. We assume that $M_0\cap \partial O=\emptyset$ where $O:=O_+\cup O_-$, that is, the initial surface $M_0$ is not in contact with the obstacles. We mention that $C^{2,\beta}$-regularity of $\partial O$ for some $\beta\in(0,1)$ is assumed for technical reasons throughout the paper, and we will make remarks this assumption later (see Remarks \ref{rmk:C2beta1}) and \ref{rmk:C2beta2}). Basic properties from the regularity assumptions of $M_0$ and $O$ are given in Section \ref{sec:settings}.

\medskip

We let $W(r):=\frac{1}{2}(1-r^2)^2$ and $k(r):=\int_0^r\sqrt{2W(u)}du=r-\frac{1}{3}r^3$ for $r\in\mathbb{R}$. Our main equation is given by
\begin{align}\label{eq:phi-epsilon}
\begin{cases}
\varepsilon\partial_t\phi^{\varepsilon}=\varepsilon\Delta\phi^{\varepsilon}-\frac{1}{\varepsilon}W'(\phi^\varepsilon)+g^{\varepsilon}\sqrt{2W(\phi^{\varepsilon})} \quad &\text{in}\quad\Omega\times(0,\infty),\\
\phi^{\varepsilon}(\cdot,0)=\phi^{\varepsilon}_0 \quad &\text{on}\quad \Omega.
\end{cases}
\end{align}
To give the precise definition of the forcing $g^{\varepsilon}$, we let $\eta\in C^{\infty}(\mathbb{R})$ be a nonincreasing smooth cut-off function such that $0\leq\eta\leq1$, $\eta(r)\equiv1$ near $r=0$ and when $r\leq0$, $\eta(r)\equiv0$ near $r=1$ and when $r\geq0$. Let $\eta_{\varepsilon}(r):=\eta\left(\left(\frac{1}{2}\sqrt{\varepsilon}\right)^{-1}r\right)$ for $r\in\mathbb{R}$. For the signed distance function $s(x)$ from $\partial O$ (positive in $O$),
\begin{align*}
s(x):=\begin{cases}
\mathrm{dist}(x,\partial O) \qquad\qquad &\text{for}\quad x\in O,\\
-\mathrm{dist}(x,\partial O) \qquad\qquad &\text{for}\quad x\notin O,
\end{cases}
\end{align*}
the cut-off function $x\mapsto\eta_{\varepsilon}(s(x))$ is $C^{2,\beta}$ for $\varepsilon\in\left(0,\frac{1}{2}\right)$ sufficiently small. With $\alpha\in\left(0,1\right)$ fixed throughout the paper, we define the multiplier $\lambda^{\varepsilon}(t)$ and the forcing term $g^{\varepsilon}(x,t)$ as follows:
\begin{align}\label{eq:lambda}
\lambda^{\varepsilon}(t):=\frac{1}{\varepsilon^{\alpha}}\int_{\left\{x\in\Omega:s(x)\leq\frac{1}{2}\sqrt{\varepsilon}\right\}}\eta_{\varepsilon}(s(x))\left(k(\phi_0^{\varepsilon}(x))-k(\phi^{\varepsilon}(x,t))\right)dx,
\end{align}
and
\begin{align}\label{eq:forcing}
g^{\varepsilon}(x,t):=\begin{cases}
\lambda^{\varepsilon}(t)\cdot\eta_{\varepsilon}(s(x)) \qquad\qquad &\text{for}\quad x\in\Omega,\ s(x)\leq\frac{1}{2}\sqrt{\varepsilon},\\
\frac{d}{R_0}\left(1-\eta_{\varepsilon}\right)\left(s(x)-\frac{1}{2}\sqrt{\varepsilon}\right) \qquad\qquad &\text{for}\quad x\in O_+,\ \mathrm{dist}(x,\partial O_+)\geq\frac{1}{2}\sqrt{\varepsilon},\\
-\frac{d}{R_0}\left(1-\eta_{\varepsilon}\right)\left(s(x)-\frac{1}{2}\sqrt{\varepsilon}\right) \qquad\qquad &\text{for}\quad x\in O_-,\ \mathrm{dist}(x,\partial O_-)\geq\frac{1}{2}\sqrt{\varepsilon}.
\end{cases}
\end{align}

\begin{figure}[htbp]
	\begin{center}
            \includegraphics[height=7cm]{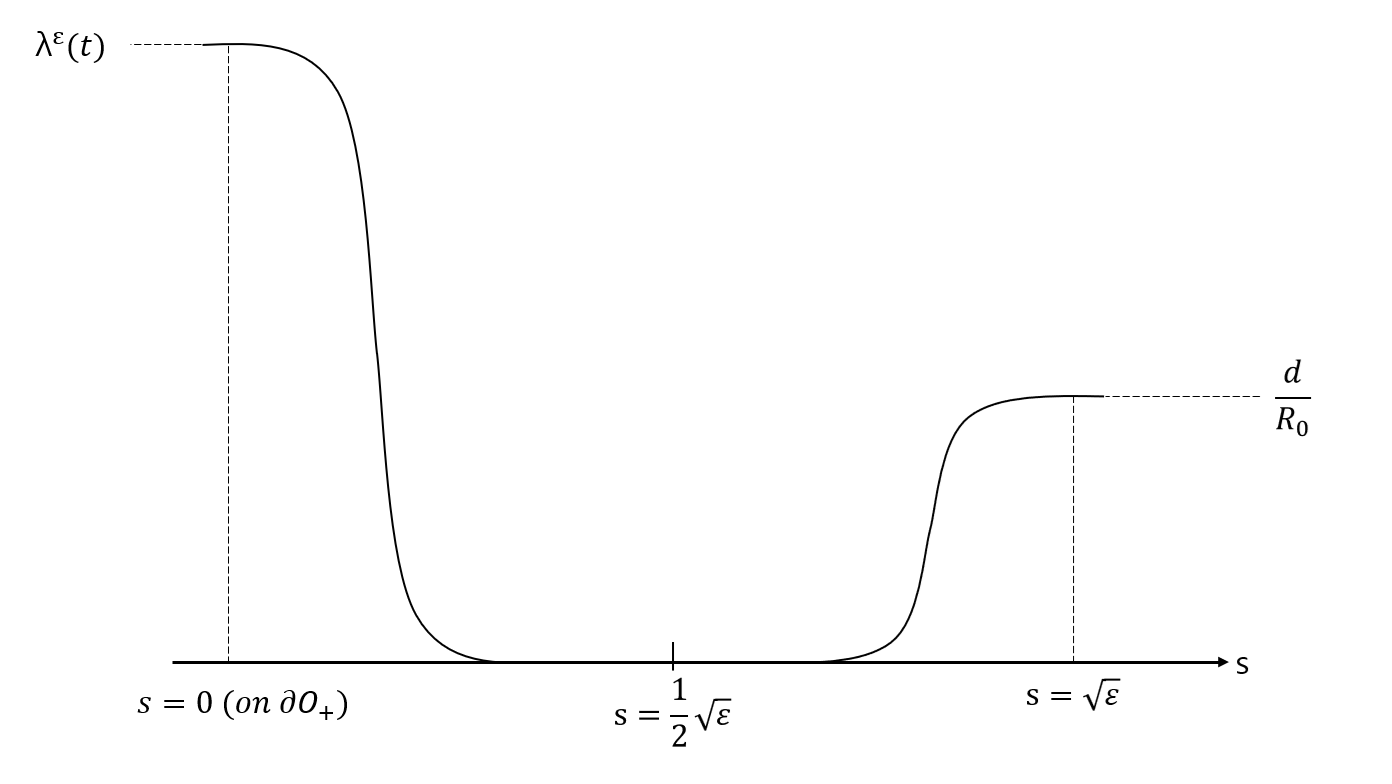}
		\vskip 0pt
		\caption{The forcing term $g^{\varepsilon}$ near $\partial O_+$.}
        \label{fig:forcing}
	\end{center}
\end{figure}

\medskip

The choices \eqref{eq:lambda} and \eqref{eq:forcing} are due to the expectations that $O_+\subset U_t$, $O_-\subset\Omega\setminus U_t$ for all time and thus that the volume preservation only on $\Omega\setminus O$ suffices. We remark that as typically happens to Allen-Cahn-type equations, $\phi^{\varepsilon}$ assumes values close enough to $\pm1$, for which we have $k(\phi^{\varepsilon})\approx\frac{2}{3}\phi^{\varepsilon}$, whose integration then represents (a constant multiple of) volume of the regions $\{\phi^{\varepsilon}\approx\pm1\}$.

The energy associated to \eqref{eq:phi-epsilon} is given by
\begin{align}\label{eq:energy}
E^{\varepsilon}(t):=\int_{\Omega}\left(\frac{\varepsilon|\nabla\phi^{\varepsilon}|^2}{2}+\frac{W(\phi^{\varepsilon})}{\varepsilon}\right)dx+\frac{1}{2\varepsilon^{\alpha}}\left(\int_{\left\{x\in\Omega:s(x)\leq\frac{1}{2}\sqrt{\varepsilon}\right\}}\eta_{\varepsilon}(s(x))\left(k(\phi_0^{\varepsilon})-k(\phi^{\varepsilon})\right)dx\right)^2,
\end{align}
where the solution $\phi^{\varepsilon}$ to \eqref{eq:phi-epsilon} is evaluated at time $t\geq0$.

The precise construction of initial data $\{\phi^{\varepsilon}_{0}\}_{\varepsilon\in(0,1)}$ is given in Subsection \ref{subsec:initial-data}. We give the basic properties of the setup in Section \ref{sec:settings} such as the well-preparedness of $\{\phi^{\varepsilon}_{\text{0}}\}_{\varepsilon\in(0,1)}$ and energy dissipation.

\subsection{Basic definitions and main results}\label{subsec:basic-defns-main-results}

Some definitions from the geometric measure theory are provided (see \cite{S83} for more details).

For $1\leq k< d$, we let $\mathbb{G}(d,k)$ be the space of $k$-dimensional linear subspaces of $\mathbb{R}^d$. We say that $V$ is a \emph{$k$-varifold} on $\Omega$ if it is a Radon measure on $\Omega\times\mathbb{G}(d,k)$. For a $k$-varifold $V$ on $\Omega$, we define the \emph{weight measure $\|V\|$} of $V$ by
\begin{align*}
\|V\|(\phi):=\int_{\Omega\times\mathbb{G}(d,k)}\phi(x)dV(x,S)\qquad\text{for }\phi\in C_c(\Omega).
\end{align*}
We say that a $k$-varifold $V$ on $\Omega$ is \emph{$k$-rectifiable} if there are a $k$-rectifiable set $M\subset\Omega$ and $\theta\in L^1_{loc}(\mathcal{H}^k\lfloor_M;\mathbb{R}_{>0})$ such that $V=\theta(x)\mathcal{H}^k\lfloor_{M}(dx)\otimes\delta_{T_xM}$, that is,
\begin{align*}
V(\phi)=\int_{M} \phi(x, T_{x} M) \theta(x) d \mathcal{H}^{k}(x) \quad \text { for any } \phi \in C_{c}(\Omega\times\mathbb{G}(d,k)).
\end{align*}
Here, $T_xM$ denotes the approximate tangent space of $M$ at $x$. We say a rectifiable varifold $V$ is \emph{integral} if $\theta$ is integer-valued $\mathcal{H}^k\lfloor_M$-almost everywhere, and that a Radon measure $\mu$ on $\Omega$ is $k$-integral if there is a $k$-integral vafifold $V$ such that $\mu=\|V\|$. In this case, we define the approximate tangent space $T_x\mu$ of $\mu$ at $x$ by $T_x\mu:=T_xM$, where a $k$-rectifiable set $M\subset\Omega$ is found as above.

We define the first variation of a $k$-varifold $V$ on $\Omega$ by
\begin{align*}
\delta V(\vec{\phi}):=\int_{\Omega\times\mathbb{G}(d,k)} \nabla \vec{\phi}(x) : S d V(x, S) \quad \text { for any } \zeta \in C_{c}^{1}\left(\Omega ; \mathbb{R}^{d}\right).
\end{align*}
Here, for two square matrices $A,B$, the notation $A:B$ denotes the trace of $AB^t$, $B^t$ the transpose of $B$. The matrix $S$ in the above definition, by abuse of notations, denotes the matrix that corresponds to the orthogonal projection of $\mathbb{R}^d$ onto the subspace $S$. We say that a $k$-varifold $V$ on $\Omega$ has a \emph{generalized mean curvature vector $h$} if
\begin{align*}
\delta V(\zeta)=-\int_{\Omega} \zeta(x) \cdot h(x) d\|V\|(x) \quad \text { for any } \zeta \in C_{c}^{1}\left(\Omega ; \mathbb{R}^{d}\right).
\end{align*}
For the weight measure $\mu=\|V\|$, we say that the Radon measure $\mu$ on $\Omega$ has a generalized mean curvature vector $h$ if the varifold $V$ has a generalized mean curvature vector $h$.
\medskip

With the notations introduced above, the weak notion of solutions on an open subset $U\subset\Omega$ considered in this paper is as follows:

\begin{defn}[$L^2$-flow \cite{MR08}]\label{def:L2-flow}
Let $\{\mu_t\}_{t>0}$ be a family of $(d-1)$-integral Radon measures on $U$ such that $\mu_t$ has a generalized mean curvature vector $h\in L^2(\mu_t;\mathbb{R}^d)$ for a.e. $t\in(0,\infty)$. We call $\{\mu_t\}_{t>0}$ an $L^2$-flow on $U$ if there exist a constant $C>0$ and a vector field $v\in L^2_{loc}\left(0,\infty;L^2(\mu_t)^d\right)$ such that
\[
v(x,t)\perp T_x\mu_t\quad\text{for $\mu$-a.e. }(x,t)\in U\times(0,\infty)
\]
and
\[
\left|\int_0^{\infty}\int_{U}(\eta_t+\nabla\eta\cdot v)d\mu_tdt\right|\leq C\|\eta\|_{C^0(U\times(0,\infty))}
\]
for any $\eta\in C^1_c(U\times(0,\infty))$. A vector field $v\in L^2_{loc}\left(0,\infty;L^2(\mu_t)^d\right)$ satisfying the above conditions is called a generalized velocity vector.
\end{defn}

\medskip

Let $\sigma:=\int_{-1}^1\sqrt{2W(u)}du$. For $\varepsilon\in(0,1),\ t\geq0$, define the Radon measure $\mu_t^{\varepsilon}$ by
\begin{align}\label{eq:def-mu}
    \mu_t^{\varepsilon}(\eta):=\frac{1}{\sigma}\int_{\Omega}\left(\frac{\varepsilon|\nabla\phi^{\varepsilon}|^2}{2}+\frac{W(\phi^{\varepsilon})}{\varepsilon}\right)\eta dx\qquad\text{for all }\eta\in C_c(\Omega),
\end{align}
and we let $d\mu^{\varepsilon}:=d\mu_t^{\varepsilon}dt$ on $\Omega\times[0,\infty)$. We recall that for $f\in L^1(\Omega)$, the distributional derivative $\nabla f$ is defined by $$\int_{\Omega}\varphi\,d[\nabla f]=-\int_{\Omega}f\,\mathrm{div}(\varphi)\,dx\qquad\text{for }\varphi\in C_c^1(\Omega;\mathbb{R}^d),$$
with the variation measure $\|\nabla f\|$ defined by $$\|\nabla f\|(V):=\sup\left\{\int_{\Omega}f\mathrm{div}(\varphi)\,dx\,:\,\varphi\in C_c^1(V;\mathbb{R}^d),\,\|\varphi\|_{L^{\infty}(\Omega)}\leq1\right\}\quad\text{for open sets }V\subset\subset\Omega.$$ If $\|\nabla f\|(\Omega)<\infty$, the variation measure is a Radon measure, and also, the derivative $\nabla f$ is a vector-valued Radon measure.

In case that $f=\chi_E$ is a characteristic function with support $E\subset\Omega$, we say that $E$ is of \emph{finite perimeter} if $\|\nabla\chi_E\|(\Omega)<+\infty$. For a subset $E\subset\Omega$ of finite perimeter, \emph{the reduced boundary} $\partial^{\ast}E$ is defined by $$\partial^{\ast}E=\left\{x\in\mathrm{supp}(\nabla\chi_E)\,:\,\nu(x) := \lim_{r \downarrow 0} \frac{\nabla\chi_E(B_r(x))}{\|\nabla\chi_E\|(B_r(x))} \text{ exists and } |\nu(x)| = 1\right\},$$
where $\mathrm{supp}(\nabla\chi_E)$ is the set $\{x\in\Omega\,:\,0<|E\cap B_r(x)|<\omega_dr^d\,\,\,\text{for }r>0\}$, and $\nu=\nu(\cdot)$ is called the measure-theoretic inner unit normal vector of $E$ on $\partial^{\ast}E$.

\medskip

Now we are ready to state our main theorem. 

\begin{theorem}\label{thm:main-thm}
Suppose that $d\geq2$ and an open set $U_0\subset\Omega$ has a $C^1$ boundary $M_0$ such that $\overline{O_+}\subset U_0$, $\overline{O_-}\subset \Omega\setminus\overline{U_0}$. Suppose that $O:=O_+\cup O_-$ has $C^{2,\beta}$ boundary for some $\beta\in(0,1)$. Let $\phi^{\varepsilon}$ be the solution to \eqref{eq:phi-epsilon} with initial data $\phi_{\textrm{0}}^{\varepsilon}$ 
constructed as in \eqref{eq:initial-data}. Let $\mu^{\varepsilon}_t$ be the associated Radon measure on $\Omega$ defined as in \eqref{eq:def-mu} for each $\varepsilon\in(0,1),\ t\geq0$.

Then, there exists a subsequence of $\varepsilon\to0$ (still denoted by $\varepsilon\to0$) such that
    \begin{itemize}
        \item[(A)] There exists a family of Radon measures $(\mu_t)_{t\geq0}$ on $\Omega$ such that
        \begin{itemize}
            \item[(A.i)] $\mu_t^{\varepsilon}\to\mu_t$ as Radon measures on $\Omega$ for all $t\geq0$ as $\varepsilon\to0$. That is, $\mu_t^{\varepsilon}(\eta)\to\mu_t(\eta)$ for all $\eta\in C_c(\Omega)$ and $t\geq0$ as $\varepsilon\to0$.
            \item[(A.ii)]  $\mu^{\varepsilon}\to\mu$ as Radon measures on $\Omega\times[0,\infty)$, where $d\mu^{\varepsilon}:=d\mu^{\varepsilon}_tdt,\ d\mu:=d\mu_tdt$. That is, $\mu^{\varepsilon}(\eta)\to\mu(\eta)$ for all $\eta\in C_c(\Omega\times[0,\infty))$ as $\varepsilon\to0$.
            \item[(A.iii)] $\mu_t$ is $(d-1)$-integral for a.e. $t\geq0$.
            \item[(A.iv)] $\mu_s(\Omega)\leq\mu_t(\Omega)$ for $t,s\in[0,\infty)\setminus B$ with $t\leq s$, where $B$ is the countable set as in Proposition \ref{prop:convergence-measure}.
        \end{itemize}
        \item[(B)] There exists a characteristic function (say, with support $E(t)\subset\Omega$ at time $t\geq0$)
        \[
        \psi\in C^{0,1/2}\left([0,\infty);L^1(\Omega)\right)\cap BV_{loc}\left(\Omega\times[0,\infty)\right)
        \]
        such that
        \begin{itemize}
            \item[(B.i)] $E(0)=U_0$ a.e. in $\Omega$, and moreover, $\phi^{\varepsilon}\to2\psi-1$ strongly in $L^1_{loc}(\Omega\times[0,\infty))$ and a.e. pointwise.
            \item[(B.ii)] (Volume preservation) It holds that $\mathcal{L}^d(E(t))=\mathcal{L}^d(E(0))$ for all $t\geq0$.
            \item[(B.iii)] $\|\nabla\chi_{E(t)}\|(\Phi)\leq\mu_t(\Phi)$ for any $t>0$ and $\Phi\in C_c(\Omega;\mathbb{R}_{\geq0})$. Also, $\partial^{\ast}E(t)\subset\mathrm{supp}(\mu_t)$ for any $t>0$.
            \item[(B.iv)] (Nonoverlapping with obstacles) For all $t\geq0$, $\mathrm{supp}(\mu_t)\cap O=\emptyset$, $O_+\subset E(t)$, and $O_-\subset\Omega\setminus E(t)$ upto a set of Lebesgue measure zero.

        \end{itemize}
        \item[(C)] There exist $\lambda(t)\in L^2_{loc}(0,\infty)$ and $\varepsilon_1\in(0,1)$ such that
        \[
        \sup_{\varepsilon\in(0,\varepsilon_1)}\int_0^T|\lambda^{\varepsilon}(t)|^2dt<\infty\qquad\text{and}\qquad\lambda^{\varepsilon}\to\lambda\text{ weakly in }L^2(0,T)\quad\text{for any }T>0.
        \]
        \item[(D)] ($L^2$-flow with motion law on $\Omega\setminus\overline{O}$) There exists a measurable function $\theta:\partial^{\ast}E(t)\to\mathbb{N}$ such that $(\mu_t)_{t>0}$ is an $L^2$-flow on $\Omega\setminus\overline{O}$ with a generalized velocity vector \[
        v=h-\frac{\lambda}{\theta}\nu(\cdot,t)\quad\text{$\mathcal{H}^{d-1}$-a.e. on }\partial^{\ast}E(t)\setminus\overline{O}.
        \]
        Here, $h$ is the generalized mean curvature vector of $\mu_t$ and $\nu(\cdot,t)$ is the inner unit normal vector of $E(t)$ on $\partial^{\ast}E(t)$.
        
    \end{itemize}
\end{theorem}

\begin{remark}
Under the above setting, the change of volume is zero, that is,
\begin{align*}
\int_{\Omega}v\cdot\nu(\cdot,t)d\|\nabla\chi_{E(t)}\|=0\quad\text{for a.e. }t>0.
\end{align*}
Indeed, for any $\zeta\in C^1_c((0,T))$, which is seen as a test function in $C^1_c\left(\Omega\times(0,T)\right)$ spatially constant over $\Omega$, \cite[Proposition 4.5]{MR08} and (B.ii) yield that
\[
\int_0^T\zeta\int_{\Omega}v\cdot\nu d\|\nabla\chi_{E(t)}\|dt=\int_0^T\zeta_t\int_{\Omega}\chi_{E(t)}dxdt=\mathcal{L}^d(U_0)\int_0^T\zeta_tdt=0.
\]
\end{remark}

\begin{remark}\label{rmk:C2beta1}
We assumed that the obstacles are $C^{2,\beta}$-regular. Although the $C^{1,1}$-condition would suffice to employ the interior ball condition (see \eqref{eq:R0}) to choose the forcing term inside the obstacles with a reasonable magnitude, we impose further regularity in this paper. The main reason is in order to use the classical Schauder theory when we prove an $\varepsilon$-uniform $L^2$-estimate of the multiplier $\lambda^{\varepsilon}(t)$ in Proposition \ref{prop:L2-estimate}, and further technical reasons will be explained after the proof of Proposition \ref{prop:L2-estimate}.

Although we may think an approximation of a given rough obstacle by more regular obstacles as $\varepsilon\to0$, we do not pursue in this paper finding the optimal regularity condition on the obstacles with which we can obtain the result of Theorem \ref{thm:main-thm}. However, we mention that establishing the result with obstacles with rough boundaries is also an important problem, as such obstacles can naturally arise from physical phenomena.
\end{remark}

\section{Settings}\label{sec:settings}
This section is devoted to basic facts from the setting given in Subsection \ref{subsec:setup}.

\subsection{Obstacles and well-prepared initial data}\label{subsec:initial-data}

We remark that as $O$ is $C^{2,\beta}$, it satisfies the interior ball condition so that there exists $R_0>0$ such that
\begin{align}\label{eq:R0}
    O=\bigcup_{B_{R_0}(x)\subset O}B_{R_0}(x).
\end{align}

\medskip

It is known (see \cite{G84}) that from the assumption that $M_0$ is $C^1$, there exists a family $\{U_0^i\}_{i=1}^{\infty}$ of open sets with $C^3$ boundaries $M_0^i=\partial U_0^i$ for each $i=1,2,\cdots,$ such that
\begin{align*}
\lim_{i\to\infty}\mathcal{L}^d(U_0\Delta U_0^i)=0\qquad\text{and}\qquad\lim_{i\to\infty}\|\nabla\chi_{U_0^i}\|=\|\nabla\chi_{U_0}\|\text{ as measures.}
\end{align*}
For a sequence $\{\varepsilon_i\}_{i=1}^{\infty}$ that converges to 0 as $i\to\infty$, we let 
\begin{align*}
r_0^{i}(x):=\begin{cases}
\mathrm{dist}(x,M_0^i) \qquad\qquad &\text{for}\quad x\in U_0^i,\\
-\mathrm{dist}(x,M_0^i) \qquad\qquad &\text{for}\quad x\notin U_0^i,
\end{cases}
\end{align*}
and $\widetilde{r}_0^{i}\in C^3(\Omega)$ be a smoothing of $r_0^{i}(x)$ such that, by taking a subsequence if necessary,
\begin{align}\label{eq:derivative-r}
\sup_{i\geq1}\|\nabla^j\widetilde{r}_0^{i}\|_{L^{\infty}(\Omega)}\leq\varepsilon_i^{-j+1}\qquad\text{for }j=1,2,3,
\end{align}
and $\widetilde{r}_0^{i}=r_0^{i}$ near $M_0^i$. With the functions $q^{\varepsilon}(r):=\tanh\left({\frac{r}{\varepsilon}}\right)$ for $r\in\mathbb{R},\ \varepsilon\in(0,1)$, we set
\begin{align}\label{eq:initial-data}
\phi_0^{\varepsilon_i}(x):=q^{\varepsilon_i}\left(\widetilde{r}_0^{i}(x)\right)\qquad\text{for }x\in\Omega.    
\end{align}

Also, as we are assuming that $\overline{O_-}\subset\Omega\setminus\overline{U_0}$, we see that $U_0$ does not occupy the full region of $\Omega\setminus\overline{O_-}$ (so that the meaningful motion happens). Consequently, there exists a constant $\omega\in(0,1)$ such that $\frac{2}{3}-\left|\fint_{\Omega\setminus O}k(\chi_{U_0})dx\right|\geq3\omega>0$ and such that
\begin{align}\label{eq:initial-volume-lower-bound}
\frac{2}{3}-\left|\fint_{\Omega\setminus O}k(\phi^{\varepsilon_i}_0)dx\right|\geq2\omega>0\qquad\text{for all }i\geq1,
\end{align}
where $\fint_{\Omega\setminus O}f(x) dx$ is the average of the integration over $\Omega\setminus O$ of a given integrand $f(x)$.

\medskip

Throughout the rest of the paper, we write $\varepsilon\in(0,1)$ instead of $\{\varepsilon_i\}_{i=1}^{\infty}$ with the full indexing by abuse of notations. We will denote by $C>0$ positive constants that may vary line by line. Dependency of numbers (such as $C>0$, $\varepsilon\in(0,1)$) on the quantities given by the problem (such as $M_0$, $\{U_0^i\}_{i=1}^{\infty}$, $O$, the radius $R_0>0$, the spatial dimension $d$, and $\alpha,\beta,\omega\in(0,1)$) will be omitted unless necessary.

\subsection{Energy dissipation}\label{subsec:energy-dissipation}

It is well known that the well-prepared initial data \eqref{eq:initial-data} have $\varepsilon$-uniform initial energy, that is,
\begin{align}\label{eq:well-preparedness}
\sup_{\varepsilon\in(0,1)}\mu_0^{\varepsilon}(\chi_{\Omega})\leq C\qquad\text{for some constant }C>0.    
\end{align}
Utilizing this fact, we obtain the following.

\begin{proposition}\label{prop:energy-dissipation}
Let $\phi^{\varepsilon}$ be the solution to \eqref{eq:phi-epsilon} with the initial data \eqref{eq:initial-data}, where $g^{\varepsilon},\lambda^{\varepsilon}$ are given by \eqref{eq:forcing}, \eqref{eq:lambda}, respectively. Let $E^{\varepsilon}(t)$ be the energy given by \eqref{eq:energy} and $\mu_t^{\varepsilon}$ the Radon measure given by \eqref{eq:def-mu}.

Then, there exists a constant $C>0$ such that for $\varepsilon\in(0,1)$, $t\geq0$,
\begin{align}\label{eq:energy-dissipation-equality}
E^{\varepsilon}(t)+\int_0^t\int_{\Omega}\varepsilon(\phi_t^{\varepsilon})^2dxdt=E^{\varepsilon}(0)+\int_{\{x:s(x)>\frac{1}{2}\sqrt{\varepsilon}\}}g^{\varepsilon}\cdot\left(k(\phi^{\varepsilon}(t))-k(\phi^{\varepsilon}(0)\right)dx,
\end{align}
and
\begin{align}\label{eq:estimate-from-energy-dissipation}
\mu_t^{\varepsilon}(\chi_{\Omega})+\int_0^t\int_{\Omega}\varepsilon(\phi^{\varepsilon}_t)^2dxdt\leq C.    
\end{align}
\end{proposition}
\begin{proof}
We differentiate to obtain
\begin{align*}
&\frac{d}{dt}\int_{\Omega}\left(\frac{\varepsilon|\nabla\phi^{\varepsilon}|^2}{2}+\frac{W(\phi^{\varepsilon})}{\varepsilon}\right)dx=\int_{\Omega}\varepsilon\nabla\phi^{\varepsilon}\cdot\nabla\phi^{\varepsilon}_t+\varepsilon^{-1}W'(\phi^{\varepsilon})\phi^{\varepsilon}_t dx\\
&=\int_{\Omega}\left(-\varepsilon\Delta\phi^{\varepsilon}+\varepsilon^{-1}W'(\phi^{\varepsilon})\right)\phi^{\varepsilon}_t dx=\int_{\Omega}\left(-\varepsilon\phi^{\varepsilon}_t+g^{\varepsilon}\sqrt{2W(\phi^{\varepsilon})}\right)\phi^{\varepsilon}_tdx.
\end{align*}
Thus,
\begin{align*}
\frac{d}{dt}E^{\varepsilon}(t)&=-\int_{\Omega}\varepsilon(\phi^{\varepsilon}_t)^2dx+\int_{\Omega}g^{\varepsilon}\sqrt{2W(\phi^{\varepsilon})}\phi^{\varepsilon}_tdx+\lambda^{\varepsilon}(t)\cdot\left(-\int_{\left\{x:s(x)\leq\frac{1}{2}\sqrt{\varepsilon}\right\}}\eta_{\varepsilon}(s(x))\left(k(\phi^{\varepsilon})\right)_tdx\right)\\
&=-\int_{\Omega}\varepsilon(\phi^{\varepsilon}_t)^2dx+\int_{\left\{x:s(x)>\frac{1}{2}\sqrt{\varepsilon}\right\}}g^{\varepsilon}\left(k(\phi^{\varepsilon})\right)_tdx.
\end{align*}
By integrating in time from $0$ to $t$, we obtain \eqref{eq:energy-dissipation-equality}. We note that $g^{\varepsilon}$ is independent of time and bounded by $\frac{d}{R_0}$ in the set $\left\{x:s(x)>\frac{1}{2}\sqrt{\varepsilon}\right\}$. Inequality \eqref{eq:estimate-from-energy-dissipation} follows from \eqref{eq:well-preparedness}, \eqref{eq:energy-dissipation-equality}, and the fact that $\left|\phi^{\varepsilon}\right|$ is bounded by 1 (see Proposition \ref{prop:strong-max-principle}).
\end{proof}

From \eqref{eq:energy-dissipation-equality}, we also have the following properties that will be crucially used later.

\begin{proposition}\label{prop:VP-and-multiplier}
Under the same setting as in Proposition \ref{prop:energy-dissipation}, there exists a constant $C>0$ such that for $\varepsilon\in(0,1)$, $t\geq0$,
\begin{align}\label{eq:VP-from-energy-dissipation}
\frac{1}{\varepsilon^{\alpha/2}}\left|\int_{\left\{x\in\Omega:s(x)\leq\frac{1}{2}\sqrt{\varepsilon}\right\}}\eta_{\varepsilon}(s(x))\left(k(\phi_0^{\varepsilon})-k(\phi^{\varepsilon})\right)dx\right|\leq C,
\end{align}
or equivalently,
\begin{align}\label{eq:estimate-multiplier}
|\lambda^{\varepsilon}(t)|\leq C\varepsilon^{-\alpha/2}.
\end{align}
\end{proposition}
\begin{proof}
Inequality \eqref{eq:VP-from-energy-dissipation} follows from \eqref{eq:well-preparedness}, \eqref{eq:energy-dissipation-equality}, and Proposition \ref{prop:strong-max-principle}.
\end{proof}

We also record the following that will be used in the next section.

\begin{proposition}\label{prop:strong-max-principle}
Under the same setting as in Proposition \ref{prop:energy-dissipation}, it holds that
$$\|\phi^{\varepsilon}(\cdot,t)\|_{L^{\infty}(\Omega)}<1\qquad\text{for }\varepsilon\in(0,1),\ t\geq0.$$
\end{proposition}
The proposition can be proved by the strong maximum principle (see \cite[Proposition 3.3]{T21}) and we omit the proof here.

\section{Barrier functions for obstacles}\label{sec:barrier-functions}

We construct barrier functions for obstacles in this section. We let $\underline{r}_{y}(x):=\frac{1}{2R_0}\left(R_0^2-|x-y|^2\right)$, $\overline{r}_y:=-\underline{r}_y$, $\underline{\phi}^{\varepsilon}_y(x):=q^{\varepsilon}(\underline{r}_y(x))$, $\overline{\phi}^{\varepsilon}_y(x):=q^{\varepsilon}(\overline{r}_y(x))$. We recall that $q^{\varepsilon}(r):=\tanh{\left(\frac{r}{\varepsilon}\right)}$ for $r\in\mathbb{R}$. We consider the periodical extensions of $\Omega$, $O$, and $\phi^{\varepsilon}$ to $\mathbb{R}^d$.

\begin{lemma}\label{lem:barrier-functions}
There exists $\varepsilon_0\in(0,1)$ satisfying the following:
\begin{itemize}
\item[(i)] If $B_{R_0+2\sqrt{\varepsilon}}(y)\subset O_+$, then $\underline{\phi}^{\varepsilon}_y(x)$ is a subsolution to the first equation of \eqref{eq:phi-epsilon} with $\mathbb{R}^d$ in place of $\Omega$ for any $\varepsilon\in(0,\varepsilon_0)$.
\item[(ii)] If $B_{R_0+2\sqrt{\varepsilon}}(z)\subset O_-$, then $\overline{\phi}^{\varepsilon}_z(x)$ is a supersolution to the first equation of \eqref{eq:phi-epsilon} with $\mathbb{R}^d$ in place of $\Omega$ for any $\varepsilon\in(0,\varepsilon_0)$.
\end{itemize}


\end{lemma}
\begin{proof}
We prove only the first statement as the second statement can be shown similarly.

Without loss of generality, we may assume $y=0$. Due to Proposition \ref{prop:strong-max-principle}, the function $r^{\varepsilon}:=(q^{\varepsilon})^{-1 }(\phi^{\varepsilon})$ is defined.

By computation from the first equation of \eqref{eq:phi-epsilon} and the definition $r^{\varepsilon}=(q^{\varepsilon})^{-1 }(\phi^{\varepsilon})$, we see that $r^{\varepsilon}$ solves
\begin{align*}
r^{\varepsilon}_t-\Delta r^{\varepsilon}+\frac{2q^{\varepsilon}(r^{\varepsilon})}{\varepsilon}\left(|\nabla r^{\varepsilon}|^2-1\right)-g^{\varepsilon}=0\qquad\text{in }\Omega\times(0,\infty).
\end{align*}
We substitute $\underline{r}_0$ to get
\begin{align*}
(\underline{r}_0)^{\varepsilon}_t-\Delta \underline{r}_0^{\varepsilon}+\frac{2q^{\varepsilon}(\underline{r}_0^{\varepsilon})}{\varepsilon}\left(|\nabla \underline{r}_0^{\varepsilon}|^2-1\right)-g^{\varepsilon}=\frac{d}{R_0}+\frac{2q^{\varepsilon}(\underline{r}_0^{\varepsilon})}{\varepsilon}\left(\frac{|x|^2}{R_0^2}-1\right)-g^{\varepsilon}.
\end{align*}
Our goal is to show that the right-hand side is nonpositive.

On the region $\{x:|x|\leq R_0\}$, we have $q^{\varepsilon}(\underline{r}_0(x))\geq0$ and $\frac{|x|^2}{R_0^2}-1\leq0$, and thus,
\begin{align*}
\frac{d}{R_0}+\frac{2q^{\varepsilon}(\underline{r}_0^{\varepsilon})}{\varepsilon}\left(\frac{|x|^2}{R_0^2}-1\right)-g^{\varepsilon}\leq\frac{d}{R_0}-\frac{d}{R_0}=0.
\end{align*}

On the region $\{x:R_0\leq |x|\leq R_0+\sqrt{\varepsilon}\}$, we have $q^{\varepsilon}(\underline{r}_0(x))\leq0$ and $\frac{|x|^2}{R_0^2}-1\geq0$, and thus,
\begin{align*}
\frac{d}{R_0}+\frac{2q^{\varepsilon}(\underline{r}_0^{\varepsilon})}{\varepsilon}\left(\frac{|x|^2}{R_0^2}-1\right)-g^{\varepsilon}\leq\frac{d}{R_0}-\frac{d}{R_0}=0.
\end{align*}

On the region $\{x:|x|\geq R_0+\sqrt{\varepsilon}\}$, it holds that
\begin{align*}
\frac{2q^{\varepsilon}(\underline{r}_0^{\varepsilon})}{\varepsilon}\left(\frac{|x|^2}{R_0^2}-1\right)\leq\frac{2}{\varepsilon}\tanh{\left(-\frac{2R_0\sqrt{\varepsilon}+\varepsilon}{2R_0\varepsilon}\right)}\frac{2R_0\sqrt{\varepsilon}+\varepsilon}{R_0^2}\leq-\frac{4}{R_0\sqrt{\varepsilon}}\tanh{\left(\frac{1}{\sqrt{\varepsilon}}\right)},
\end{align*}
and therefore,
\begin{align*}
\frac{d}{R_0}+\frac{2q^{\varepsilon}(\underline{r}_0^{\varepsilon})}{\varepsilon}\left(\frac{|x|^2}{R_0^2}-1\right)-g^{\varepsilon}\leq\frac{d}{R_0}-\frac{4}{R_0\sqrt{\varepsilon}}\tanh{\left(\frac{1}{\sqrt{\varepsilon}}\right)}+\max\left\{|\lambda^{\varepsilon}(t)|,\frac{d}{R_0}\right\}.
\end{align*}
By \eqref{eq:estimate-multiplier}, we have
\begin{align*}
\frac{d}{R_0}+\frac{2q^{\varepsilon}(\underline{r}_0^{\varepsilon})}{\varepsilon}\left(\frac{|x|^2}{R_0^2}-1\right)-g^{\varepsilon}\leq\frac{d}{R_0}+C\varepsilon^{-\alpha/2}-\frac{4}{R_0\sqrt{\varepsilon}}\tanh{\left(\frac{1}{\sqrt{\varepsilon}}\right)}
\end{align*}
for some constant $C>0$, and if $\varepsilon\in(0,1)$ is sufficiently small, we see that the right-hand side is nonpositive, and we complete the proof.
\end{proof}

We state the following proposition but omit the proof due to similarity to \cite[Lemma 4.3, Proposition 4.4]{T21}.

\begin{proposition}\label{prop:barrier-functions}
Suppose that $B_{R_0+2\sqrt{\varepsilon}}(y)\subset O_+$ and $B_{R_0+2\sqrt{\varepsilon}}(z)\subset O_-$. Then, for sufficiently small $\varepsilon\in(0,1)$, it holds that
\begin{align*}
\underline{\phi}^{\varepsilon}_y(x)\leq\phi^{\varepsilon}(x,t)\leq\overline{\phi}^{\varepsilon}_z(x)\qquad\text{for all }(x,t)\in\mathbb{R}^d\times[0,\infty).
\end{align*}
\end{proposition}

\section{A $L^2$-estimate of the forcing}\label{sec:L2-estimate}
Next, we prove the following estimate on the $L^2$-regularity of the forcing term $g^{\varepsilon}$. The control of integrations on the set $\Omega\setminus O$ where the surfaces are allowed to flow is inspired by \cite{T23,BS97} and the references therein for the related part. The main focus is on the simultaneous treatment of the integrations near the boundary $\partial O$.

\begin{proposition}\label{prop:L2-estimate}
There exists a number $\varepsilon_1\in(0,1)$ such that for any $T>0$, there exists a constant $C_T>0$ depending on $T$ such that
\begin{align*}
\sup_{\varepsilon\in(0,\varepsilon_1)}\int_0^T|\lambda^{\varepsilon}(t)|^2dt\leq C_T.
\end{align*}
\end{proposition}
\begin{proof}
Let $\zeta$ be a test vector field on $\Omega$ (or a periodic test vector field on $\mathbb{R}^d$). We multiply the first equation of \eqref{eq:phi-epsilon} by $\nabla\phi^{\varepsilon}\cdot\zeta$ and perform integration by parts to obtain
\begin{align}
\int_{\Omega}\varepsilon\phi^{\varepsilon}_t\nabla\phi^{\varepsilon}\cdot\zeta dx+\int_{\Omega}\varepsilon\nabla\phi^{\varepsilon}&\otimes\nabla\phi^{\varepsilon}:\nabla\zeta dx-\int_{\Omega}\left(\frac{\varepsilon|\nabla\phi^{\varepsilon}|^2}{2}+\frac{W(\phi^{\varepsilon})}{\varepsilon}\right)\mathrm{div}(\zeta)dx\notag\\
&=-\int_{\Omega}g^{\varepsilon}k(\phi^{\varepsilon})\mathrm{div}(\zeta)dx-\int_{\Omega}k(\phi^{\varepsilon})\nabla g^{\varepsilon}\cdot\zeta dx.\label{eq:integration-by-parts}
\end{align}

\medskip

Let $\delta>0$ be a positive number to be chosen later and $\Psi_{\delta}\in C_c^{\infty}(B_{\delta}(0))$ a standard mollifier. We consider the solution $u$ to the following equation (for a fixed $t\geq0$):
\begin{align}\label{eq:u}
\begin{cases}
-\Delta u=k(\phi^{\varepsilon}(t))\ast\Psi_{\delta}-\fint_{\Omega\setminus O}k(\phi^{\varepsilon}(t))\ast\Psi_{\delta}dx \quad &\text{in}\quad\Omega\setminus\overline{O},\\
\nabla u\cdot n=0 \quad &\text{on}\quad\partial(\Omega\setminus O),\\
\int_{\Omega\setminus\overline{O}}u(x) dx=0. \quad &
\end{cases}
\end{align}
The solution $u$ satisfies a classical Schauder estimate upto the boundary (see \cite{N14} for instance) so that
\begin{align}\label{eq:schauder-estimate}
\|u\|_{C^{2,\beta}(\overline{\Omega\setminus O})}\leq C(1+\delta^{-1})
\end{align}
for some constant $C>0$. Here, $n$ denotes the unit normal vector on $\partial(\Omega\setminus O)$ pointing inward $O$, and we have used Proposition \ref{prop:strong-max-principle}.

By \cite{W34} (see also \cite[Chapter 1]{M66}), there exists a function $\widetilde{u}\in C^2(\Omega)$ (or a periodic function on $\mathbb{R}^d$) such that $\widetilde{u}\lfloor_{\Omega\setminus O}=u$ and that
\begin{align}\label{eq:C2-u-tilde}
\|\widetilde{u}\|_{C^2(\Omega)}\leq C_d\|u\|_{C^2(\overline{\Omega\setminus O})}.
\end{align}
On the whole domain $\Omega$, we take $\zeta:=\nabla\widetilde{u}$.

\medskip

We estimate the terms of \eqref{eq:integration-by-parts} by beginning with the left-hand side. Due to \eqref{eq:schauder-estimate}, \eqref{eq:C2-u-tilde}, and \eqref{eq:estimate-from-energy-dissipation}, the terms of \eqref{eq:integration-by-parts} on the left-hand side satisfy
\begin{align}
\left|\int_{\Omega}\varepsilon\phi^{\varepsilon}_t\nabla\phi^{\varepsilon}\cdot\zeta dx\right|&\leq C_{\delta}\left(\int_{\Omega}\varepsilon(\phi^{\varepsilon}_t)^2dx\right)^{1/2}\left(\int_{\Omega}\varepsilon|\nabla\phi^{\varepsilon}|^2dx\right)^{1/2}\notag\\
&\leq C_{\delta}\left(\int_{\Omega}\varepsilon(\phi^{\varepsilon}_t)^2dx\right)^{1/2}\label{eq:LHS-1}
\end{align}
and
\begin{align}\label{eq:LHS-2}
\left|\int_{\Omega}\varepsilon\nabla\phi^{\varepsilon}\otimes\nabla\phi^{\varepsilon}:\nabla\zeta dx\right|\leq C_{\delta}\int_{\Omega}\varepsilon|\nabla\phi^{\varepsilon}|^2dx\leq C_{\delta}.
\end{align}
Lastly,
\begin{align}\label{eq:LHS-3}
\left|\int_{\Omega}\left(\frac{\varepsilon|\nabla\phi^{\varepsilon}|^2}{2}+\frac{W(\phi^{\varepsilon})}{\varepsilon}\right)\mathrm{div}(\zeta)dx\right|\leq C_{\delta}\mu_t^{\varepsilon}(\chi_{\Omega})\leq C_{\delta}.
\end{align}

We estimate the term $\int_{\Omega}k(\phi^{\varepsilon})\nabla g^{\varepsilon}\cdot\zeta dx$. Let $\varepsilon\in(0,1)$ be small enough that $\nabla s$ is defined on the region $\{x:0\leq s(x)\leq \sqrt{\varepsilon}\}$, which serves as a local extension of the unit normal vector $n$ on $\partial(\Omega\setminus O)$. We first note that as $\nabla s\cdot\nabla\widetilde{u}$ vanishes on $\partial(\Omega\setminus O)$, we have
\begin{align*}
|\nabla s\cdot\nabla\widetilde{u}|\leq\|\widetilde{u}\|_{C^2(\Omega)}\sqrt{\varepsilon}
\end{align*}
on the region $\{x:0\leq s(x)\leq \sqrt{\varepsilon}\}$. Utilizing this fact together with \eqref{eq:schauder-estimate}, \eqref{eq:C2-u-tilde}, \eqref{eq:estimate-multiplier} Proposition \ref{prop:strong-max-principle}, we have
\begin{align}
\left|\int_{\Omega}k(\phi^{\varepsilon})\nabla g^{\varepsilon}\cdot\zeta dx\right|\notag&\leq\int_{\{x:0\leq s(x)\leq \sqrt{\varepsilon}\}}|k(\phi^{\varepsilon})|\cdot|\nabla g^{\varepsilon}\cdot\nabla\widetilde{u}|dx\notag\\
&\leq\int_{\{x:0\leq s(x)\leq \sqrt{\varepsilon}\}}C\varepsilon^{-1/2}\max\left\{|\lambda^{\varepsilon}(t)|,\frac{d}{R_0}\right\}|\nabla s\cdot\nabla\widetilde{u}|dx\notag\\
&\leq C_{d,\delta}\sqrt{\varepsilon}\max\left\{|\lambda^{\varepsilon}(t)|,\frac{d}{R_0}\right\}\notag\\
&\leq C_{d,\delta}\varepsilon^{\frac{1-\alpha}{2}}\qquad\text{for }\varepsilon\in(0,1)\text{ small enough}.\label{eq:RHS-1}
\end{align}
In transition from the third inequality to the fourth, we used the fact that $\mathcal{L}^d\left(\{x:0\leq s(x)\leq\sqrt{\varepsilon}\}\right)\leq C\sqrt{\varepsilon}$ for some constant $C>0$ that depends on the given obstacle $O$.

We estimate the term $\int_{\Omega}g^{\varepsilon}k(\phi^{\varepsilon})\mathrm{div}(\zeta)dx$. We note that $\mathrm{div}(\zeta)=\Delta \widetilde{u}$ satisfies $\|\mathrm{div}(\zeta)\|_{L^{\infty}(\Omega)}\leq C_{d,\delta}$ thanks to \eqref{eq:schauder-estimate}, \eqref{eq:C2-u-tilde}, and that $|k(\phi^{\varepsilon})|$ is bounded by an absolute number. By the definition of the forcing $g^{\varepsilon}$, the integration on the set $\{x:s(x)\geq0\}$ satisfies
\begin{align}
\left|\int_{\{x:s(x)\geq0\}}g^{\varepsilon}k(\phi^{\varepsilon})\mathrm{div}(\zeta)dx\right|&\leq\left|\int_{\text{$\{x:0\leq s(x)\leq\sqrt{\varepsilon}\}$}}g^{\varepsilon}k(\phi^{\varepsilon})\mathrm{div}(\zeta)dx\right|+\left|\int_{\{x:s(x)>\sqrt{\varepsilon}\}}g^{\varepsilon}k(\phi^{\varepsilon})\mathrm{div}(\zeta)dx\right|\notag\\
&\leq C_{d,\delta}\sqrt{\varepsilon}\max\left\{|\lambda^{\varepsilon}(t)|,\frac{d}{R_0}\right\}+C_{d,\delta,R_0}\notag\\
&\leq C_{d,\delta,R_0}\quad\text{for $\varepsilon\in(0,1)$ small enough}.\label{eq:RHS-2-1}
\end{align}
Also, the integration on the other set $\{x:s(x)<0\}$ satisfies
\begin{align*}
-\int_{\{x:s(x)<0\}}g^{\varepsilon}k(\phi^{\varepsilon})\mathrm{div}(\zeta)dx=\lambda^{\varepsilon}(t)\int_{\Omega\setminus\overline{O}}k(\phi^{\varepsilon})(-\Delta u)dx
\end{align*}

\medskip

Now, we aim to prove
\begin{align}\label{eq:lower-bound}
\int_{\Omega\setminus\overline{O}}k(\phi^{\varepsilon})(-\Delta u)dx\geq\frac{1}{2}\omega^2\mathcal{L}^d(\Omega\setminus O),
\end{align}
where $\omega\in(0,1)$ is from \eqref{eq:initial-volume-lower-bound} and $\mathcal{L}^d(\Omega\setminus O)$ is the $d$-dimensional Lebesgue measure of $\Omega\setminus O$. If this can be verified, then the above estimates (\eqref{eq:LHS-1}, \eqref{eq:LHS-2}, \eqref{eq:LHS-3}, \eqref{eq:RHS-1}, \eqref{eq:RHS-2-1}) yield
\begin{align*}
\int_0^T|\lambda^{\varepsilon}(t)|^2dt\leq \frac{C_{d,\delta,R_0}}{\omega^2\mathcal{L}^d(\Omega\setminus O)}\left(\left(\int_0^T\int_{\Omega}\varepsilon(\phi^{\varepsilon}_t)^2dxdt\right)+1+T\right),
\end{align*}
which completes the proof (with a choice of $\delta\in(0,1)$ to be taken soon that depends only on the quantities given by the problem such as $d,R_0,\omega,\mathcal{L}^d(\Omega\setminus O)$).

\medskip

First of all, from \eqref{eq:VP-from-energy-dissipation}, \eqref{eq:initial-volume-lower-bound}, and Proposition \ref{prop:strong-max-principle}, we can see that there exists $\varepsilon_1\in(0,1)$ such that
\begin{align}\label{eq:volume-lower-bound}
\frac{2}{3}-\left|\fint_{\Omega\setminus O}k(\phi^{\varepsilon})dx\right|\geq\omega>0\qquad\text{for all }\varepsilon\in(0,\varepsilon_1).
\end{align}
We write
\begin{align*}
&\int_{\Omega\setminus\overline{O}}k(\phi^{\varepsilon})(-\Delta u)dx\\
&=\int_{\Omega\setminus\overline{O}}k(\phi^{\varepsilon})\left(k(\phi^{\varepsilon})\ast\Psi_{\delta}-\fint_{\Omega\setminus O}k(\phi^{\varepsilon})\ast\Psi_{\delta}dx\right)dx\\
&=\frac{4}{9}\mathcal{L}^d(\Omega\setminus O)+\int_{\Omega\setminus O}k(\phi^{\varepsilon})^2-\frac{4}{9}dx\\
&\qquad\qquad+\int_{\Omega\setminus O}k(\phi^{\varepsilon})\left(k(\phi^{\varepsilon})\ast\Psi_{\delta}-k(\phi^{\varepsilon})\right)dx-\frac{1}{\mathcal{L}^d(\Omega\setminus O)}\left(\int_{\Omega\setminus O}k(\phi^{\varepsilon})dx\right)^2\\
&\qquad\qquad\qquad\qquad+\frac{1}{\mathcal{L}^d(\Omega\setminus O)}\int_{\Omega\setminus O}k(\phi^{\varepsilon})dx\cdot\left(\int_{\Omega\setminus O}k(\phi^{\varepsilon})dx-\int_{\Omega\setminus O}k(\phi^{\varepsilon})\ast\Psi_{\delta}dx\right).
\end{align*}
From the fact that $k(s)^2-\frac{4}{9}\geq-W(s)$ for $s\in[-1,1]$, we have from \eqref{eq:estimate-from-energy-dissipation}, Proposition \ref{prop:strong-max-principle} that
\begin{align}\label{eq:RHS-3-1}
\int_{\Omega\setminus O}k(\phi^{\varepsilon})^2-\frac{4}{9}dx\geq-C\varepsilon.
\end{align}
Also, since
\begin{align}\label{eq:modica-mortola}
\int_{\Omega}|\nabla(k(\phi^{\varepsilon}))|dx=\int_{\Omega}|\sqrt{2W(\phi^{\varepsilon})}||\nabla\phi^{\varepsilon}|dx\leq\int_{\Omega}\frac{\varepsilon|\nabla\phi^{\varepsilon}|^2}{2}+\frac{W(\phi^{\varepsilon})}{\varepsilon}dx\leq C,
\end{align}
we have
\begin{align}\label{eq:RHS_3_2}
\left|\int_{\Omega\setminus O}k(\phi^{\varepsilon})\left(k(\phi^{\varepsilon})\ast\Psi_{\delta}-k(\phi^{\varepsilon})\right)dx\right|\leq C\delta
\end{align}
and
\begin{align}\label{eq:RHS_3_3}
\left|\frac{1}{\mathcal{L}^d(\Omega\setminus O)}\int_{\Omega\setminus O}k(\phi^{\varepsilon})dx\cdot\left(\int_{\Omega\setminus O}k(\phi^{\varepsilon})dx-\int_{\Omega\setminus O}k(\phi^{\varepsilon})\ast\Psi_{\delta}dx\right)\right|\leq C\delta
\end{align}
by Lemma \ref{lem:auxiliary-lemma} (stated after the proof). Therefore, by \eqref{eq:volume-lower-bound} and the above estimates \eqref{eq:RHS-3-1}, \eqref{eq:RHS_3_2}, \eqref{eq:RHS_3_3}, we obtain
\begin{align*}
&\int_{\Omega\setminus\overline{O}}k(\phi^{\varepsilon})(-\Delta u)dx\\
&\geq\mathcal{L}^d(\Omega\setminus O)\left(\frac{4}{9}-\left(\fint_{\Omega\setminus O}k(\phi^{\varepsilon})dx\right)^2\right)-C(\varepsilon+\delta)\\
&\geq\mathcal{L}^d(\Omega\setminus O)\left(\frac{2}{3}-\left(\fint_{\Omega\setminus O}k(\phi^{\varepsilon})dx\right)\right)\left(\frac{2}{3}+\left(\fint_{\Omega\setminus O}k(\phi^{\varepsilon})dx\right)\right)-C(\varepsilon+\delta)\\
&\geq\omega^2\mathcal{L}^d(\Omega\setminus O)-C(\varepsilon+\delta)\\
&\geq\frac{1}{2}\omega^2\mathcal{L}^d(\Omega\setminus O)\qquad\text{for }\delta:=\frac{\omega^2\mathcal{L}^d(\Omega\setminus O)}{4C}\text{ and }\varepsilon\in(0,\delta).
\end{align*}
We finish the proof.
\end{proof}

We state the lemma (see \cite[Lemma 3.24]{AFP00}) used in the last part of the above proof. We keep the notation $\Psi_{\delta}\in C_c^{\infty}(B_{\delta}(0))$ for a standard mollifier.
\begin{lemma}\label{lem:auxiliary-lemma}
Let $U$ be a nonempty open subset of $\mathbb{R}^d$, and let $u\in BV(U),\ K\subset\joinrel\subset U$. For any $\delta\in(0,\mathrm{dist}(K,\partial U))$, it holds that
\begin{align*}
\int_K\left|u\ast\Psi_{\delta}-u\right|dx\leq\delta\cdot\|\nabla u\|(U).
\end{align*}
\end{lemma}
\begin{remark}\label{rmk:C2beta2}
We note that the $C^{2,\beta}$-regularity assumption on the obstacles is used in \eqref{eq:schauder-estimate}, \eqref{eq:C2-u-tilde}. To control the $C^1$-norm of the vector field $\zeta=\nabla\widetilde{u}$, we needed to keep the $C^2$-norms of the function $u$ and its extension $\widetilde{u}$ under the control.
\end{remark}

\section{Discrepancy and density ratio estimates}\label{sec:discrepancy-density-ratio-estimates}
For $\varepsilon\in(0,1),\ t\geq0$, define the Radon measure $\xi_t^{\varepsilon}$ by
\begin{align}\label{eq:def-xi}
    \xi_t^{\varepsilon}(\eta):=\frac{1}{\sigma}\int_{\Omega}\left(\frac{\varepsilon|\nabla\phi^{\varepsilon}|^2}{2}-\frac{W(\phi^{\varepsilon})}{\varepsilon}\right)\eta\, dx\qquad\text{for all }\eta\in C(\Omega),
\end{align}
and we let $d\xi^{\varepsilon}:=d\xi_t^{\varepsilon}dt$ on $\Omega\times[0,\infty)$. We call $d\xi_t^{\varepsilon}$ and $d\xi^{\varepsilon}$ the discrepancy measures. It is well known that, as well as \eqref{eq:well-preparedness}, the well-prepared initial data \eqref{eq:initial-data} have nonpositive discrepancy measure, that is, for small $\varepsilon>0$,
\begin{align}\label{eq:xi-0-nonpositive}
\xi_0^{\varepsilon}(\chi_{\Omega})\leq 0.    
\end{align}

The main goal is to prove that $d\xi^{\varepsilon}\to0$ as $\varepsilon\to0$ (Theorem \ref{thm:vanishing-xi}). This is crucial to verify the ``surface-like" behavior of our energy measures $d\mu^{\varepsilon}$ as $\varepsilon\to0$. The motivation of this framework is from \cite{I93}, which is further developed by \cite{T23} with the Lagrange multiplier for the volume constraint.

The main challenge when following the framework of \cite{T23} arises from the fact that the forcing term is no longer constant in the spatial variable due to the presence of obstacles, which fails the availability of \cite[Proposition 6]{T23}, namely the nonpositivity of $\xi^{\varepsilon}$. This then results in another term in the monotonicity formula (Proposition \ref{prop:monotonicity-formula}) from $\xi^{\varepsilon}$ which we have to control.

We resolve this technical issue by considering an upper bound of $\xi^{\varepsilon}$ (Proposition \ref{prop:upper-bound-discrepancy}), deriving an upper bound of the density ratio (Proposition \ref{prop:upper-bound-density-ratio}) as motivated from \cite{NT25} which considers a spatial forcing term. We then achieve to prove the vanishing of the discrepancy measure (Theorem \ref{thm:vanishing-xi}).


Keeping this viewpoint, we proceed to follow the framework in \cite[Section 4]{T23}. Necessary adaptations, in particular the ones that appear before proving Theorem \ref{thm:vanishing-xi}, are left for the sake of completeness.

\medskip

Before we begin, we record the following property of our forcing term in the following.

\begin{lemma}\label{lem:L2-g-1-varepilson}
There exists a constant $C>0$ such that for each $\varepsilon\in(0,1)$, we have
\begin{align}\label{eq:L-infty-g-ep}
\|g^{\varepsilon}\|_{L^{\infty}(\Omega\times[0,\infty))}\leq C\varepsilon^{-\alpha/2}.
\end{align}
Moreover, the number $\varepsilon_1\in(0,1)$ given by Proposition \ref{prop:L2-estimate} satisfies that for any $T>0$, there is $C_T>0$ such that
\begin{align}\label{eq:L-2-g-1-ep}
\sup_{\varepsilon\in(0,\varepsilon_1)}\int_0^T\int_{\Omega}\frac{1}{\varepsilon}\left|\max\{\lambda^{\varepsilon}(t),1\}\sqrt{2W(\phi^{\varepsilon}(x,t))}\right|^2\,dxdt\leq C_T.
\end{align}
\end{lemma}
\begin{proof}
The proof of \eqref{eq:L-infty-g-ep} follows from \eqref{eq:estimate-multiplier} and Proposition \ref{prop:strong-max-principle}.
The proof of \eqref{eq:L-2-g-1-ep} follows from the construction of $g^{\varepsilon}$, \eqref{eq:estimate-from-energy-dissipation}, and Proposition \ref{prop:L2-estimate}. 
\end{proof}

\subsection{The monotonicity formula and an upper bound of the discrepancy}\label{subsec:monotonicity-formula-upper-bound-xi}
The monotonicity formula is established in \cite{H90} for the mean curvature flow and in \cite{I93} for the Allen-Cahn equation in $\mathbb{R}^d$. We denote the backward heat kernel $\rho = \rho_{(y,s)}(x, t)$ by
$$
\rho_{(y,s)}(x, t) = \frac{1}{(4\pi (s - t))^{\frac{d-1}{2}}} e^{-\frac{|x-y|^2}{4(s-t)}}, \quad x, y \in \mathbb{R}^d, \; 0 \leq t < s.
$$

We state the monotonicity formula in the proposition, although we do not include the proof as identical to that of \cite[Proposition 2.7]{T17}.

\begin{proposition}\label{prop:monotonicity-formula}
For any $y\in\mathbb{R}^d$ and $0\leq t<s<\infty$, it holds that
\begin{align}\label{eq:monotonicity-formula}
\frac{d}{dt} \int_{\mathbb{R}^d} \rho_{(y,s)}(x, t) \, d\mu_t^\varepsilon(x)
\leq \frac{1}{2(s - t)} \int_{\mathbb{R}^d} \rho_{(y,s)}(x, t) \, d\xi_t^\varepsilon(x) + \frac{1}{2}  \int_{\mathbb{R}^d} (g^\varepsilon(x,t))^2 \rho_{(y,s)}(x, t) \, d\mu_t^\varepsilon(x).
\end{align}   
For any $y\in\mathbb{R}^d$, $T>0$, and $0\leq t_1\leq t_2<s<\infty$ and $t_2\leq T$, we have
\begin{align}\label{eq:monotonicity-formula-integrating-factor}
\int_{\mathbb{R}^{d}} \rho_{(y, s)}(x, t) d \mu_{t}^{\varepsilon}(x)\bigg|_{t=t_{2}} &\leq\left(\int_{\mathbb{R}^{d}} \rho_{(y, s)}(x, t) d \mu_{t}^{\varepsilon}(x)\bigg|_{t=t_{1}}\right) e^{\frac{1}{2} \int_{t_{1}}^{t_{2}}\max\{{|\lambda^{\varepsilon}(t)|^{2},1}\} d t}\notag\\
&\qquad\qquad+\int_{t_1}^{t_2}e^{\int_{t}^{t_2}\frac{1}{2}\max\{{|\lambda^{\varepsilon}(\tau)|^{2},1}\}\,d\tau}\int_{\mathbb{R}^d}\frac{\rho_{(y, s)}(x, t)}{2(s-t)}d\xi_{t}^{\varepsilon}\,dt\notag\\
&\leq C_T\left(\int_{\mathbb{R}^{d}} \rho_{(y, s)}(x, t) d \mu_{t}^{\varepsilon}(x)\bigg|_{t=t_{1}}+\int_{t_1}^{t_2}\int_{\mathbb{R}^d}\frac{\rho_{(y, s)}(x, t)}{2(s-t)}d\xi_{t}^{\varepsilon}(x)\,dt\right)
\end{align}
for $\varepsilon\in(0,\varepsilon_1)$, where $C_T>0,\ \varepsilon_1\in(0,1)$ are as in Proposition \ref{prop:L2-estimate} (replacing $C_T$ by a larger constant). Here, $\mu_t^{\varepsilon}$ is extended periodically on $\R^d$.
\end{proposition}

\medskip

Abusing notations, we also write
$$
\xi^\varepsilon(x,t):=\frac{\varepsilon|\nabla \phi^\varepsilon(x,t)|^2}{2}-\frac{W(\phi^\varepsilon(x,t))}{\varepsilon},\quad (x,t)\in \mathbb{R}^d\times [0,\infty).
$$
We now give an upper bound of $\xi^{\varepsilon}$. In order to achieve this, we introduce the rescaling by
$$
\hat{\phi}^\varepsilon(\hat{x}, \hat{t}) := \phi^\varepsilon(\varepsilon \hat{x}, \varepsilon^2 \hat{t})\qquad\text{and}\qquad
\hat{g}^\varepsilon(\hat{x},\hat{t}) := g^\varepsilon(\varepsilon \hat{x},\varepsilon^2 \hat{t})\qquad\text{for }(\hat{x},\hat{t})\in \mathbb{R}^d\times [0,\infty).
$$
Then, we have, from \eqref{eq:L-infty-g-ep},
\begin{align}\label{eq:L-infty-g-ep-rescaled}
\|\hat{g}^\varepsilon \|_{L^{\infty}(\mathbb{R}^d\times [0,\infty))} \leq C\varepsilon^{-\alpha/2}, \qquad \qquad \| \nabla_{\hat{x}} \hat{g}^\varepsilon \|_{L^{\infty}(\mathbb{R}^d\times [0,\infty))} \leq C\varepsilon^{(1-\alpha)/2},
\end{align}
and
\begin{align}\label{eq:phi-epsilon-rescaled}
\hat{\phi}_{\hat{t}}^\varepsilon = \Delta_{\hat{x}} \hat{\phi}^\varepsilon - W'(\hat{\phi}^\varepsilon) + \varepsilon \hat{g}^\varepsilon \sqrt{2W(\hat{\phi}^\varepsilon)} \qquad \text{in}\,\, \mathbb{R}^d \times (0, \infty).
\end{align}

Before we prove Proposition \ref{prop:upper-bound-discrepancy}, we state the following lemma, which is obtained by a standard gradient estimate argument using \eqref{eq:derivative-r}, \eqref{eq:initial-data}, \eqref{eq:L-infty-g-ep-rescaled}.

\begin{lemma}\label{lem:gradient-estimate}
There exists a constant $C>0$ such that
\begin{align*}
\sup_{\varepsilon\in(0,1)}\|\nabla\hat{\phi}^{\varepsilon}\|_{L^{\infty}(\mathbb{R}^d\times[0,\infty))}\leq C.
\end{align*}
\end{lemma}

\begin{proposition}\label{prop:upper-bound-discrepancy}
There exists a constant $C>0$ such that for any $\varepsilon\in\left(0,1\right)$, it holds that
\begin{align*}
\sup_{\mathbb{R}^d\times[0,\infty)}\xi^{\varepsilon}\leq C\varepsilon^{-\frac{1}{2}\left(1+\frac{\alpha}{2}\right)}.
\end{align*}
\end{proposition}
\begin{proof}
The proof follows the arguments developed in \cite{C96,M85} with suitable adaptations, which we include for completeness. Let
\begin{align}\label{eq:xi-rescaled}
\hat{\xi}^\varepsilon(\hat{x}, \hat{t}) := \frac{|\nabla_{\hat{x}} \hat{\phi}^\varepsilon(\hat{x}, \hat{t})|^2}{2} - W(\hat{\phi}^\varepsilon(\hat{x}, \hat{t})) - G(\hat{\phi}^\varepsilon(\hat{x}, \hat{t})),
\end{align}
where $G\in C^{\infty}(\mathbb{R})$ will be taken later. We will rely on the maximum principle. We compute that
\begin{align*}
\partial_{\hat{t}} \hat{\xi}^\varepsilon - \Delta_{\hat{x}} \hat{\xi}^\varepsilon = \nabla_{\hat{x}}& \hat{\phi}^\varepsilon \cdot \nabla_{\hat{x}} \partial_{\hat{t}} \hat{\phi}^\varepsilon - (W' + G') \partial_{\hat{t}} \hat{\phi}^\varepsilon - |\nabla_{\hat{x}}^2 \hat{\phi}^\varepsilon|^2\\& - \nabla_{\hat{x}} \hat{\phi}^\varepsilon \cdot \nabla_{\hat{x}} (\Delta_{\hat{x}} \hat{\phi}^\varepsilon)
+ (W' + G') \Delta_{\hat{x}} \hat{\phi}^\varepsilon + (W'' + G'') |\nabla_{\hat{x}} \hat{\phi}^\varepsilon|^2,
\end{align*}
where $W'=W'(\hat{\phi}^{\varepsilon}),\ G'=G'(\hat{\phi}^{\varepsilon})$. By \eqref{eq:phi-epsilon-rescaled}, we subsequently have
\begin{align}\label{eq:parabolic-derivative}
\partial_{\hat{t}} \hat{\xi}^\varepsilon - \Delta_{\hat{x}} \hat{\xi}^\varepsilon = W'&(W'+G') - |\nabla_{\hat{x}}^2 \hat{\phi}^\varepsilon|^2 + \text{$G''$} |\nabla_{\hat{x}} \hat{\phi}^\varepsilon|^2\\& + \varepsilon \nabla_{\hat{x}} \hat{\phi}^\varepsilon \cdot \nabla_{\hat{x}} \left(\hat{g}^\varepsilon \sqrt{2W(\hat{\phi}^{\varepsilon})}\right) - \varepsilon (W' + G') \hat{g}^\varepsilon \sqrt{2W(\hat{\phi}^{\varepsilon})}.
\end{align}
We differentiate \eqref{eq:xi-rescaled} with respect to $\hat{x}_j$ and use Cauchy-Schwarz's inequality to obtain
\begin{align*}
|\nabla_{\hat{x}} \hat{\phi}^\varepsilon|^2 |\nabla_{\hat{x}}^2 \hat{\phi}^\varepsilon|^2 & \geq \sum_{j=1}^n \left( \sum_{i=1}^n \partial_{\hat{x}_i} \hat{\phi}^\varepsilon \partial_{\hat{x}_i \hat{x}_j} \hat{\phi}^\varepsilon \right)^2 = \sum_{j=1}^n \left( \partial_{\hat{x}_j} \hat{\xi}^\varepsilon + (W' + G') \partial_{\hat{x}_j} \hat{\phi}^\varepsilon \right)^2\\
& \geq 2((W' + G') \nabla_{\hat{x}} \hat{\phi}^\varepsilon) \cdot \nabla_{\hat{x}} \hat{\xi}^\varepsilon\notag+ (W' + G')^2 |\nabla_{\hat{x}} \hat{\phi}^\varepsilon|^2 .%
\end{align*}
We divide by $|\nabla_{\hat{x}}\hat{\phi}^{\varepsilon}|^2$ on $\{|\nabla_{\hat{x}}\hat{\phi}^{\varepsilon}|\neq0\}$ and put into \eqref{eq:parabolic-derivative} to obtain
\begin{align}\label{eq:parabolic-derivative-cauchy-schwarz}
\partial_{\hat{t}} \hat{\xi}^\varepsilon - \Delta_{\hat{x}} \hat{\xi}^\varepsilon \leq& -(G')^2 - W'G' - \frac{2(W' + G') (\nabla_{\hat{x}} \hat{\phi}^\varepsilon \cdot \nabla_{\hat{x}} \hat{\xi}^\varepsilon)}{|\nabla_{\hat{x}} \hat{\phi}^\varepsilon|^2} + G'' |\nabla_{\hat{x}} \hat{\phi}^\varepsilon|^2\notag \\
&+ \varepsilon \nabla_{\hat{x}} \hat{\phi}^\varepsilon \cdot \nabla_{\hat{x}} \left(\hat{g}^\varepsilon \sqrt{2W(\hat{\phi}^{\varepsilon})}\right) - \varepsilon (W' + G') \hat{g}^\varepsilon \sqrt{2W(\hat{\phi}^{\varepsilon})}.
\end{align}
Let $G(s) := \varepsilon^{\gamma} \left(1 - \frac{1}{8}s^2\right)$ for $s\in\mathbb{R}$ where $\gamma:=\frac{1}{2}\left(1-\frac{\alpha}{2}\right)\in\left(0,\frac{1}{2}\right)$, so that
\begin{align}\label{eq:property-G}
0<G\leq \varepsilon^{\gamma}, \quad W'G' \geq 0, \quad G'' = -\frac{\varepsilon^{\gamma}}{4}\qquad\text{for }s\in(-1,1).
\end{align}
Then, \eqref{eq:property-G} yields, with $W=W(\hat{\phi}^{\varepsilon})$,
\begin{align}\label{eq:M1}
\partial_t \hat{\xi}^{\varepsilon} - \Delta_{\hat{x}} \hat{\xi}^{\varepsilon} \leq& -\frac{2((W' + G')(\nabla_{\hat{x}} \hat{\phi}^{\varepsilon} \cdot \nabla_{\hat{x}} \hat{\xi}^{\varepsilon})}{|\nabla_{\hat{x}} \hat{\phi}^{\varepsilon}|^2} - \frac{\varepsilon^{\gamma}}{4} |\nabla_{\hat{x}} \hat{\phi}^{\varepsilon}|^2 \\
&+ \varepsilon |\nabla_{\hat{x}} \hat{\phi}^{\varepsilon}| |\nabla_{\hat{x}} (\hat{g}^{\varepsilon} \sqrt{2W})| + \varepsilon \sqrt{2W} (|W'| + |G'|) |\hat{g}^{\varepsilon}|.
\end{align}
The last two terms are bounded by, from Lemma \ref{lem:gradient-estimate}, \eqref{eq:L-infty-g-ep-rescaled}, and Proposition \ref{prop:strong-max-principle},
\begin{align*}
\varepsilon |\nabla_{\hat{x}} \hat{\phi}^{\varepsilon}| |\nabla_{\hat{x}} (\hat{g}^{\varepsilon} \sqrt{2W})| &+\varepsilon \sqrt{2W} (|W'| + |G'|) |\hat{g}^{\varepsilon}|\\
&\leq C\varepsilon  \left( \sqrt{2W} |\nabla_{\hat{x}} \hat{g}^{\varepsilon}| + |\hat{g}^{\varepsilon}| \frac{|W'|}{\sqrt{2W}} |\nabla_{\hat{x}} \hat{\phi}^{\varepsilon}| \right)+C\varepsilon^{1-\alpha/2} \leq C\varepsilon^{1-\alpha/2}.
\end{align*}
Therefore, from \eqref{eq:M1}, we deduce
\begin{align}\label{eq:max-principle-to-be-applied}
\partial_t \hat{\xi}^{\varepsilon} - \Delta_{\hat{x}} \hat{\xi}^{\varepsilon} \leq& -\frac{2((W' + G')(\nabla_{\hat{x}} \hat{\phi}^{\varepsilon} \cdot \nabla_{\hat{x}} \hat{\xi}^{\varepsilon})}{|\nabla_{\hat{x}} \hat{\phi}^{\varepsilon}|^2} - \frac{\varepsilon^{\gamma}}{4} |\nabla_{\hat{x}} \hat{\phi}^{\varepsilon}|^2 + C\varepsilon^{1-\alpha/2}.
\end{align}
Fix $\hat{T}>0$ and assume $\sup_{\mathbb{R}^d\times[0,\hat{T}]}\hat{\xi}^{\varepsilon}\geq C_1\varepsilon^{\gamma}$ with $C_1>0$ to be chosen. Since the function $G$ serves as a penalization term, there exists a maximizer $(\hat{x}_0, \hat{t}_0) \in \mathbb{R}^d \times (0,\hat{T}]$ of $\xi^{\varepsilon}$ by the fact that $\xi^{\varepsilon}(\cdot,0)\leq0$ (from \eqref{eq:xi-0-nonpositive} with modifications), Proposition \ref{prop:strong-max-principle}, and Lemma \ref{lem:gradient-estimate}. Then, $|\nabla_{\hat{x}} \hat{\phi}^{\varepsilon}(\hat{x}_0, \hat{t}_0)|^2 \geq 2C_1\varepsilon^{\gamma}$ so that \eqref{eq:max-principle-to-be-applied} is available near $(\hat{x}_0, \hat{t}_0)$.

We therefore are able to apply the maximum principle to \eqref{eq:max-principle-to-be-applied} at $(\hat{x}_0, \hat{t}_0)$ to obtain
\begin{align*}
0 \leq - \frac{C_1\varepsilon^{1-\alpha/2}}{2} + C\varepsilon^{1-\alpha/2}.
\end{align*}
We derive a contradiction for $C_1>2C$. As $\hat{T}>0$ was arbitrary, we derive
\begin{align*}
\sup_{\mathbb{R}^d\times[0,\infty)}\hat{\xi}^{\varepsilon}\leq C_1\varepsilon^{\frac{1}{2}\left(1-\frac{\alpha}{2}\right)}.
\end{align*}
From the fact that $G\leq\varepsilon^{\gamma}$ and the rescaling, we obtain the conclusion.
\end{proof}

\subsection{An upper bound of the density ratio}\label{subsec:upper-bound-density-ratio}
We define the density ratio by
\begin{align*}
D^{\varepsilon}(t):=\sup_{x\in\R^d,\ r\in(0,4)} \frac{\mu_t^{\varepsilon}(B_r(x))}{\omega_{d-1}r^{d-1}}
\end{align*}
for $t\geq0$, $\varepsilon\in(0,1)$, where $\omega_{d-1}$ is the volume of the $(d-1)$-dimensional unit ball. Here and so forth, we abuse notations by $\mu_t^{\varepsilon}(A):=\mu_t^{\varepsilon}(\chi_A)$ in the definition \eqref{eq:def-mu}.

For our choice of initial data \eqref{eq:initial-data}, the density ratio is bounded by an $\varepsilon$-independent number $D_0>0$ (see \cite[p. 423]{I93}) so that
\begin{align}\label{eq:initial-density-ratio}
    \sup_{\varepsilon\in(0,1)}D^{\varepsilon}(0)\leq D_0.
\end{align}
We will suppress the $D_0$-dependence henceforth as this is determined by the problem.

\medskip

From now until Lemma \ref{lem:lemma4.7-TT16}, we suppose that
\begin{align}\label{eq:temporary-density-ratio}
\sup_{t\in[0,T_1]}D^{\varepsilon}(t)\leq D_1
\end{align}
for some numbers $T_1,D_1>0$. Here, $D_1$ is taken to be larger than $D_0$ in \eqref{eq:initial-density-ratio} and does not depend on $\varepsilon$. We mention that assuming \eqref{eq:temporary-density-ratio} does not cause a circularity with the conclusion of Proposition \ref{prop:upper-bound-density-ratio}, which will be explained in detail in the proof of the proposition. For this sake, dependency of numbers on other quantities will be carefully tracked by labeling with numbers in the rest of this subsection. The quantities given by the setting (such as $M_0$, the double-well function $W$, the obstacle $O$, the spatial dimension $d$, the initial datum $U_0$ and its approximation $\{U_0^i\}_{i=1}^{\infty}$, the numbers $\alpha,\beta,\omega\in(0,1)$, and so on) that are already fixed from the beginning will be referred as the data or the problem when we proceed dependency tracking in proofs.


In the following, we set
\begin{align}\label{eq:lambda-remind}
\lambda:=\frac{1}{2}\left(1+\frac{\alpha}{2}\right)\in\left(\frac{1}{2},1\right),\qquad\lambda':=\frac{1}{2}\left(1+\lambda\right)\in\left(\lambda,1\right).
\end{align}

\begin{lemma}\label{lem:lemma4.5-TT16}
Assume \eqref{eq:temporary-density-ratio} and let $\varepsilon_1\in(0,1)$ be as in Proposition \ref{prop:L2-estimate}. Then, there exist $c_1=c_1(D_1,T_1)>1$, $c_2=c_2(D_1,T_1)>0$, and $\varepsilon_3=\varepsilon_3(D_1,T_1)\in(0,\varepsilon_1)$ satisfying the following property: for any $\varepsilon\in(0,\varepsilon_3)$, $(y,s)\in\Omega\times(0,T_1]$ with $|\phi^{\varepsilon}(y,s)|<\sqrt{\frac23}$, and $t\in[\max\{0,s-2\varepsilon^{2\lambda'}\},s]$,
\begin{align*}
c_2R^{d-1}\leq\mu_t^{\varepsilon}(B_R(y)),
\end{align*}
where $R=c_1(s+\varepsilon^2-t)^{1/2}$.
\end{lemma}
\begin{proof}
First of all, there exists a constant $\gamma_0\in(0,1)$ depending only on the data such that
\begin{align*}
|\phi^{\varepsilon}(x, s)| \leq \gamma_0 \sup_{x \in \Omega} \varepsilon |\nabla \phi^{\varepsilon}(y, s)| + |\phi^{\varepsilon}(y, s)| \leq C \gamma_0 + \sqrt{\frac23} \leq \frac12\left(1+\sqrt{\frac23}\right) < 1
\end{align*}
for any $x\in B_{\gamma_0\varepsilon}(y)$,where $C>0$ is the constant in Lemma \ref{lem:gradient-estimate}. Hence, we have
\begin{align*}
W(\phi^{\varepsilon}(x,s))\geq \inf_{|s|\leq\frac12\left(1+\sqrt{\frac23}\right)}W(s)(=:c)>0\qquad\text{for }x\in B_{\gamma_0\varepsilon}(y)
\end{align*}
and
\begin{align}\label{eq:lemma4.5-TT16-1}
\int_{B_{\gamma_0 \varepsilon}(y)} \rho_{(y,s+\varepsilon^2)}(x, s) \, d\mu^{\varepsilon}_s(x) \geq \frac{c}{(4\pi)^{\frac{d-1}{2}} \varepsilon^d} \int_{B_{\gamma_0 \varepsilon}(y) } e^{-\frac{|x-y|^2}{4\varepsilon^2}} \, dx \geq M,
\end{align}
where $M>0$ is a constant depending only on the data.

For $\tau\in(0,s+\varepsilon^2)$, we record the fact that
\begin{align}\label{eq:lemma4.5-TT16-2}
\int_{\R^d}\rho_{(y,s+\varepsilon^2)}(x,\tau)\,dx
&=(4\pi(s+\varepsilon^2-\tau))^{1/2},
\end{align}
which follows from the standard properties of the heat kernel.

\medskip

Let the constants $C_T$ and $C$ be as in \eqref{eq:monotonicity-formula-integrating-factor} and Proposition \ref{prop:upper-bound-discrepancy}, respectively. We apply Proposition \ref{prop:upper-bound-discrepancy} and \eqref{eq:monotonicity-formula-integrating-factor} with $t\leq s<s+\varepsilon^2$ (and with $s\leq T_1$) in place of $t_1\leq t_2<s$ as well as \eqref{eq:lemma4.5-TT16-1}, \eqref{eq:lemma4.5-TT16-2} to obtain
\begin{align}\label{eq:lemma4.5-TT16-3}
M&\leq C_{T_1}\left(\int_{\R^d}\rho_{(y,s+\varepsilon^2)}(x,t)\,d\mu^{\varepsilon}_t(x)+\int_t^sC\varepsilon^{-\lambda}\frac{(4\pi(s+\varepsilon^2-\tau))^{1/2}}{2(s+\varepsilon^2-\tau)}\,d\tau\right)\notag\\
&\leq C_{T_1}\left(\int_{\R^d}\rho_{(y,s+\varepsilon^2)}(x,t)\,d\mu^{\varepsilon}_t(x)+2C\sqrt{\pi}\varepsilon^{\lambda'-\lambda}\right).
\end{align}
Here, we used the fact that $\int_t^s(s+\varepsilon^2-\tau)^{-1/2}\,d\tau\leq2\varepsilon^{\lambda'}$ from the condition $t\in[\max\{0,s-2\varepsilon^{2\lambda'}\},s]$ and we restrict $\varepsilon\in(0,\varepsilon_1)$ small enough for Propositions \ref{prop:L2-estimate}, \ref{prop:monotonicity-formula} to be valid.

\medskip

Let $R:=c_1(s+\varepsilon^2-t)^{1/2}$ so that $s+\varepsilon^2-t=\frac{R^2}{c_1^2}$. Here, $c_1>1$ is to be chosen. We first note that as $\mu_t^{\varepsilon}$ is extended periodically on $\R^d$, we have
\begin{align}\label{eq:from-periodicity}
\mu_t^{\varepsilon}(B_r(y))\leq C(2\lceil r\rceil)^d\leq C(2r+2)^d\leq4^dCr^d\qquad\text{for }r\geq1.
\end{align}
Here, $C>0$ is the constant as in \eqref{eq:estimate-from-energy-dissipation}. Then, using \eqref{eq:temporary-density-ratio}, \eqref{eq:from-periodicity},
\begin{align}\label{eq:lemma4.5-TT16-4}
\int_{\R^d}\rho_{(y,s+\varepsilon^2)}(x,t)\,d\mu^{\varepsilon}_t(x)&=\frac{1}{\left(4\pi(s+\varepsilon^2-\tau)\right)^{\frac{d-1}{2}}}\int_{\R^d}e^{-\frac{|x-y|^2}{4(s+\varepsilon^2-\tau)}}\,d\mu_t^{\varepsilon}(x)\notag\\
&=\frac{c_1^{d-1}}{\left(2\sqrt{\pi}R\right)^{d-1}}\int_{\R^d}e^{-\frac{c_1^2|x-y|^2}{4R^2}}\,d\mu_t^{\varepsilon}(x)\notag\\
&=\frac{c_1^{d-1}}{\left(2\sqrt{\pi}R\right)^{d-1}}\int_0^1\mu_t^{\varepsilon}\left(\left\{x\in\R^d\,:\,e^{-\frac{c_1^2|x-y|^2}{4R^2}}\geq l\right\}\right)\,dl\notag\\
&=\frac{c_1^{d-1}}{\left(2\sqrt{\pi}R\right)^{d-1}}\int_0^1\mu_t^{\varepsilon}\left(B_{\frac{2R}{c_1}(-\log(l))^{1/2}}(y)\right)\,dl\notag\\
&\leq\frac{c_1^{d-1}}{\left(2\sqrt{\pi}R\right)^{d-1}}\left(\int_0^{e^{-\frac{4c_1^2}{R^2}}}\mu_t^{\varepsilon}\left(B_{\frac{2R}{c_1}(-\log(l))^{1/2}}(y)\right)\,dl\right.\notag\\
&\qquad+\left.\int_{e^{-\frac{4c_1^2}{R^2}}}^{e^{-\frac{c_1^2}{4}}}\mu_t^{\varepsilon}\left(B_{\frac{2R}{c_1}(-\log(l))^{1/2}}(y)\right)\,dl+\int_{e^{-\frac{c_1^2}{4}}}^1\mu_t^{\varepsilon}(B_R(y))\,dl\right).
\end{align}

We estimate the first term of the right-hand side of \eqref{eq:lemma4.5-TT16-4} by using \ref{eq:from-periodicity}. Let $c(d)>0$ be a constant that depends only on $d$ such that $a^{\frac{d}{2}}e^{-a}\leq c(d)e^{-\frac{a}{2}}$ for any $a>0$. Also, from the assumption of the lemma, we have $s+\varepsilon^2-t\leq\varepsilon^2+2\varepsilon^{2\lambda'}\leq3\varepsilon^{2\lambda'}$, and therefore,
\begin{align}\label{eq:lemma4.5-TT16-4-1}
\frac{c_1^{d-1}}{\left(2\sqrt{\pi}R\right)^{d-1}}\int_0^{e^{-\frac{4c_1^2}{R^2}}}\mu_t^{\varepsilon}\left(B_{\frac{2R}{c_1}(-\log(l))^{1/2}}(y)\right)\,dl&\leq\frac{c_1^{d-1}}{\left(2\sqrt{\pi}R\right)^{d-1}}\int_0^{e^{-\frac{4c_1^2}{R^2}}}4^dC\left(\frac{2R}{c_1}\right)^d\left(-\log(l)\right)^{\frac{d}{2}}\,dl\notag\\
&\leq \frac{2^{2d+1}}{\pi^{\frac{d-1}{2}}}C\left(\frac{R}{c_1}\right)\int_{\frac{4c_1^2}{R^2}}^{\infty}a^{\frac{d}{2}}e^{-a}da\notag\\
&\leq \frac{2^{2d+1}}{\pi^{\frac{d-1}{2}}}C\left(\frac{R}{c_1}\right)c(d)\int_{\frac{4c_1^2}{R^2}}^{\infty}e^{-\frac{a}{2}}da\notag\\
&\leq\frac{3\cdot2^{2d+2}c(d)C}{\pi^{\frac{d-1}{2}}}\varepsilon^{\lambda'}.
\end{align}
The second and third terms of the right-hand side of \eqref{eq:lemma4.5-TT16-4} are estimated by the assumption \eqref{eq:temporary-density-ratio} as
\begin{align}\label{eq:lemma4.5-TT16-4-23}
&\frac{c_1^{d-1}}{\left(2\sqrt{\pi}R\right)^{d-1}}\left(\int_{e^{-\frac{4c_1^2}{R^2}}}^{e^{-\frac{c_1^2}{4}}}\mu_t^{\varepsilon}\left(B_{\frac{2R}{c_1}(-\log(l))^{1/2}}(y)\right)\,dl+\int_{e^{-\frac{c_1^2}{4}}}^1\mu_t^{\varepsilon}(B_R(y))\,dl\right)\notag\\
&\leq\frac{c_1^{d-1}}{\left(2\sqrt{\pi}R\right)^{d-1}}\left(D_1\omega_{d-1}\int_{e^{-\frac{4c_1^2}{R^2}}}^{e^{-\frac{c_1^2}{4}}}\left(\frac{2R}{c_1}(-\log(l))^{1/2}\right)^{d-1}\,dl+\mu_t^{\varepsilon}(B_R(y))\right)\notag\\
&\leq\frac{c_1^{d-1}}{\left(2\sqrt{\pi}R\right)^{d-1}}\mu_t^{\varepsilon}(B_R(y))+\frac{\omega_{d-1}}{\pi^{\frac{d-1}{2}}}D_1\int_0^{e^{-\frac{c_1^2}{4}}}\left(-\log(l)\right)^{\frac{d-1}{2}}\,dl\notag\\
&\leq\frac{c_1^{d-1}}{\left(2\sqrt{\pi}R\right)^{d-1}}\mu_t^{\varepsilon}(B_R(y))+2^{\frac{d+1}{2}}e^{-\frac{c_1^2}{8}}D_1.
\end{align}
Here, we used the fact that $\int_0^{e^{-\frac{c_1^2}{4}}}\left(-\log(l)\right)^{\frac{d-1}{2}}\,dl\leq\frac{2^{\frac{d+1}{2}}\pi^{\frac{d-1}{2}}}{\omega_{d-1}}e^{-\frac{c_1^2}{8}}$ in the last line.

\medskip

Combining \eqref{eq:lemma4.5-TT16-1}-\eqref{eq:lemma4.5-TT16-4-23}, we obtain
\begin{align*}
M\leq C_{T_1}\left(\frac{c_1^{d-1}}{\left(2\sqrt{\pi}R\right)^{\frac{d-1}{2}}}\mu_t^{\varepsilon}(B_R(y))+\frac{3\cdot2^{2d+2}c(d)C}{\pi^{\frac{d-1}{2}}}\varepsilon^{\lambda'}+2^{\frac{d+1}{2}}e^{-\frac{c_1^2}{8}}D_1+2C\sqrt{\pi}\varepsilon^{\lambda'-\lambda}\right).
\end{align*}
Choose $c_1>1$ large depending only on $D_1,T_1$ and the data so that
\begin{align*}
C_{T_1}2^{\frac{d+1}{2}}e^{-\frac{c_1^2}{8}}D_1\leq\frac14M.
\end{align*}
Choose $\varepsilon_3\in(0,\varepsilon_1)$ small depending only on $D_1,T_1$ and the data so that
\begin{align*}
\frac{3\cdot2^{2d+2}c(d)C}{\pi^{\frac{d-1}{2}}}\varepsilon^{\lambda'}+2C_{T_1}C\sqrt{\pi}\varepsilon^{\lambda'-\lambda}\leq\frac14M\qquad\text{for all }\varepsilon\in(0,\varepsilon_3)
\end{align*}
With the choices of $c_1$ and $\varepsilon_3$, we have
\begin{align*}
c_2:=\frac{(2\sqrt{\pi})^{d-1}M}{2C_{T_1}c_1^{d-1}}\leq R^{-(d-1)}\mu_t^{\varepsilon}(B_R(y)),
\end{align*}
which proves the lemma.
\end{proof}

We mention that the rescaling property the forcing term, which follows from \eqref{eq:L-infty-g-ep},
\begin{align}\label{eq:forcing-term-rescaling}
\|g^{\varepsilon}\|_{L^{\infty}(\Omega\times[0,\infty))}\leq C\varepsilon^{-\alpha/2}\qquad\text{and}\qquad\|\nabla g^{\varepsilon}\|_{L^{\infty}(\Omega\times[0,\infty))}\leq C\varepsilon^{-(\alpha+1)/2}.
\end{align}

\begin{lemma}\label{lem:lemma4.6-TT16}
Assume \eqref{eq:temporary-density-ratio} and let $\varepsilon_3\in(0,1)$ be as in Lemma \ref{lem:lemma4.5-TT16}. Then, there exist $c_3=c_3(D_1,T_1)>0$ and $\varepsilon_4=\varepsilon_4(D_1,T_1)\in(0,\varepsilon_3)$ satisfying the following property: for any $y\in\Omega$, $\varepsilon\in(0,\varepsilon_4)$, $r\in(\varepsilon^{\lambda'},1)$, and $t_0\in[2\varepsilon^{2\lambda'},T]\cap[0,T_1]$, it holds that
\begin{align*}
\int_{B_{r}(y)}\left(\frac{\varepsilon|\nabla \phi^{\varepsilon}|^{2}}{2}-\frac{W(\phi^{\varepsilon})}{\varepsilon}\right)_{+}(x, t_0)\, d x\leq c_3\varepsilon^{\lambda'-\lambda}r^{d-1}.
\end{align*}
Here, $(\cdot)_+$ denotes the positive part, that is, $x_+=\max\{x,0\}$ for $x\in\R$.
\end{lemma}

\begin{proof}
We assume $T_1\geq2\varepsilon^{2\lambda'}$ as the lemma is vacuously true otherwise. Fix $y\in\Omega$, $r\in\left(\varepsilon^{\lambda'},1\right)$, and $t_0\in[2\varepsilon^{2\lambda'},T]\cap[0,T_1]$. Set
\begin{align*}
\begin{cases}
\widetilde{A}:=\left\{ x \in B_{2r}(y)\, : \text{ there is } \tilde{t}\in[t_0-\varepsilon^{2\lambda'},t_0] \text{ such that } |\phi^{\varepsilon}(x, \tilde{t})| \leq \sqrt{\frac23} \right\},\\
A:=\left\{ x \in B_{2r+2c_1\varepsilon^{\lambda'}}(y)\, :\, \mathrm{dist}(x,\tilde{A}) < 2c_1\varepsilon^{\lambda'} \right\},
\end{cases}
\end{align*}
where $c_1$ is as in Lemma \ref{lem:lemma4.5-TT16}. By the Vitali covering theorem applied to the covering $\left\{\overline{B}_{2c_1\varepsilon^{\lambda'}}(x)\right\}_{x\in\widetilde{A}}$ of $A$, there exist pairwise disjoint balls $\left\{B_{2c_1\varepsilon^{\lambda'}}(x_i)\right\}_{i=1}^N$ such that
\begin{align*}
x_i \in \tilde{A} \text{ for each } i = 1, \dots, N \quad \text{and} \quad A \subset \cup_{i=1}^N \bar{B}_{10c_1\varepsilon^{\lambda'}}(x_i).
\end{align*}
By the definition of $\widetilde{A}$, for each $x_i$, there exists $\widetilde{t}_i\in[t_0-\varepsilon^{2\lambda'},t_0]$ such that $|\phi^{\varepsilon}(x_i,\widetilde{t}_i)|\leq\sqrt{\frac23}$. Let $\hat{t}:=t_0-2\varepsilon^{2\lambda'}$ so that $\hat{t}\geq0$ and
\begin{align}\label{eq:tilde-hat-close}
\varepsilon^{2\lambda'} \leq \tilde{t}_i - \hat{t} \leq 2\varepsilon^{2\lambda'}.
\end{align}
For $\varepsilon\in(0,\varepsilon_3)$, the assumption of Lemma \ref{lem:lemma4.5-TT16} is satisfied with $\widetilde{t}_i,x_i,\hat{t}$, and $R_i:=c_1(\widetilde{t}_i+\varepsilon^2-\hat{t})^{1/2}$ in place of $s,y,t$, and $R$, respectively, and thus, we have
\begin{align*}
c_2R_i^{d-1}\leq\mu_{\hat{t}}^{\varepsilon}(B_{R_i}(x_i))\qquad\text{for }i=1,\cdots,N.
\end{align*}
By \eqref{eq:tilde-hat-close}, we see that $c_1(\varepsilon^{2\lambda'} + \varepsilon^2)^{\frac{1}{2}} \leq R_i \leq c_1(2\varepsilon^{2\lambda'} + \varepsilon^2)^{\frac{1}{2}} \leq 2c_1\varepsilon^{\lambda'}
$, which implies
\begin{align*}
c_1^{d-1}c_2\varepsilon^{\lambda'(d-1)}\leq\mu_{\hat{t}}^{\varepsilon}(B_{2c_1\varepsilon^{\lambda'}}(x_i))\qquad\text{for }i=1,\cdots,N.
\end{align*}
As the balls $\left\{B_{2c_1\varepsilon^{\lambda'}}(x_i)\right\}_{i=1}^N$ are pairwise disjoint and $B_{2c_1\varepsilon^{\lambda'}}(x_i)\subset B_{2r+2c_1\varepsilon^{\lambda'}}(y)$, we have
\begin{align}\label{eq:consequence-lemma4.5}
Nc_1^{d-1}c_2\varepsilon^{\lambda'(d-1)}\leq\mu_{\hat{t}}^{\varepsilon}(B_{2r+2c_1\varepsilon^{\lambda'}}(y)).
\end{align}
Therefore, using \eqref{eq:consequence-lemma4.5} and the fact that $A \subset \cup_{i=1}^N \bar{B}_{10c_1\varepsilon^{\lambda'}}(x_i)$, we obtain
\begin{align*}
\mathcal{L}^d(A)\leq N\omega_d(10c_1\varepsilon^{\lambda'})^d&\leq 10^d\omega_dc_1c_2^{-1}\varepsilon^{\lambda'}\mu_{\hat{t}}^{\varepsilon}(B_{2r+2c_1\varepsilon^{\lambda'}}(y))\\
&\leq10^d\omega_dc_1c_2^{-1}\varepsilon^{\lambda'}\omega_{d-1}D_1(2r+2c_1\varepsilon^{\lambda'})^{d-1}\\
&\leq4^{d-1}10^dc_1^dc_2^{-1}\omega_d\omega_{d-1}D_1\varepsilon^{\lambda'}r^{d-1}.
\end{align*}
In the last line, we used the fact that $2r<2c_1r$ and $\varepsilon^{\lambda'}<r$. Consequently, by Proposition \ref{prop:upper-bound-discrepancy}, we obtain
\begin{align}\label{eq:part1-lemma4.6}
\int_{A\cap B_{r}(y)}\left(\frac{\varepsilon|\nabla \phi^{\varepsilon}|^{2}}{2}-\frac{W(\phi^{\varepsilon})}{\varepsilon}\right)_{+}(x, t_0)\, d x\leq \mathcal{L}^d(A)C\varepsilon^{-\lambda} \leq c_4\varepsilon^{\lambda'-\lambda}r^{d-1}.
\end{align}
Here, the constant $C$ is as in Proposition \ref{prop:upper-bound-discrepancy} and $c_4:=4^{d-1}10^dCc_1^dc_2^{-1}\omega_d\omega_{d-1}D_1$.

\medskip

We now estimate $\int_{B_{r}(y)\setminus A}\left(\frac{\varepsilon|\nabla \phi^{\varepsilon}|^{2}}{2}-\frac{W(\phi^{\varepsilon})}{\varepsilon}\right)_{+}(x, t_0)\, d x$. Let $\phi\in\mathrm{Lip}(B_{2r}(y))$ satisfying
\begin{align}\label{eq:phi-definition}
\phi(x) = \begin{cases} 1 & \text{if } x \in B_r(y)\setminus A, \\ 0 & \text{if } \mathrm{dist}(x, B_r(y)\setminus A) \geq \varepsilon^{\lambda'}, \end{cases}
\qquad|\nabla\phi|\leq 2\varepsilon^{-\lambda'},\quad\text{and}\quad0\leq\phi\leq1.
\end{align}
Since $r\geq\varepsilon^{\lambda'}$, $2c_1\varepsilon^{\lambda'}>\varepsilon^{\lambda'}$, and by the definitions of $\widetilde{A}$ and $\phi$, we have $\mathrm{supp}(\phi)\cap\widetilde{A}=\emptyset$, which implies
\begin{align}\label{eq:on-support-phi}
|\phi^{\varepsilon}(x, s)| \geq \sqrt{\frac23}\quad\left(\text{and thus,}\,\, W''(\phi^{\varepsilon}(x,s))\geq2\right), \quad \text{for }\, x \in \mathrm{supp}(\phi), \ s \in [t_0 - \varepsilon^{2\lambda'}, t_0].
\end{align}
We differentiate the equation \eqref{eq:phi-epsilon} with respect to $x_j$ for each $j$, multiply by $\phi^2\phi^{\varepsilon}_{x_j}$, sum over $j=1,\cdots,d$, and integrate over $\Omega$ (here, $\Omega$ represents the whole domain $\R^d$) to get
\begin{align*}
\frac{d}{dt} \int_{\Omega} \frac{1}{2} |\nabla \phi^{\varepsilon}|^2 \phi^2 \,dx &= \int_{\Omega} \left( \nabla \phi^{\varepsilon} \cdot \Delta \nabla \phi^{\varepsilon} - \frac{W''(\phi^{\varepsilon})}{\varepsilon^2} |\nabla \phi^{\varepsilon}|^2 \right) \phi^2 \,dx \\
&\qquad + \int_{\Omega} \left( \frac{\sqrt{2W(\phi^{\varepsilon})}}{\varepsilon} \nabla g^{\varepsilon} \cdot \nabla \phi^{\varepsilon} + g^{\varepsilon} \frac{W'(\phi^{\varepsilon})}{\varepsilon \sqrt{2W(\phi^{\varepsilon})}} |\nabla \phi^{\varepsilon}|^2 \right) \phi^2 \,dx.
\end{align*}
By integration by parts and the Cauchy-Schwarz inequality, we further obtain
\begin{align*}
\frac{d}{dt} \int_{\Omega} \frac{1}{2} |\nabla \phi^{\varepsilon}|^2 \phi^2 \,dx \
&\leq \int_{\Omega} |\nabla \phi|^2 |\nabla \phi^{\varepsilon}|^2 \,dx - \int_{\Omega} \frac{W''(\phi^{\varepsilon})}{\varepsilon^2} |\nabla \phi^{\varepsilon}|^2 \phi^2 \,dx \\
&\quad + \int_{\Omega} |\nabla g^{\varepsilon}| |\nabla \phi^{\varepsilon}| \frac{\sqrt{2W(\phi^{\varepsilon})}}{\varepsilon} \phi^2 \,dx + \int_{\Omega} |g^{\varepsilon}| \frac{|W'(\phi^{\varepsilon})|}{\varepsilon \sqrt{2W(\phi^{\varepsilon})}} |\nabla \phi^{\varepsilon}|^2 \phi^2 \,dx.
\end{align*}
By using \eqref{eq:phi-definition}, \eqref{eq:on-support-phi}, Young's inequality, and the fact that $\frac{(W'(s))^2}{2W(s)}$ is smooth, we obtain
\begin{align}\label{eq:part2-main-ineq-1-lemma4.6}
\frac{d}{dt} \int_{\Omega} \frac{1}{2} |\nabla \phi^{\varepsilon}|^2 \phi^2 \,dx \
&\leq 4\varepsilon^{-2\lambda'}\int_{\mathrm{supp}(\phi)\cap \Omega} |\nabla \phi^{\varepsilon}|^2 \,dx - \int_{\Omega} \frac{2}{\varepsilon^2} |\nabla \phi^{\varepsilon}|^2 \phi^2 \,dx\notag \\
&\qquad + \frac12\varepsilon^{-2\lambda'}\int_{\Omega} |\nabla \phi^{\varepsilon}|^2\phi^4\,dx + \frac12\varepsilon^{2\lambda'}\int_{\mathrm{supp}(\phi)\cap\Omega} |\nabla g^{\varepsilon}|^2\frac{2W(\phi^{\varepsilon})}{\varepsilon^2} \,dx\notag\\
&\qquad\qquad\qquad\qquad\qquad+ \sup_{s\in[-1,1]}\left|\frac{W'(s)}{\sqrt{2W(s)}}\right|\|g^{\varepsilon}\|_{L^{\infty}(\Omega)}\varepsilon^{-1}\int_{\Omega}|\nabla \phi^{\varepsilon}|^2 \phi^2 \,dx.
\end{align}

We bound some terms in \eqref{eq:part2-main-ineq-1-lemma4.6} by using \eqref{eq:forcing-term-rescaling}. By the fact that $\mathrm{supp}(\phi)\subset B_{2r}(y)$ and the assumption \eqref{eq:temporary-density-ratio}, it holds, with power comparison $2\lambda'-(\alpha+1)-1=-\frac12-\frac34\alpha$, that
\begin{align}\label{eq:part2-main-ineq-1-lemma4.6-term1}
\frac12\varepsilon^{2\lambda'}\int_{\mathrm{supp}(\phi)\cap\Omega} |\nabla g^{\varepsilon}|^2\frac{2W(\phi^{\varepsilon})}{\varepsilon^2} \,dx\leq C\varepsilon^{2\lambda'-(\alpha+1)-1}D_1\omega_{d-1}(2r)^{d-1}=c_5r^{d-1}\varepsilon^{-\frac12-\frac34\alpha}
\end{align}
with $c_5:=2^{d-1}\omega_{d-1}CD_1$. Here, $C>0$ is a constant that depends only on the data. Also, by \eqref{eq:forcing-term-rescaling} again and power comparison $2\lambda'-\left(1+\frac{1}{2}\alpha\right)=\frac12-\frac14\alpha\in\left(\frac14,\frac12\right)$, we have
\begin{align}\label{eq:part2-main-ineq-1-lemma4.6-term2}
\sup_{s\in[-1,1]}\left|\frac{W'(s)}{\sqrt{2W(s)}}\right|\|g^{\varepsilon}\|_{L^{\infty}(\Omega)}\varepsilon^{-1}\int_{\Omega}|\nabla \phi^{\varepsilon}|^2 \phi^2 \,dx&\leq C\varepsilon^{-1-\frac{1}{2}\alpha}\int_{\mathrm{supp}(\phi)\cap\Omega}\frac{1}{2}|\nabla\phi^{\varepsilon}|^2\,dx\notag\\
&\leq\frac12\varepsilon^{-2\lambda'}\int_{\mathrm{supp}(\phi)\cap\Omega}|\nabla\phi^{\varepsilon}|^2\,dx
\end{align}
for $\varepsilon\in(0,1)$ small enough depending only the data. Consequently, continuing the estimate \eqref{eq:part2-main-ineq-1-lemma4.6} with \eqref{eq:part2-main-ineq-1-lemma4.6-term1}, \eqref{eq:part2-main-ineq-1-lemma4.6-term2}, we have
\begin{align}\label{eq:part2-main-ineq-2-lemma4.6}
\frac{d}{dt} \int_{\Omega} \frac{1}{2} |\nabla \phi^{\varepsilon}|^2 \phi^2 \,dx \
&\leq 5\varepsilon^{-2\lambda'}\int_{\mathrm{supp}(\phi)\cap\Omega}|\nabla\phi^{\varepsilon}|^2\,dx- \int_{\Omega} \frac{2}{\varepsilon^2} |\nabla \phi^{\varepsilon}|^2 \phi^2 \,dx + c_5r^{d-1}\varepsilon^{-\frac12-\frac34\alpha}.
\end{align}

\medskip

We thus have, by \eqref{eq:part2-main-ineq-2-lemma4.6},
\begin{align*}
\frac{d}{dt}\left(e^{4\varepsilon^{-2}(t-t_0)}\int_{\Omega}\frac12|\nabla\phi^{\varepsilon}|^2\phi^2\,dx \right)
&=e^{4\varepsilon^{-2}(t-t_0)}\left(\frac{4}{\varepsilon^2}\int_{\Omega}\frac12|\nabla\phi^{\varepsilon}|\phi^2\,dx+\frac{d}{dt}\int_{\Omega}\frac12|\nabla\phi^{\varepsilon}|^2\phi^2\,dx\right)\\
&\leq e^{4\varepsilon^{-2}(t-t_0)}\left(5\varepsilon^{-2\lambda'}\int_{\mathrm{supp}(\phi)\cap\Omega}|\nabla\phi^{\varepsilon}|^2\,dx + c_5r^{d-1}\varepsilon^{-\frac12-\frac34\alpha}\right).
\end{align*}
By integrating on $[t_0-\varepsilon^{2\lambda'},t_0]$, we obtain
\begin{align*}
\int_{\Omega}\frac12|\nabla\phi^{\varepsilon}|^2\phi^2(x,&t_0)\,dx-e^{-4\varepsilon^{-2+2\lambda'}}\int_{\Omega}\frac12|\nabla\phi^{\varepsilon}|^2\phi^2(x,t_0-\varepsilon^{2\lambda'})\,dx\\
&\leq\int_{t_0-\varepsilon^{2\lambda'}}^{t_0}e^{4\varepsilon^{-2}(\tau-t_0)}\left(5\varepsilon^{-2\lambda'}\int_{\mathrm{supp}(\phi)\cap\Omega}|\nabla\phi^{\varepsilon}|^2(x,\tau)\,dx + c_5r^{d-1}\varepsilon^{-\frac12-\frac34\alpha}\right)\,d\tau.
\end{align*}
Put
\begin{align*}
M:=\sup_{\tau\in[t_0-\varepsilon^{2\lambda'},t_0]}\int_{\mathrm{supp}(\phi)\cap\Omega}\frac{1}{2}|\nabla\phi^{\varepsilon}|^2(x,\tau)\,dx.
\end{align*}
Then, with $2\lambda'-\frac12-\frac34\alpha=1-\frac12\alpha$, we have
\begin{align}\label{eq:part2-main-ineq-3-lemma4.6}
&\int_{\Omega}\frac12|\nabla\phi^{\varepsilon}|^2\phi^2(x,t_0)\,dx\notag\\
&\leq e^{-4\varepsilon^{-2+2\lambda'}}\int_{\Omega}\frac12|\nabla\phi^{\varepsilon}|^2\phi^2(x,t_0-\varepsilon^{2\lambda'})\,dx + \int_{t_0-\varepsilon^{2\lambda'}}^{t_0}e^{4\varepsilon^{-2}(\tau-t_0)}\left(5\varepsilon^{-2\lambda'}M + c_5r^{d-1}\varepsilon^{-\frac12-\frac34\alpha}\right)\,d\tau\notag\\
&\leq\left(e^{-4\varepsilon^{-2+2\lambda'}}+\frac52\varepsilon^{2-2\lambda'}\right)M+c_5r^{d-1}\varepsilon^{1-\frac12\alpha}\notag\\
&\leq 3\varepsilon^{2-2\lambda'}M+c_5r^{d-1}\varepsilon^{1-\frac12\alpha}.
\end{align}
We note that, by $\mathrm{supp}(\phi)\subset B_{2r}(y)$ and the assumption \eqref{eq:temporary-density-ratio}, we have
\begin{align}\label{eq:temporary-applied-to-M}
\varepsilon M\leq \omega_{d-1}D_1(2r)^{d-1}.
\end{align}
We also note that $B_r(y)\setminus A\subset\{\phi=1\}$ by the definition of $\phi$. Therefore, by \eqref{eq:part2-main-ineq-3-lemma4.6}, \eqref{eq:temporary-applied-to-M}, we obtain
\begin{align}\label{eq:part2-main-ineq-last-lemma4.6}
\int_{B_r(y)\setminus A}\frac{\varepsilon|\nabla\phi^{\varepsilon}|^2}{2}(x,t_0)\,dx&\leq\int_{\Omega}\frac{\varepsilon|\nabla\phi^{\varepsilon}|^2}{2}\phi^2(x,t_0)\,dx\notag\\
&\leq3\cdot2^{d-1}D_1\omega_{d-1}r^{d-1}\varepsilon^{2-2\lambda'}
+c_5r^{d-1}\varepsilon^{2-\frac12\alpha}\leq c_6r^{d-1}\varepsilon^{\lambda'-\lambda}
\end{align}
with $c_6=3\cdot2^{d-1}D_1\omega_{d-1}+c_5$ that depends only on $D_1$ and the data. Adding \eqref{eq:part2-main-ineq-last-lemma4.6} to \eqref{eq:part1-lemma4.6}, we prove the lemma with
\begin{align*}
c_3:=4^{d-1}10^dCc_1^dc_2^{-1}\omega_d\omega_{d-1}D_1 + 2^{d-1}\omega_{d-1}CD_1 + 3\cdot2^{d-1}D_1\omega_{d-1},
\end{align*}
which depends only on $D_1,T_1$, and the data.
\end{proof}


\begin{lemma}\label{lem:lemma4.7-TT16}
Assume \eqref{eq:temporary-density-ratio} and let $\varepsilon_4\in(0,1)$ be as in Lemma \ref{lem:lemma4.6-TT16}. Then, there exist $c=c(D_1,T_1)>0$ and $\varepsilon_5=\varepsilon_5(D_1,T_1)\in(0,\varepsilon_4)$ such that for any $(y,t)\in\Omega\times[0,T_1]$, $s>t$, and $\varepsilon\in(0,\varepsilon_5)$, we have
\begin{align*}
\int_{0}^{t} & \left( \int_{\R^d} \left( \frac{\varepsilon|\nabla\phi^{\varepsilon}|^2}{2} - \frac{W(\phi^{\varepsilon})}{\varepsilon} \right)_+ \frac{\rho_{(y,s)}(x, \tau)}{2(s - \tau)} \, dx \right) \, d\tau \\
&\leq c \varepsilon^{\lambda' - \lambda} \left(1+|\log (\varepsilon)|+(\log s)_+\right)+\frac{2^{2d+1}\sqrt{\pi}}{\omega_d}C\sqrt{s}.
\end{align*}
Here, $C$ is the constant as in \eqref{eq:estimate-from-energy-dissipation}.
\end{lemma}
\begin{proof}
If $t\leq2\varepsilon^{2\lambda'}$, we use Proposition \ref{prop:upper-bound-discrepancy} and the fact that $\int_{\R^d}\rho_{(y,s)}(x,\tau)\,dx=\sqrt{4\pi(s-\tau)}$ to have
\begin{align*}
\int_{0}^{t} \left( \frac{1}{2(s - \tau)} \int_{\R^d} \left( \frac{\varepsilon|\nabla\phi^{\varepsilon}|^2}{2} - \frac{W(\phi^{\varepsilon})}{\varepsilon} \right)_+ \rho_{(y,s)}(x, \tau) \, dx \right) \, d\tau \
\leq \int_{0}^{t} \frac{C\varepsilon^{-\lambda}\sqrt{\pi}}{\sqrt{s - \tau}} \, d\tau \leq 2\sqrt{2\pi}C \varepsilon^{\lambda'-\lambda}.
\end{align*}
Here, $C>0$ is the constant as in Proposition \ref{prop:upper-bound-discrepancy}. In case, $s-2\varepsilon^{2\lambda'}\leq t<s$, we obtain with a similar argument that
\begin{align*}
\int_{s-2\varepsilon^{2\lambda'}}^{t} \left( \frac{1}{2(s - \tau)} \int_{\R^d} \left( \frac{\varepsilon|\nabla\phi^{\varepsilon}|^2}{2} - \frac{W(\phi^{\varepsilon})}{\varepsilon} \right)_+ \rho_{(y,s)}(x, \tau) \, dx \right) \, d\tau  \leq 2\sqrt{2\pi}C \varepsilon^{\lambda'-\lambda}.
\end{align*}

We now estimate the integral on $[2\varepsilon^{2\lambda'},t]$ with $t\leq s-2\varepsilon^{2\lambda'}$. The integral on $B_{\varepsilon^{\lambda'}}(y)$ is estimated with Proposition \ref{prop:upper-bound-discrepancy} and $s-t\geq2\varepsilon^{2\lambda'}$ as
\begin{align*}
\int_{2\varepsilon^{2\lambda'}}^{t} \frac{1}{2(s - \tau)} &\int_{B_{\varepsilon^{\lambda'}}(y)} \left( \frac{\varepsilon|\nabla\phi^{\varepsilon}|^2}{2} - \frac{W(\phi^{\varepsilon})}{\varepsilon} \right)_+ \rho_{(y,s)}(x, \tau) \, dx d\tau \\
&\leq \int_{2\varepsilon^{2\lambda'}}^{t} \frac{C\varepsilon^{-\lambda}\varepsilon^{d\lambda'}\omega_d}{2(s - \tau)^{\frac{d+1}{2}} (\sqrt{4\pi})^{d-1}} \, d\tau \leq \frac{C\varepsilon^{\lambda'-\lambda}\omega_d}{(\sqrt{8\pi})^{d-1}(d - 1)}.
\end{align*}

We estimate on $B_1(y)\setminus B_{\varepsilon^{\lambda'}(y)}$. By using Lemma \ref{lem:lemma4.6-TT16}, we have
\begin{align*}
&\int_{2\varepsilon^{2\lambda'}}^{t} \frac{1}{2(s - \tau)} \int_{B_1(y)\setminus B_{\varepsilon^{\lambda'}}(y)} \left( \frac{\varepsilon|\nabla\phi^{\varepsilon}|^2}{2} - \frac{W(\phi^{\varepsilon})}{\varepsilon} \right)_+ \rho_{(y,s)}(x,\tau) \, dx d\tau \\
&\leq \int_{2\varepsilon^{2\lambda'}}^{t} \frac{1}{2(s - \tau)^{\frac{d+1}{2}} (\sqrt{4\pi})^{d-1}}
\int_{0}^{1} \left( \int_{B_1(y)\cap\left\{x\, :\, e^{-\frac{|x-y|^2}{4(s-\tau)}} \geq l\right\} \setminus B_{\varepsilon^{\lambda'}}(y)} \left( \frac{\varepsilon|\nabla\phi^{\varepsilon}|^2}{2} - \frac{W(\phi^{\varepsilon})}{\varepsilon} \right)_+ \,dx\right) \, dld\tau \\
&\leq \int_{2\varepsilon^{2\lambda'}}^{t} \frac{1}{2(s - \tau)^{\frac{d+1}{2}} (\sqrt{4\pi})^{d-1}}
\int_{0}^{1} \left( \int_{B_1(y)\cap B_{(4(s-\tau)(-\log(l)))^{1/2}}(y)\setminus B_{\varepsilon^{\lambda'}}(y)} \left( \frac{\varepsilon|\nabla\phi^{\varepsilon}|^2}{2} - \frac{W(\phi^{\varepsilon})}{\varepsilon} \right)_+ \,dx\right) \, dld\tau\\
&\leq \int_{2\varepsilon^{2\lambda'}}^{t} \frac{1}{2(s - \tau)^{\frac{d+1}{2}} (\sqrt{4\pi})^{d-1}}
\int_{0}^{1} c_3\varepsilon^{\lambda'-\lambda}(4(s-\tau)(-\log(l)))^{\frac{d-1}{2}}\, dld\tau \\
&\leq \int_{2\varepsilon^{2\lambda'}}^t\,\frac{c_3\varepsilon^{\lambda'-\lambda}}{2\omega_{d-1}(s-\tau)}d\tau \\
&\leq \frac{c_3}{2\omega_{d-1}}\varepsilon^{\lambda'-\lambda}\left(\log(2)+2\lambda'|\log(\varepsilon)|+(\log s)_+\right).
\end{align*}

We now estimate on $\R^d\setminus B_1(y)$ as, using \eqref{eq:from-periodicity},
\begin{align*}
&\int_{2\varepsilon^{2\lambda'}}^{t} \frac{1}{2(s - \tau)} \int_{\R^d\setminus B_1(y)} \left( \frac{\varepsilon|\nabla\phi^{\varepsilon}|^2}{2} - \frac{W(\phi^{\varepsilon})}{\varepsilon} \right)_+ \rho_{(y,s)}(x,\tau) \, dx d\tau \\
&\leq \int_{2\varepsilon^{2\lambda'}}^{t} \frac{1}{2(s - \tau)^{\frac{d+1}{2}} (\sqrt{4\pi})^{d-1}}
\int_{0}^{1} \left( \int_{\left\{x\, :\, e^{-\frac{|x-y|^2}{4(s-\tau)}} \geq l\right\} \setminus B_{1}(y)} \left( \frac{\varepsilon|\nabla\phi^{\varepsilon}|^2}{2} + \frac{W(\phi^{\varepsilon})}{\varepsilon} \right) \,dx\right) \, dld\tau \\
&\leq \int_{2\varepsilon^{2\lambda'}}^{t} \frac{1}{2(s - \tau)^{\frac{d+1}{2}} (\sqrt{4\pi})^{d-1}}
\int_{0}^{1} \mu_t^{\varepsilon}\left(B_{(4(s-\tau)(-\log(l)))^{1/2}}(y)\setminus B_{1}(y)\right) \, dld\tau\\
&\leq \int_{2\varepsilon^{2\lambda'}}^{t} \frac{2^{2d}C}{\pi^{\frac{d-1}{2}}\sqrt{s-\tau}}\int_{0}^{1} (-\log(l))^{\frac{d}{2}}\, dld\tau \\
&\leq \int_{2\varepsilon^{2\lambda'}}^t\frac{2^{2d}\sqrt{\pi}C}{\omega_d\sqrt{s-\tau}}\,d\tau \\
&\leq \frac{2^{2d+1}\sqrt{\pi}}{\omega_d}C\sqrt{s}.
\end{align*}

All the above cases together, we finish the proof.
\end{proof}

We now prove the following $\varepsilon$-uniform upper bound of $D^{\varepsilon}(t)$ which is the goal of this subsection.

\begin{proposition}\label{prop:upper-bound-density-ratio}
For any $T>0$, there exist $D_T>0$ and $\varepsilon_2=\varepsilon_2(T)\in(0,\varepsilon_1)$ such that
\begin{align*}
    \sup_{\varepsilon\in(0,\varepsilon_2),\ t\in[0,T]}D^{\varepsilon}(t)\leq D_T.
\end{align*}
Here, $\varepsilon_1\in(0,1)$ is as in Proposition \ref{prop:L2-estimate}.
\end{proposition}
\begin{proof}

Fix $T>0$ and find $C_T>0$ as in Proposition \ref{prop:monotonicity-formula}. Let  $C>0$ be the constant as in \eqref{eq:estimate-from-energy-dissipation} depending only the data. Let
\begin{align*}
C_1:=e^{1/4}(4\pi)^{\frac{d-1}{2}}(\omega_{d-1})^{-1}C_T(D_0+1),\qquad D_1:=C_1+\frac{2^{2d+1}\sqrt{\pi}}{\omega_d}C\sqrt{T+16}+1.
\end{align*}
We assume for the contradiction that there exist $y\in\Omega,\ \widetilde{t}\in(0,T],\ r\in(0,4)$ such that $\frac{\mu_{\tilde{t}}^{\varepsilon}(B_r(y))}{\omega_{d-1}r^{d-1}} > C_1
$ and $\sup_{t\in[0,\widetilde{t}]}D^{\varepsilon}(t)\leq D_1$. Then, by applying \eqref{eq:monotonicity-formula-integrating-factor} with $t_1=0,\ t_2=\widetilde{t},\ s=\widetilde{t}+r^2$, we have
\begin{align}\label{eq:integration-monotonicity}
\int_{\mathbb{R}^{d}} \rho_{(y, s)}(x, t) d \mu_{t}^{\varepsilon}(x)\bigg|_{t=\widetilde{t}} 
\leq C_T\left(\int_{\mathbb{R}^{d}} \rho_{(y, s)}(x, t) d \mu_{t}^{\varepsilon}(x)\bigg|_{t=0}+\int_{0}^{\widetilde{t}}\int_{\mathbb{R}^d}\frac{\rho_{(y, s)}(x, t)}{2(s-t)}d\xi_{t}^{\varepsilon}\,dt\right).
\end{align}

We bound each term of \eqref{eq:integration-monotonicity}. First of all, by the choices of $y\in\Omega,\ \widetilde{t}\in(0,T],\ r\in(0,4)$, we have
\begin{align*}
\int_{\mathbb{R}^d} \rho_{(y,s)}(x,\tilde{t}) \, d\mu_{\tilde{t}}^\varepsilon &\ge\frac{1}{(4\pi r^2)^{\frac{d-1}{2}}} \int_{B_r(y)} e^{-\frac{|x-y|^2}{4r^2}} \, d\mu_{\tilde{t}}^\varepsilon(x)\\
&\ge \frac{e^{-\frac{1}{4}}}{(4\pi)^{\frac{d-1}{2}}r^{d-1}} \mu_{\tilde{t}}^\varepsilon(B_r(y))\\
&\ge \frac{e^{-\frac{1}{4}}\text{$\omega_{d-1}$}}{(4\pi)^{\frac{d-1}{2}}}C_1.
\end{align*}
Next,
\begin{align*}
\int_{\mathbb{R}^d} \rho_{(y,s)}(x,0) \, d\mu_0^\varepsilon(x) &= \frac{1}{(4\pi s)^{\frac{d-1}{2}}} \int_{\mathbb{R}^d} e^{-\frac{|x-y|^2}{4s}} \, d\mu_0^\varepsilon(x)\\
&= \frac{1}{(4\pi s)^{\frac{d-1}{2}}} \int_0^1 \mu_0^\varepsilon \left( \left\{ x \in \mathbb{R}^d : e^{-\frac{|x-y|^2}{4s}} \ge \ell \right\} \right) d\ell \\
&\le D_0 \cdot \frac{\omega_{d-1}}{\pi^{\frac{d-1}{2}}} \int_0^1 \left(-\log \ell \right)^{\frac{d-1}{2}} d\ell\\
&\le D_0.
\end{align*}
We now apply Lemma \ref{lem:lemma4.7-TT16} with the availability of $\sup_{t\in[0,\widetilde{t}]}D^{\varepsilon}(t)\leq D_1$, for $\varepsilon\in(0,1)$ small enough depending on $D_1,T$ (and thus only on $T$), we obtain, with $s\leq T+16$, that
\begin{align*}
\frac{e^{-1/4}\omega_{d-1}}{(4\pi)^{\frac{d-1}{2}}}C_1\leq C_T\left(D_0+c\varepsilon^{\lambda'-\lambda}\left(1+|\log(\varepsilon)|+\log(T+16)\right)+\frac{2^{2d+1}\sqrt{\pi}}{\omega_d}C\sqrt{T+16}\right)
\end{align*}
where $c>0$ is the constant depending on $D_1,T$ (and thus only on $T$) as in Lemma \ref{lem:lemma4.7-TT16}.


Therefore, the bounds of each term of \eqref{eq:integration-monotonicity} yield
\begin{align*}
\frac{e^{-1/4}\omega_{d-1}}{(4\pi)^{\frac{d-1}{2}}}C_1\leq C_T\left(D_0+\frac{2^{2d+1}\sqrt{\pi}}{\omega_d}C\sqrt{T+16}+\frac12\right),
\end{align*}
which contradicts the choice of $C_1$ for $\varepsilon\in(0,1)$ sufficiently small depending on $T>0$.
\end{proof}

For future uses, we leave the following corollary.

\begin{corollary}\label{cor:cor4.1-TT16-restate}
For any $T>0$, there exist $c'=c'(T)>0$ and $\epsilon=\epsilon(T)\in(0,1)$ satisfying the following:
\begin{itemize}
\item[(i)] For any $\varepsilon\in(0,\epsilon)$, $y\in\Omega$, $r\in[0,1]$, and $t\in[2\varepsilon^{2\lambda'},T]$, it holds that
\begin{align*}
\int_{0}^{r} \frac{1}{\tau^d} \left(\int_{B_{\tau}(y)} \left( \frac{\varepsilon |\nabla \phi^{\varepsilon}|^2}{2} - \frac{W(\phi^{\varepsilon})}{\varepsilon} \right)_{+} (x, t) \, dx\right)d\tau \leq c' \varepsilon^{\lambda' - \lambda} |\log \varepsilon|.
\end{align*}
\item[(ii)] For any $y\in\Omega$, $0\leq t_1\leq t_2<s$ with $t_2\leq T$, and $\varepsilon\in(0,\epsilon)$, we have
\begin{align*}
\int_{t_1}^{t_2} & \left(  \int_{\R^d} \left( \frac{\varepsilon|\nabla\phi^{\varepsilon}|^2}{2} - \frac{W(\phi^{\varepsilon})}{\varepsilon} \right)_+ \frac{\rho_{(y,s)}(x, \tau)}{2(s - \tau)} \, dx \right) \, d\tau \\
&\leq c' \varepsilon^{\lambda' - \lambda} \left(1+|\log (\varepsilon)|+(\log s)_+\right)+\frac{2^{2d+2}\sqrt{\pi}C}{\omega_d}\cdot\frac{t_2-t_1}{\sqrt{s-t_1}}.
\end{align*}
Here, $C$ is the constant as in \eqref{eq:estimate-from-energy-dissipation}.
\end{itemize}

\end{corollary}
\begin{proof}
(i) The integral is estimated separately on the range $\tau\in(0,\varepsilon^{\lambda'})$ with Proposition \ref{prop:upper-bound-discrepancy} and on the range $\tau\in(\varepsilon^{\lambda'},1)$ with Lemma \ref{lem:lemma4.7-TT16}, the latter of which is applied from the availability of Proposition \ref{prop:upper-bound-density-ratio}.

(ii) This follows from the same computations of the proof of Lemma \ref{lem:lemma4.7-TT16} with the availability of Proposition \ref{prop:upper-bound-density-ratio}.
\end{proof}

\section{Rectifiability and integrality of $\mu_t$}\label{sec:rectifiability-integrality}

\subsection{Convergence of $\mu^{\varepsilon}$}\label{subsec:convergence-mu}

\begin{lemma}\label{lem:BV-in-t-mu-ep-t}
For any $\varphi\in C_c^2(\Omega;\mathbb{R}_{\geq0})$ and $T>0$, there exists $C(\|\varphi\|_{C^2(\Omega)},T)>0$ such that
\begin{align*}
\int_{0}^{T} \left|\frac{d}{dt} \mu_{t}^{\varepsilon}(\varphi)\right| \mathrm{d}t \leq C(\|\varphi\|_{C^2(\Omega)},T)\quad \text { for any } \varepsilon \in\left(0, \varepsilon_{1}\right).
\end{align*}
Here, $\varepsilon_1\in(0,1)$ is as in Proposition \ref{prop:L2-estimate}.
\end{lemma}
\begin{proof}
The proof is similar to that of \cite[Lemma 3]{T23}, and we include the proof for the reader's sake. By integration by parts,
\begin{align*}
&\frac{d}{dt} \int_{\Omega} \varphi \left(\frac{\varepsilon|\nabla \phi^{\varepsilon}|^{2}}{2}+\frac{W(\phi^{\varepsilon})}{\varepsilon}\right) \mathrm{d}x \\
&=-\int_{\Omega} \varepsilon \varphi(\partial_{t} \phi^{\varepsilon})^{2} \mathrm{d}x + \int_{\Omega} \varphi g^{\varepsilon}\sqrt{2W(\phi^{\varepsilon})} \partial_{t}\phi^{\varepsilon} \mathrm{d}x - \int_{\Omega} \varepsilon(\nabla \varphi \cdot \nabla \phi^{\varepsilon}) \partial_{t}\phi^{\varepsilon} \mathrm{d}x.
\end{align*}
Using the fact that there exists a constant $C_d>0$ such that $\sup_{\{\varphi>0\}}\frac{|\nabla\varphi|^2}{\varphi}\leq2\|\varphi\|_{C^2(\Omega)}$ (which follows from Cauchy's mean value theorem), we have
\begin{align*}
&\left|\frac{d}{dt} \int_{\Omega} \varphi \left(\frac{\varepsilon|\nabla \phi^{\varepsilon}|^{2}}{2}+\frac{W(\phi^{\varepsilon})}{\varepsilon}\right) \mathrm{d}x\right|\\
&\leq\int_{\Omega} \varphi |g^{\varepsilon}|^{2} \frac{2W(\phi^{\varepsilon})}{\varepsilon} \mathrm{d}x+\int_{\Omega} \frac{|\nabla \varphi|^{2}}{\varphi} \varepsilon|\nabla \phi^{\varepsilon}|^{2} \mathrm{d}x+2 \int_{\Omega} \varepsilon \varphi(\partial_{t} \phi^{\varepsilon})^{2} \mathrm{d}x \\
&\leq C \left(1+|\lambda^{\varepsilon}|^{2}+\int_{\Omega} \varepsilon(\partial_{t} \phi^{\varepsilon})^{2} \mathrm{d}x\right).
\end{align*}
The proof follows from Propositions \ref{prop:energy-dissipation}, \ref{prop:L2-estimate}.
\end{proof}

Thanks to Lemma \ref{lem:BV-in-t-mu-ep-t}, we have the following propositions. As the proofs are identical to those of \cite[Propositions 9, 11]{T23}, respectively, we skip the proofs.

\begin{proposition}\label{prop:convergence-measure}
There exist a subsequence of $\varepsilon\to0$, still denoted by $\varepsilon\to0$ by abuse of notations, and a family of Radon measures $\{\mu_t\}_{t\geq0}$ such that
\begin{align*}
\mu_{t}^{\varepsilon} \to \mu_{t} \quad \text {as $\varepsilon\to0$ as Radon measures on } \Omega\text{ for all }t\geq0.
\end{align*}
Also, there exists a countable set $B\subset[0,\infty)$ such that $t\mapsto\mu_t(\Omega)$ is continuous on $[0,\infty)\setminus B$.
\end{proposition}

\begin{proposition}\label{prop:support-inclusion}
There exists a countable set $\widetilde{B}\subset (0,\infty)$ such that
\begin{align*}
\mathrm{supp}(\mu_t)\subset\{x\in\Omega\ |\ (x,t)\in\mathrm{supp}(\mu)\}
\end{align*}
for $t\in(0,\infty)\setminus\widetilde{B}$.
\end{proposition}

\subsection{Vanishing of the discrepancy measure}\label{subsec:vanishing-xi}
We prove that $|\xi|=0$ on $\R^d\times[0,\infty)$ in this subsection. The proofs of the lemmas and the theorem in this subsection follow from \cite[Subsection 4.2]{T23} with adaptations, which we record here for completeness.

\begin{lemma}\label{lem:auxiliary-lemma1}
For any $(x_0,t_0)\in\mathrm{supp}(\mu)$ with $t_0>0$, there exist a sequence $\{(x_j,t_j)\}_{j}$ and a subsequence $\{\varepsilon_j\}_j$ of $\varepsilon\to0$ such that 
\begin{align*}
\lim_{j\to\infty}(x_j,t_j)=(x_0,t_0)\quad\text{and}\quad\left|\phi^{\varepsilon_j}(x_j,t_j)\right|<\sqrt{\frac{2}{3}}\qquad\text{for all }\,j.
\end{align*}
\end{lemma}
\begin{proof}
The proof follows from arguments of the proof of \cite[Lemma 6.4]{NT25}. The only replacement that is necessary when following the proof of \cite[Lemma 6.4]{NT25} is
\begin{align*}
&\int_{\Omega} \phi^{2}\left(2\varepsilon_j|\nabla \phi^{\varepsilon_j}|^{2}+\frac{\left(W^{\prime}(\phi^{\varepsilon_j})\right)^{2}}{2 \varepsilon_j}\right) d x \\
&\leq C\left( \varepsilon_j^{3} \int_{\Omega}|\nabla \phi|^{2}|\nabla \phi^{\varepsilon_j}|^{2} d x+ \varepsilon_j^{1-\alpha}  \int_{\Omega} \phi^{2} d x\right)-\frac{d}{d t} \int_{\Omega} \varepsilon_j \phi^{2} W(\phi^{\varepsilon_j}) d x,
\end{align*}
where $C>0$ is a constant larger than the one in \eqref{eq:estimate-multiplier}. The rest of the proof of \cite[Lemma 6.4]{NT25} follows with this adaptation.
\end{proof}

\begin{lemma}\label{lem:auxiliary-lemma2}
There exist $\gamma,\eta_1,\eta_2\in(0,1)$ depending only on $T>0$ such that the following holds: If $t,s\leq\left[0,\frac{T}{2}\right)$ with $0< s-t<\eta_1$ and $x\in\Omega$ satisfy
\begin{align*}
\int_{\mathbb{R}^{d}} \rho_{\left(x, t_0\right)}(y, s)\, d \mu_{s}(y)<\eta_{2},
\end{align*}
then $\left(B_{\gamma r}(x) \times\left\{t_0\right\}\right) \cap \operatorname{supp} (\mu)=\emptyset$, where $r:=\sqrt{2(s-t)}$ and $t_0:=s+\frac{1}{2}r^2$.
\end{lemma}
\begin{proof}
Suppose for the contradiction that there exists $x_0\in B_{\gamma r}(x)$ such that $(x_0,t_0)\in\mathrm{supp}(\mu)$. Find a sequence $\{(x_j,t_j)\}_{j}$ and a subsequence $\{\varepsilon_j\}_j$ of $\varepsilon\to0$ as in Lemma \ref{lem:auxiliary-lemma1}.

For $y\in B_{\gamma'\varepsilon_j}(x_j)$ with $\gamma'>0$ sufficiently small, we have $W(\phi^{\varepsilon_j}(y,t_j))\geq W\left(\alpha\right)$ where $\alpha_0:=\frac{1}{2}\left(1+\sqrt{\frac{2}{3}}\right)$ since, thanks to Lemma \ref{lem:gradient-estimate}, it holds that
\begin{align*}
|\phi^{\varepsilon_j}(y,t_j) - \phi^{\varepsilon_j}(x_j,t_j)| \leq C\gamma'<\frac{1}{2}\left(1-\sqrt{\frac{2}{3}}\right)\qquad\text{for sufficiently small }\gamma'>0.
\end{align*}
Using the fact that $\inf_{y \in B_{\gamma_2 \varepsilon_j}(x_j)} \rho_{x_j, t_j + \varepsilon_j^2}(y, t_j) > c\varepsilon_j^{1-d}$ for a small constant $c>0$, we see that there exists a constant $\eta_3>0$ such that
\begin{align*}
\eta_3 \leq \int_{B_{\gamma_2 \varepsilon_j}(x_j)} \frac{W\left(\alpha_0\right)}{\varepsilon_j} \rho_{(x_j, t_j + \varepsilon_j^2)}(y, t_j) \, dy \leq \int_{B_{\gamma_2 \varepsilon_j}(x_j)} \frac{W(\phi^{\varepsilon_j}(y, t_j))}{\varepsilon_j} \rho_{(x_j, t_j + \varepsilon_j^2)}(y, t_j) \, dy.
\end{align*}
By Proposition \ref{prop:monotonicity-formula},
\begin{align*}
\eta_{3} &\leq \int_{\mathbb{R}^{d}} \rho_{\left(x_{j}, t_{j}+\varepsilon_j^{2}\right)}\left(y, t_{j}\right) d \mu_{t_{j}}^{\varepsilon_j}(y)\\
&\leq C_T \left(\int_{\mathbb{R}^{d}} \rho_{\left(x_{j}, t_{j}+\varepsilon_j^{2}\right)}(y, s) d \mu_{s}^{\varepsilon_j}(y)+\int_s^{t_j}\int_{\mathbb{R}^d}\frac{\rho_{(x_j, t_j+\varepsilon_j^2)}(y, \tau)}{2(t_j+\varepsilon_j^2-\tau)}d\xi_{\tau}^{\varepsilon_j}(y)d\tau\right).
\end{align*}
Therefore, by Corollary \ref{cor:cor4.1-TT16-restate}(ii) with $(s,t_j,t_j+\varepsilon_j^2)$ in place of $(t_1,t_2,s)$, and by $0<s-t<\eta_1$, we recover
\begin{align*}
3 \eta_{2} \leq \int_{\mathbb{R}^{d}} \rho_{\left(x_{j}, t_{j}+\varepsilon_j^{2}\right)}(y, s) d \mu_{s}^{\varepsilon_j}(y) c + c'\varepsilon_j^{\lambda'-\lambda}(1+|\log(\varepsilon_j)|+(\log (t_j+\varepsilon_j^2))_+) + C\sqrt{\eta_1+|t_j-t_0|}
\end{align*}
for some $\eta_2=\eta_2(\eta_3,T)\in(0,1)$. Here, $C>0$ is a constant that depends only on the data and $c'=c'(T)>0$ is as in Corollary \ref{cor:cor4.1-TT16-restate}. Sending $j\to\infty$ gives
\begin{align*}
3 \eta_{2} \leq \int_{\mathbb{R}^{d}} \rho_{\left(x_{0}, t_{0}\right)}(y, s) d \mu_{s}(y)+C\sqrt{\eta_1}
\end{align*}
Thanks to \cite[Lemma 13]{T23} and taking $\eta_1=C^{-2}\eta_2^2$ that depends only $T$, we see that there exists $\gamma=\gamma(\eta_2,D_T)\in(0,1)$ such that
\begin{align*}
\eta_{2} \leq \int_{\mathbb{R}^{d}} \rho_{\left(x, t_{0}\right)}(y, s) d \mu_{s}(y)\qquad\text{for }x\in B_{\gamma r}(x_0).
\end{align*}
This contradicts the hypothesis of the lemma.
\end{proof}

As the following two lemmas are proved with the identical arguments of the proofs of \cite[Lemmas 7, 8]{T23}, we state them without proofs.

\begin{lemma}\label{lem:auxiliary-lemma3}
There exists a constant $C'_T>0$ depending on $T>0$ such that
\begin{align*}
\mathcal{H}^{d-1}(\operatorname{supp} \mu_{t} \cap U) \leq C'_T \liminf _{r \to 0} \mu_{t-r^{2}}(U)
\end{align*}
for any $t\in(0,T)\subset\widetilde{B}$ and any open set $U\subset\Omega$, where $\widetilde{B}$ is the countable subset of $(0,\infty)$ as in Proposition \ref{prop:support-inclusion}.
\end{lemma}

\begin{lemma}\label{lem:auxiliary-lemma4}
Let $T\geq1$ and let $\eta_2\in(0,1)$ be a number as in Lemma \ref{lem:auxiliary-lemma2}. Let
\begin{align*}
Z_{T}:=\left\{(x, t) \in \operatorname{supp} (\mu) \mid 0 \leq t \leq T / 2,\, \limsup _{s \to t^+} \int_{\mathbb{R}^{d}} \rho_{(y, s)}(x, t) d \mu_{s}(y) \leq \eta_{2} / 2\right\}.
\end{align*}
Then, it holds that $\mu(Z_T)=0$.
\end{lemma}




\begin{theorem}\label{thm:vanishing-xi}
It holds that $|\xi|=0$ on $\R^d\times(0,\infty)$.
\end{theorem}
\begin{proof}
Fix $T>0$ and find $C_T,D_T>0$ as in Propositions \ref{prop:monotonicity-formula}, \ref{prop:upper-bound-density-ratio}, respectively (as well as $\min\{\varepsilon_1,\varepsilon_2(T)\}\in(0,1)$). Due to \eqref{eq:monotonicity-formula-integrating-factor}, we can write by using the fact that $|\xi^{\varepsilon}_t|+\xi^{\varepsilon}_t\leq2(\xi^{\varepsilon}_t)_+$, for $0<t<s<T$,
\begin{align*}
\int_{0}^{t}\int_{\mathbb{R}^d}\frac{\rho_{(y, s)}(x, t)}{2(s-t)}d\left|\xi_{t}^{\varepsilon}\right|\,dt\leq \int_{\mathbb{R}^{d}} \rho_{(y, s)}(x, t) d \mu_{t}^{\varepsilon}(x)\bigg|_{t=0}+2\int_{0}^{t}\int_{\mathbb{R}^d}\frac{\rho_{(y, s)}(x, t)}{2(s-t)}d\left(\xi_{t}^{\varepsilon}\right)_+\,dt.
\end{align*}
Letting $t\to s^-$, $\varepsilon\to0$, and integrating over $d\mu_sds$ over $\R^d\times(0,T)$ in order, we obtain, by Lemma \ref{lem:lemma4.7-TT16},
\begin{align*}
\int_{\R^d\times(0,T)}  \int_{\R^d \times(0, s)}\frac{\rho_{(y, s)}(x, t)}{2(s-t)} \,d|\xi|(x, t)\,d\mu_{s}(y)ds<\infty.
\end{align*}
Thus, by Fubini's theorem,
\begin{align}\label{eq:after-fubini}
\int_{\R^d \times(t, T)} \frac{\rho_{(y, s)}(x, t)}{s-t} d \mu_{s}(y) d s<\infty \quad \text { for }|\xi|\text{-a.e. }(x, t) \in \R^d \times(0, T).
\end{align}

\medskip

We now claim that
\begin{align}\label{eq:a-is-zero}
a(x, t) := \limsup_{s \to t^+} \int_{\R^d} \rho_{(y,s)}(x, t) \, d\mu_s(y) = 0 \qquad \text{for } |\xi|\text{-a.e. } (x, t) \in \R^d \times (0, T).
\end{align}
Put $\beta:=\log (s-t)$ and
\begin{align*}
h(s) := \int_{\R^d} \rho_{(y,s)}(x, t) \, d\mu_s(y).
\end{align*}
For $(x, t) \in \R^d \times(0, T)$ satisfying \eqref{eq:after-fubini}, we have
\begin{align*}
\int_{-\infty}^{\log(T-t)} h(t + e^{\beta}) \, d\beta < \infty,
\end{align*}
and therefore, for any $\theta\in(0,1)$, we can find a sequence $\{\beta_i\}_{i=1}^{\infty}$ such that $\beta_i \to -\infty\text{ as }i\to\infty$, $0 < \beta_i - \beta_{i+1} \leq \theta$, and $h(t + e^{\beta_i}) \leq \theta$. For $\beta\in(-\infty,\beta_1)$, choose $i$ such that $\beta\in[\beta_i,\beta_{i-1})$. We note that for $M>0$ and $r:=\sqrt{2(2e^{\beta}-e^{\beta_i})}$, it holds that
\begin{align}\label{eq:on-Mr}
\sup_{y \in B_{Mr}(x)} \frac{\rho_{(y, t+2e^{\beta}-e^{\beta_i})}(x, t)}{\rho_{(y, t+e^{\beta_i})}(x, t)} \leq \left(\frac{e^{\beta_i}}{2e^{\beta}-e^{\beta_i}}\right)^{\frac{d-1}{2}}e^{M^2(e^{\beta-\beta_i}-1)} \leq e^{M^2(e^{\theta}-1)}.
\end{align}
Applying \eqref{eq:monotonicity-formula} and Corollary \ref{cor:cor4.1-TT16-restate}(ii), we obtain
\begin{align*}
h(t + e^{\beta}) &= \int_{\R^d} \rho_{(y,t+e^{\beta})}(x, t) \, d\mu_{t+e^{\beta}}(y) = \int_{\R^d} \rho_{(y,t+2e^{\beta})}(x, t + e^{\beta}) \, d\mu_{t+e^{\beta}}(y) \\
&\leq C_T\left( \int_{\R^d} \rho_{(y,t+2e^{\beta})}(x, t + e^{\beta_i}) \, d\mu_{t+e^{\beta_i}}(y)+\int_{t+e^{\beta_i}}^{t+e^{\beta}}\int_{\R^d}\frac{\rho_{(y,t+2e^{\beta})}(x,\tau)}{2(t+2e^{\beta}-\tau)}\,d\xi^{\varepsilon}_{\tau}(x)d\tau\right)\\
&\leq C_T\left(\int_{\R^d} \rho_{(y,t+2e^{\beta})}(x, t + e^{\beta_i}) \, d\mu_{t+e^{\beta_i}}(y)\right.\\
&\qquad\qquad\left.+c'\varepsilon^{\lambda'-\lambda}\left(1+|\log(\varepsilon)|+\log(t+2e^{\beta})_+\right)+\frac{2^{2d+2}\sqrt{\pi}C}{\omega_d}\cdot\frac{e^{\beta}-e^{\beta_i}}{\sqrt{2e^{\beta}-e^{\beta_i}}}\right).
\end{align*}
We further have that by \eqref{eq:on-Mr}, Proposition \ref{prop:upper-bound-density-ratio}, and \cite[Lemma 13]{T23},
\begin{align*}
\int_{\R^d} \rho_{(y,t+2e^{\beta})}(x, t + e^{\beta_i}) \, d\mu_{t+e^{\beta_i}}(y) &= \int_{\R^d} \rho_{(y,t+2e^{\beta}-e^{\beta_i})}(x, t) \, d\mu_{t+e^{\beta_i}}(y)\\
&\leq \int_{B_{Mr}(x)} \rho_{(y,t+2e^{\beta}-e^{\beta_i})}(x, t) \, d\mu_{t+e^{\beta_i}}(y)+2^{d-1}e^{-\frac{3M^2}{8}}D_T\\
&\leq e^{M^2(e^{\theta}-1)}\int_{B_{Mr}(x)} \rho_{(y,t+e^{\beta_i})}(x, t) \, d\mu_{t+e^{\beta_i}}(y)+2^{d-1}e^{-\frac{3M^2}{8}}D_T\\
&\leq e^{M^2(e^{\theta}-1)}\theta+2^{d-1}e^{-\frac{3M^2}{8}}D_T.
\end{align*}
Also, we can check that $\frac{e^{\beta}-e^{\beta_i}}{\sqrt{2e^{\beta}-e^{\beta_i}}}\leq(1+e^{\frac{1}{2}\theta})\sqrt{T-t}\cdot\theta\leq4\sqrt{T}\cdot\theta$. All together, we obtain
\begin{align*}
h(t + e^{\beta}) &\leq C_T\left(e^{M^2(e^{\theta}-1)}\theta+2^{d-1}e^{-\frac{3M^2}{8}}D_T \right.\\
&\left.\qquad\qquad\qquad+ c'\varepsilon^{\lambda'-\lambda}\left(1+|\log(\varepsilon)|+\log(2T)_+\right)+\frac{2^{2d+4}\sqrt{\pi}C\sqrt{T}}{\omega_d}\theta\right).
\end{align*}
Taking the limits $\beta\to-\infty$, $\theta\to0$, $M\to\infty$, and $\varepsilon\to0$ in turn, we obtain $\limsup_{\beta\to-\infty}h(t+e^{\beta})=0$, which proves \eqref{eq:a-is-zero}.

\medskip

Set
\begin{align*}
\begin{cases}
A:=\{(x, t) \in \Omega \times (0, T)\,:\,a(x, t) = 0\}\quad\text{and}\\
B:=\{(x, t) \in \Omega \times (0, T)\,:\,a(x, t) >0\}.
\end{cases}
\end{align*}
Then, $\R^d\times(0,T)=A\cup B$ and $|\xi|(B)=0$ by \eqref{eq:a-is-zero}. Furthermore, Lemma \ref{lem:auxiliary-lemma4} implies that $\mu(A)=0$, which then implies $|\xi|(A)=0$. Since $T>0$ was arbitrary, we conclude that $|\xi|(\R^d\times(0,\infty))=0$.
\end{proof}

\subsection{Rectifiability and integrality}\label{subsec:rectifiability-integrality}

We first state the rectifiability of $\mu_t$ as a consequence of Theorem \ref{thm:vanishing-xi}. Reference \cite[Subsection 4.3]{T23} treats the rectifiability, and we include the details for completeness.

\medskip

Define the varifold $V^{\varepsilon}_t$ by
\begin{align*}
V_t^{\varepsilon}(\phi) := \int_{\Omega \cap \{|\nabla\phi^{\varepsilon}(x,t)| \neq 0\}} \phi \left( x, I - \frac{\nabla\phi^{\varepsilon}(x,t)}{|\nabla\phi^{\varepsilon}(x,t)|} \otimes \frac{\nabla\phi^{\varepsilon}(x,t)}{|\nabla\phi^{\varepsilon}(x,t)|} \right) \mathrm{d}\mu_t^{\varepsilon}(x)
\end{align*}
for $\phi\in C_c(\Omega\times\mathbb{G}(d,d-1))$. By the definition of the first variation of $V_t^{\varepsilon}$, we have
\begin{align*}
\delta V_t^{\varepsilon}(\zeta) &= \int_{\Omega\times\mathbb{G}(d,d-1)} \nabla\zeta(x) \cdot S \, \mathrm{d}V_t^{\varepsilon}(x, S) \\
&= \int_{\Omega\cap\{|\nabla\phi^{\varepsilon}(x,t)| \neq 0\}} \nabla\zeta(x) \cdot \left( I - \frac{\nabla\phi^{\varepsilon}(x,t)}{|\nabla\phi^{\varepsilon}(x,t)|} \otimes \frac{\nabla\phi^{\varepsilon}(x,t)}{|\nabla\phi^{\varepsilon}(x,t)|} \right) \mathrm{d}\mu_t^{\varepsilon}(x)
\end{align*}
for $\zeta\in C_c^1(\Omega;\mathbb{R}^d)$. By integration by parts,
\begin{align*}
\delta V_t^{\varepsilon}(\zeta) &= \frac{1}{\sigma}\int_{\Omega} (\zeta \cdot \nabla \phi^{\varepsilon}) \left( \varepsilon \Delta \phi^{\varepsilon} - \frac{W'(\phi^{\varepsilon})}{\varepsilon} \right) \mathrm{d}x + \frac{1}{\sigma}\int_{\Omega \cap \{|\nabla \phi^{\varepsilon}(x,t)| = 0\}} \mathrm{div} (\zeta)\xi_t^{\varepsilon} \, \mathrm{d}x \\
&\quad + \frac{1}{\sigma}\int_{\Omega \cap \{|\nabla \phi^{\varepsilon}(x,t)| \neq 0\}} \nabla \zeta \cdot \left( \frac{\nabla \phi^{\varepsilon}(x,t)}{|\nabla \phi^{\varepsilon}(x,t)|} \otimes \frac{\nabla \phi^{\varepsilon}(x,t)}{|\nabla \phi^{\varepsilon}(x,t)|} \right) \xi_t^{\varepsilon} \, \mathrm{d}x.
\end{align*}
We note that the second and third terms vanish as $\varepsilon\to0$ for a.e. $t\geq0$ by Theorem \ref{thm:vanishing-xi}. We moreover have, by \eqref{eq:phi-epsilon}, \eqref{eq:estimate-from-energy-dissipation}, and Proposition \ref{prop:L2-estimate}, that there exists a constant $C_T>0$ such that
\begin{align}\label{eq:L2-approximate-mean-curvature}
\sup_{\varepsilon\in(0,1)} \int_{0}^{T} \int_{\Omega} \varepsilon \left( \Delta\phi^{\varepsilon} - \frac{W'(\phi^{\varepsilon})}{\varepsilon^2} \right)^2 \mathrm{d}x \mathrm{d}t \le C_T.
\end{align}
Fatou's lemma implies that for a.e. $t\geq0$, there exists a subsequence of $\varepsilon\to0$, still denoted by $\varepsilon\to0$ (which may depend on the choice $t\geq0$), such that
\begin{align*}
\lim_{\varepsilon\to0}\int_{\Omega} \varepsilon \left( \Delta\phi^{\varepsilon} - \frac{W'(\phi^{\varepsilon})}{\varepsilon^2} \right)^2 \mathrm{d}x < +\infty.
\end{align*}
Therefore, for $\zeta\in C_c^1(\Omega;\mathbb{R}^d)$ with $\|\zeta\|_{L^{\infty}(\Omega)}\leq1$, it holds, by \eqref{eq:estimate-from-energy-dissipation}, that
\begin{align*}
& \liminf_{\varepsilon\to0} |\delta V_t^{\varepsilon}(\zeta)| \le \liminf_{\varepsilon\to0}\frac{1}{\sigma} \left( \int_{\Omega} \varepsilon |\nabla \phi^{\varepsilon}|^2 \, \mathrm{d}x \right)^{\frac{1}{2}} \left( \int_{\Omega} \varepsilon \left( \Delta\phi^{\varepsilon} - \frac{W'(\phi^{\varepsilon})}{\varepsilon^2} \right)^2 \, \mathrm{d}x \right)^{\frac{1}{2}} <+\infty.
\end{align*}
We hence see that for each $t\in[0,\infty)$ as in Proposition \ref{prop:support-inclusion}, there exist a varifold $V_t$ and a subsequence of $\varepsilon\to0$, still denoted by $\varepsilon\to0$, such that $V_t^{\varepsilon}\to V_t$ as $\varepsilon\to0$ as Radon measures (so that $\mu_t=\|V_t\|$ by Proposition \ref{prop:convergence-measure}) and $\delta V_t$ is a Radon measure. Moreover, Proposition \ref{prop:support-inclusion} and Lemma \ref{lem:auxiliary-lemma3} imply that
\begin{align*}
V_t = V_t \lfloor_{\{x \in \Omega | \limsup_{r \to0} r^{-(d-1)} \|V_t\|(B_r(x)) > 0\} \times \mathbb{G}(d, d-1)}.
\end{align*}
The Allard rectifiability theorem now gives the following:
\begin{theorem}\label{thm:rectifiability}
For a.e. $t\geq0$, $\mu_t$ is rectifiable and has a generalized mean curvature vector $h(\cdot,t)$. Moreover, for any $\zeta\in C_c(\Omega;\mathbb{R}^d)$, $\varphi\in C_c(\Omega;\mathbb{R}_{\geq0})$, and a.e. $t\geq0$, we have
\begin{align*}
\delta V_{t}(\zeta)=-\int_{\Omega} \zeta \cdot h(\cdot, t) d \mu_{t}=\lim _{\varepsilon \rightarrow 0}\text{$\frac{1}{\sigma}$} \int_{\Omega}\left(\zeta \cdot \nabla \phi^{\varepsilon}\right)\left(\varepsilon \Delta \phi^{\varepsilon}-\frac{W^{\prime}\left(\phi^{\varepsilon}\right)}{\varepsilon}\right) d x,
\end{align*}
and
\begin{align*}
\int_{\Omega} \varphi|h|^{2} d \mu_{t} \leq \frac{1}{\sigma} \liminf _{\varepsilon \rightarrow 0} \int_{\Omega} \varphi\varepsilon^{-1}\left(\varepsilon\Delta \phi^{\varepsilon}-\frac{W^{\prime}\left(\phi^{\varepsilon}\right)}{\varepsilon}\right)^{2} d x<\infty.
\end{align*}
\end{theorem}

We now turn our attention to the integrality of $\mu_t$. Among the references \cite{HT00,T03,TT16,T23} for the integrality, we mainly follow the presentation of \cite[Subsection 4.4]{T23} where the forcing term is the Lagrange multiplier $\lambda^{\varepsilon}(t)$ that is independent of the spatial variable. As our forcing term $g^{\varepsilon}(x,t)$ is spatially dependent, we need to deal with any issues that can arise from the spatial dependence.



\medskip

We state the integrality result. We leave several statements in \nameref{sec:appendix-A} that are necessary for the proof.

\begin{theorem}\label{thm:integrality}
For a.e. $t>0$, there exist a countably $(d-1)$-rectifiable set $M_t$ and a locally $\mathcal{H}^{d-1}\lfloor_{M_t}$-integrable function $\theta_t:M_t\to\mathbb{N}$ such that we have
\begin{align*}
\mu_t=\theta_t\mathcal{H}^{d-1}\lfloor_{M_t}.
\end{align*}
\end{theorem}

\begin{proof}[Proof of Theorem \ref{thm:integrality}]
Let $h^{\varepsilon}:=\Delta \phi^{\varepsilon}-\frac{W'(\phi^{\varepsilon})}{\varepsilon^2}$. Then, for a.e. $t_0>0$, there exists a subsequence $\varepsilon_i\to0$ of $\varepsilon\to0$ such that $V_{t_0}^{\varepsilon_i}\to V_{t_0}$ as $i\to\infty$ as Radon measures,
\begin{align*}
\lim_{i\to\infty}\int_{\Omega}|\xi_{\varepsilon_i}(x,t_0)|dx=0,
\end{align*}
and
\begin{align*}
c_{H}(t_0) := \sup_{i \in \mathbb{N}} \int_{\Omega} \varepsilon_{i} |h^{\varepsilon_{i}} \nabla \phi^{\varepsilon_{i}}|(x, t_0) dx < +\infty.
\end{align*}
Here, $(V_{t_0}^{\varepsilon_i})_i$ is a sequence of varifolds as constructed before the statement of Theorem \ref{thm:rectifiability}. Let $\mu_{t_0}:=\|V_{t_0}\|$. We claim the integrality of $\mu_{t_0}$.

\medskip

Set
\begin{align*}
A_{i,m} := \left\{ x \in \Omega\,:\,\int_{B_r(x)} \varepsilon_i |h^{\varepsilon_i} \nabla \phi^{\varepsilon_i}|(x, t_0) dx \le m \mu_{t_0}^{\varepsilon_i}(B_r(x)) \text{ for any } r \in \left(0, \frac{1}{2}\right) \right\}
\end{align*}
and
\begin{align*}
A_m := \{ x \in \Omega\,:\,\text{there exists } x_i \in A_{i,m} \text{ for any } i \in \mathbb{N} \text{ such that } x_i \to x \}
\end{align*}
for $i,m\in\mathbb{N}$, and set $A:=\cup_{m\in\mathbb{N}}A_m$. Then, by the Besicovitch covering theorem, we have
\begin{align*}
\mu_{t_0}^{\varepsilon_i}(\Omega \setminus A_{i,m}) \leq \frac{C_dc_H(t_0)}{m},
\end{align*}
where $C_d>0$ is a constant that depends only on the dimension $d$. We claim that
\begin{align*}
\mu_{t_0}(\Omega\setminus A)=0.
\end{align*}
Otherwise, there exists a compact subset $K\subset\Omega\setminus A$ with $\mu_{t_0}(K)>0$. We note that $K\subset\Omega\setminus A_m$ for all $m\in\N$. By the definition of $A_m$, we see that for each $x\in K$ and $m\in\N$, there is a ball $B_r(x)$ centered at $x$ such that $B_r(x)\cap A_{i,m}=\emptyset$ for sufficiently large $i$. Therefore, by the compactness of $K$, for each $m\in\N$, there exist an open neighborhood $O_m$ of $K$ and $i_m\in\N$ such that $O_m\cap A_{i,m}=\emptyset$ for all $i\geq i_m$. Let $\phi_m\in C_c(O_m;[0,1])$ denote a smooth function such that $\phi_m=1$ on $K$. Then, for each $m\in\N$,
\begin{align*}
\mu_{t_0}(K) \leq \int_{\Omega} \phi_m \, \mathrm{d}\mu_{t_0} = \lim_{i\to\infty} \int_{\Omega} \phi_m \, \mathrm{d}\mu_{t_0}^{\varepsilon_i} &= \lim_{i\to\infty} \int_{\Omega \setminus A_{i,m}} \phi_m \, \mathrm{d}\mu_{t_0}^{\varepsilon_i} \\
&\leq \liminf_{i\to\infty} \mu_{t_0}^{\varepsilon_i}(\Omega \setminus A_{i,m})\leq\frac{C_dc_H(t_0)}{m}.
\end{align*}
As $m\in\N$ was arbitrary, we may take the limit as $m\to\infty$ to yield a contradiction to $\mu_{t_0}(K)>0$.

\medskip

Therefore, for $\mu_{t_0}$-a.e. $x\in \mathrm{supp}(\mu_{t_0})$, $x\in A_m$ for some $m\in\mathbb{N}$ and $\mu_{t_0}$ has an approximate tangent space $P$ by the rectifiability of $\mu_{t_0}$ from Theorem \ref{thm:rectifiability}. Fix such $x$ and let $\theta := \lim_{r \to 0} \frac{\mu_{t_0}(B_r(x))}{\omega_{d-1} r^{d-1}}$. Once $\theta\in\mathbb{N}$ is proved, we are done.

\medskip

By translations and rotations, we may assume without loss of generality that 
$x=0$ and $P=\{(x_1,\cdots,x_d)\in\mathbb{R}^d\,:\,x_d=0\}$. We perform rescaling as follows. Let $\Phi_r(x):=\frac{x}{r}$ for $r>0$. Then, for a sequence $r_i\to0$, we have
\begin{align*}
\lim_{i \to \infty} (\Phi_{r_i})_{\#} V^{\varepsilon_i}_{t_0} = \theta \mathcal{H}^{d-1}\lfloor_P
,
\end{align*}
where $(\cdot)_\#$ denotes the push-forward of a measure. By the definition of $A_m$, we can find a sequence $(x_i)_i$ with $x_i\in A_{i,m}$ and $x_i\to0$ as $i\to\infty$. We may assume without loss of generality, by passing to a subsequence if necessary, 
\begin{align}\label{eq:rescaling-rule}
\lim_{i \to \infty} \frac{x_i}{r_i} = 0, \quad \lim_{i \to \infty} \frac{\varepsilon_i}{r_i} = 0,\quad\text{and}\qquad\lim_{i\to\infty}\frac{1}{r_i^{d-1}} \int_{\Omega} |\xi_{\varepsilon_i}(x, t_0)| dx=0.
\end{align}

We further set $\tilde{x}=\frac{x}{r_i}$, $\tilde{t}=\frac{t-t_0}{r_i^2}$, $\tilde{\varepsilon}_i=\frac{\varepsilon_i}{r_i}$. We set $\tilde{\phi}^{\tilde{\varepsilon}_i}(\tilde{x},\tilde{t})=\phi^{\varepsilon_i}(x,t)$, $\tilde{g}^{\tilde{\varepsilon}_i}(\tilde{x},\tilde{t})=r_ig^{\varepsilon_i}(x,t)$, $\tilde{\xi}_{\tilde{\varepsilon}_i}(\tilde{x},\tilde{t})=r_i\xi_{\varepsilon_i}(x,t)$, and $\tilde{h}^{\tilde{\varepsilon}_i}(\tilde{x},\tilde{t})=r_ih^{\varepsilon_i}(x,t)$. Then, $\tilde{\phi}^{\tilde{\varepsilon}_i}$ solves
\begin{align*}
\tilde{\varepsilon_i}\partial_{\tilde{t}}\tilde{\phi}^{\tilde{\varepsilon_i}} = \tilde{\varepsilon_i}\Delta_{\tilde{x}}\tilde{\phi}^{\tilde{\varepsilon_i}} - \frac{W'(\tilde{\phi}^{\tilde{\varepsilon_i}})}{\tilde{\varepsilon_i}} + \widetilde{g}^{\tilde{\varepsilon_i}}\sqrt{2W(\tilde{\phi}^{\tilde{\varepsilon_i}})}.
\end{align*}
We also check, by \eqref{eq:rescaling-rule}, that
\begin{align}\label{eq:before-drop-1}
\int_{B_3(0)} |\tilde{\xi}_{\tilde{\varepsilon}_i}(\tilde{x}, 0)| d\tilde{x} = \frac{1}{r_i^{d-1}} \int_{B_{3r_i}(0)} |\xi_{\varepsilon_i}(x, t_0)| dx\to0\quad\text{as}\quad i\to\infty,
\end{align}
and, by using Proposition \ref{prop:upper-bound-density-ratio}, the definition of $x_i\in A_{i,m}$, that
\begin{align}\label{eq:before-drop-2}
\tilde{\varepsilon}_i \int_{B_3(0)} |\tilde{h}^{\tilde{\varepsilon}_i} \nabla_{\tilde{x}} \tilde{\phi}^{\tilde{\varepsilon}_i}(\tilde{x}, 0)| d\tilde{x} &= \frac{\varepsilon_i}{r_i^{d-2}} \int_{B_{3r_i}(0)} |h^{\varepsilon_i} \nabla \phi^{\varepsilon_i}(x, t_0)| dx\notag \\
&\le \frac{m}{r_i^{d-2}} \mu_{t_0}^{\varepsilon_i}(B_{4r_i}(x_i)) \le C_{d,t_0} m r_i \to 0 \quad\text{as}\quad i \to \infty.
\end{align}
Also, by Corollary \ref{cor:cor4.1-TT16-restate}(i), we have that
\begin{align}\label{eq:before-drop-3}
\int_0^1\frac{1}{\tau^d}\left(\int_{B_{\tau}(\tilde{x})}\left(\widetilde{\xi}_{\widetilde{\varepsilon}_i}(\widetilde{y},0)\right)^+d\widetilde{y}\right)d\tau&=\int_0^1\frac{1}{\tau^dr_i^{d-1}}\left(\int_{B_{\tau r_i}(x)}\left(\xi_{\varepsilon_i}(y,t_0)\right)^+dy\right)d\tau\notag\\
&=\int_0^{r_i}\frac{1}{\sigma^d}\left(\int_{B_{\sigma}(x)}\left(\xi_{\varepsilon_i}(y,t_0)\right)^+dy\right)d\sigma\notag\\
&\leq c'\varepsilon_i^{\lambda'-\lambda}|\log(\varepsilon_i)|\to0\quad\text{as}\quad i \to \infty.
\end{align}
We now prove that
\begin{align}\label{eq:before-drop-4}
\lim_{i \to \infty} \int_{B_3(0)} (1 - (\nu_d^i(\widetilde{x},0))^2)\tilde{\varepsilon}_i |\nabla_{\tilde{x}}\tilde{\phi}^{\tilde{\varepsilon}_i}(\widetilde{x},0)|^2 d\tilde{x} = 0,
\end{align}
where $\nu^i_d$ refers to the $\tilde{x}_d$-component of the vector $\frac{\nabla_{\tilde{x}}\tilde{\phi}^{\tilde{\varepsilon}_i}}{|\nabla_{\tilde{x}}\tilde{\phi}^{\tilde{\varepsilon}_i}|}$ when $|\nabla_{\tilde{x}}\tilde{\phi}^{\tilde{\varepsilon}_i}|\neq0$ and $0$ otherwise. Indeed, we define the function $\psi\,:\,\mathbb{G}(d,d-1)\to\R$ by $\psi(S)=1-\nu_d^2$, where $\nu=(\nu_1,\cdots,\nu_d)\in\mathbb{S}^{d-1}$ is one of the unit normal vectors to $S$. Then, $\psi(P)=0$, and for $\phi\in C_c(\R^d)$, we see that $\phi\psi\in C_c(\R^d\times\mathbb{G}(d,d-1))$. Therefore,
\begin{align*}
\lim_{i\to\infty} \tilde{V}_0^{\tilde{\varepsilon}_i}(\phi\psi) &= \int \phi(\tilde{x})(1 - (\nu_d^i)^2) \, d\|\tilde{V}_0^{\tilde{\varepsilon}_i}\|(\tilde{x}) \\
&= \lim_{i\to\infty} (\Phi_{r_i})_{\#} V_{t_0}^{\varepsilon_i} (\phi\psi) = \theta \int_P \phi(\tilde{x})\psi(P) \, d\mathcal{H}^{d-1}(\tilde{x}) = 0,
\end{align*}
which proves \eqref{eq:before-drop-4}. Here, we used the facts that $\lim_{i \to \infty} (\Phi_{r_i})_{\#} V^{\varepsilon_i}_{t_0} = \theta \mathcal{H}^{d-1}\lfloor_P$ and that $\psi(P)=0$.

\medskip

From now on, we suppress the $\widetilde{\cdot}$ symbol and the time variable for convenience. We now choose the unique positive integer $N\in\mathbb{N}$ satisfying $N-1\leq\theta< N.$

\medskip

Let $s\in(0,1)$. By Proposition \ref{prop:appendix-3} and Theorem \ref{thm:vanishing-xi}, there exists $b>0$ such that for sufficiently large $i$, it holds that
\begin{align}\label{eq:after-drop-0}
\int_{\{x \in B_3(0) | |\phi^{\varepsilon_i}(x)| \ge 1-b\}} \frac{\varepsilon_i |\nabla \phi^{\varepsilon_i}|^2}{2} + \frac{W(\phi^{\varepsilon_i})}{\varepsilon_i} dx \le s.
\end{align}
For these $s,b>0$ and $C>0$ from Lemma \ref{lem:appendix-1}, we can find $\rho\in(0,1),\,L>1$ as in Propositions \ref{prop:appendix-1}, \ref{prop:appendix-2}. We choose $a=L\varepsilon_i$ in Proposition \ref{prop:appendix-1}. Let
\begin{align*}
G_i &:= B_2(0) \cap \{|\phi^{\varepsilon_i}| \le 1 - b\} \\
    &\quad \cap \left\{x : \int_{B_r(x)} \varepsilon_i|h^{\varepsilon_i}\nabla \phi^{\varepsilon_i}| + |\xi_{\varepsilon_i}| + (1 - \nu_d^2)\varepsilon_i|\nabla \phi^{\varepsilon_i}|^2 dx   \le \rho\mu_0^{\varepsilon_i}(B_r(x))\quad\text{if } L\varepsilon_i  \le r \le 1 \right\}.
\end{align*}
The Besicovitch covering theorem, together with \eqref{eq:before-drop-1}, \eqref{eq:before-drop-2}, and \eqref{eq:before-drop-4}, implies that
\begin{align}\label{eq:after-drop-1}
\mu_0^{\varepsilon_i}((B_2(0) \cap \{|\phi^{\varepsilon_i}| \le 1 - b\}) \setminus G_i) \le \frac{C_d}{\rho} \int_{B_3(0)} \varepsilon_i|h^{\varepsilon_i}\nabla \phi^{\varepsilon_i}| + |\xi_{\varepsilon_i}| + (1 - \nu_d^2)\varepsilon_i|\nabla \phi^{\varepsilon_i}|^2 dx \to 0
\end{align}
as $i\to\infty$. Moreover, by using \eqref{eq:before-drop-3} (and changing $\rho$ by a smaller number if necessary), we can check that for sufficiently large $i$, it holds that
\begin{align}\label{eq:after-drop-2}
\frac{\mu_0^{\varepsilon_i}(B_r(x))}{\omega_{d-1} r^{d-1}} \ge 1 - 2s \quad \text{for any } x \in G_i \text{ and } r \in [L\varepsilon_i, 1].
\end{align}
Indeed, \eqref{eq:after-drop-2} holds for $r=L\varepsilon_i$ by Proposition \ref{prop:appendix-1}. Moreover, by integration by parts, we can write
\begin{align*}
\frac{d}{d\tau} \left\{ \frac{1}{\tau^{d-1}} \mu_0^{\varepsilon_i}(B_{\tau}(x)) \right\} + \frac{1}{\sigma\tau^d} \int_{B_\tau(x)} (\xi_{\varepsilon_i} + \varepsilon_i h^{\varepsilon_i}(y \cdot \nabla \phi^{\varepsilon_i})) dy \\ - \frac{\varepsilon_i}{\sigma\tau^{d+1}} \int_{\partial B_\tau(x)} (y \cdot \nabla \phi^{\varepsilon_i})^2 d\mathcal{H}^{d-1}(y) = 0.
\end{align*} Thus, thanks to Corollary \ref{cor:cor4.1-TT16-restate}(i),
\begin{align*}
\frac{1}{\tau^{d-1}} \mu_0^{\varepsilon_i}(B_{\tau}(x))\bigg|_{\tau=L\varepsilon_i}^r&\geq -\int_0^1\frac{1}{\tau^d}\left(\int_{B_{\tau}(x)}\xi_{\varepsilon_i}(y,0)^+dy\right)d\tau - \int_{L\varepsilon_i}^r \frac{1}{\sigma\tau^d} \int_{B_\tau(x)} \varepsilon_i h^{\varepsilon_i}(y \cdot \nabla \phi^{\varepsilon_i}) dyd\tau \\
&\geq -c'\varepsilon_i^{\lambda'-\lambda}|\log(\varepsilon_i)|- \int_{L\varepsilon_i}^r \frac{1}{\sigma\tau^d} \int_{B_\tau(x)} \varepsilon_i \tau |h^{\varepsilon_i}\nabla \phi^{\varepsilon_i}| dyd\tau \\
&\geq - \frac{2\rho D_T}{\sigma}.
\end{align*} for sufficiently small $\varepsilon_i\in(0,1)$. Here, we used \eqref{eq:before-drop-3}, Proposition \ref{prop:upper-bound-density-ratio}, and the definition of $G_i$. Therefore, \eqref{eq:after-drop-2} is proved for sufficiently small $\varepsilon_i\in(0,1)$ by letting $\rho>0$ be small.

Furthermore, if there exists $\delta>0$ such that $\sup_{x\in G_i}\mathrm{dist}(x,P)\geq3\delta\text{ for infinitely many }i$, there would exist $x\in G_i\cap\{|x_d|>2\delta\}$. Take a function $\phi\in C^1_c(B_3(0);\mathbb{R}_{\geq0})$ such that $\phi=1$ on $B_2(0)\cap\{|x_d|>\delta\}$ and $\phi=0$ on $\left\{|x_d|<\frac{1}{2}\delta\right\}$. From the fact that $\mu_0^{\varepsilon_i} = \|V_0^{\varepsilon_i}\| \rightharpoonup \theta \mathcal{H}^{d-1} \lfloor P$ as $i\to\infty$ and \eqref{eq:after-drop-2}, we have
\begin{align*}
(1-2s)\omega_{d-1} \delta^{d-1} \leq \mu_0^{\varepsilon_i}(B_\delta(x)) \leq \mu_0^{\varepsilon_i}(\phi) \leq (1-2s)\omega_{d-1} \frac{\delta^{d-1}}{2}\quad\text{for large }i,
\end{align*}which is a contradiction. Therefore, it holds that
\begin{align}\label{eq:after-drop-3}
\sup_{x\in G_i}\mathrm{dist}(x,P) \to 0 \quad \text{as } i \to \infty.
\end{align}

\medskip

For $x\in P\cap B_1(0)$, we set $Y:=P^{-1}(x)\cap G_i\cap\{x \,:\,\phi^{\varepsilon_i}(x)=l\}$, where $P^{-1}(x)$ denotes the preimage of $x$ under the projection, say $P$ by abuse of notations, onto the plane $P$. We now show that for sufficiently large $i$,
\begin{align}\label{eq:after-drop-4}
|Y| \le N - 1, \quad \text{for any } x \in P \cap B_1(0) \text{ and } l \in[-1+b,1-b].
\end{align}
If not, there would exist a subset $Y'=\{y_1,\cdots,y_N\}\subset Y$ of cardinality $N$. From \eqref{eq:after-drop-3}, we see $\mathrm{diam}(Y')\leq\rho$ for sufficiently large $i$. Moreover, $|y_j-y_k|>3L\varepsilon_i$ for $1\leq j<k\leq N$ due to \eqref{eq:appendix-1}, for which the assumptions are met thanks to Propositions \ref{prop:upper-bound-discrepancy}, \ref{prop:upper-bound-density-ratio}, Lemma \ref{lem:appendix-1}, and the definition of $G_i$. Now, Proposition \ref{prop:appendix-2} with $a=L\varepsilon_i$ and $R=1$ implies
\begin{align}\label{eq:after-drop-5}
\sum_{j=1}^{N} \frac{1}{(L\varepsilon_i)^{d-1}} \mu_0^{\varepsilon_i} (B_{L\varepsilon_i}(y_j)) \leq s + (1+s) \mu_0^{\varepsilon_i} (\{z\,:\, \operatorname{dist}(z, Y') < 1\}).
\end{align}
From \eqref{eq:after-drop-3} and $\mu_0^{\varepsilon_i}\to\theta\mathcal{H}^{d-1}\lfloor_P$, we have
\begin{align*}
\limsup_{i \to \infty} \mu_0^{\varepsilon_i} (\{z \mid \operatorname{dist}(z, Y') < 1\}) \leq \theta \mathcal{H}^{d-1} \lfloor_P (\overline{B_1(0)}) = \theta \omega_{d-1}.
\end{align*}
On the other hand, by \eqref{eq:after-drop-2}, the fact that $|y_j-y_k|>3L\varepsilon_i$ for $1\leq j<k\leq N$, and \eqref{eq:after-drop-5}, we obtain
\begin{align*}
N\omega_{d-1}(1-2s)\leq s+(1+s)\theta\omega_{d-1},
\end{align*}
which contradicts to the choice of $N$ that satisfies $N-1\leq\theta<N$ for sufficiently small $s\in(0,1)$.

\medskip

Let $\hat{V}_{0}^{\varepsilon_i} := V_{0}^{\varepsilon_i} \lfloor_{\{|x_d| \le 1\} \times \mathbb{G}(d, d-1)}$. Then, by \eqref{eq:before-drop-1}, \eqref{eq:before-drop-2}, \eqref{eq:after-drop-2}, we have
\begin{align*}
\omega_{d-1}\theta = \mathcal{H}^{d-1}\lfloor_P(B_1(0)) &\le \liminf_{i\to\infty} \| P_{\#} \hat{V}_0^{\varepsilon_i} \|(B_1(0)) = \liminf_{i\to\infty} \int_{B_1(0)} |\nu_d^i| d\mu_0^{\varepsilon_i} \\
&\le \liminf_{i\to\infty} \int_{B_1(0)\cap G_i} |\nu_d^i| d\mu_0^{\varepsilon_i} + 2s \\
&\le \liminf_{i\to\infty} \frac{1}{\sigma} \int_{B_1(0)\cap G_i} |\nu_d^i| \sqrt{2W(\phi^{\varepsilon_i})} |\nabla\phi^{\varepsilon_i}| dx + 2s.
\end{align*}
By the area formula and the coarea formula, we derive
\begin{align*}
& \int_{B_1(0)\cap G_i} |\nu_d^i| \sqrt{2W(\phi^{\varepsilon_i})} |\nabla\phi^{\varepsilon_i}| dx \\
&= \int_{-1+b}^{1-b} \sqrt{2W(\tau)} \int_{B_1(0)\cap G_i \cap \{\phi^{\varepsilon_i}=\tau\}} |\Lambda_{d-1}\nabla Px \circ (I - \nu^i \otimes \nu^i)| d\mathcal{H}^{d-1}(x) d\tau \\
&= \int_{-1+b}^{1-b} \sqrt{2W(\tau)} \int_{B_1(0)\cap\{x_d=0\}} \mathcal{H}^0(\{\phi^{\varepsilon_i} = \tau\} \cap G_i \cap P^{-1}(x)) d\mathcal{H}^{d-1}(x) d\tau \\
&\le \int_{-1+b}^{1-b} \sqrt{2W(\tau)} \int_{B_1(0)\cap\{x_d=0\}} (N-1) d\mathcal{H}^{d-1}(x) d\tau \\
&\le \sigma(N-1)\omega_{d-1}.
\end{align*}
Connecting the above two inequalities and letting $s\to0$, we see that $\theta\leq N-1$, and therefore, $\theta=N-1$. We finish the proof.
\end{proof}

\section{Proof of Theorem \ref{thm:main-thm}}\label{sec:pf-main-thm}
We prove Theorem \ref{thm:main-thm} in this section.


\begin{proof}[Proof of Theorem \ref{thm:main-thm}]
We begin by verifying the parts of the theorem that are readily developed in the previous sections.

\medskip

\textbf{Part (C).} It directly follows from Proposition \ref{prop:L2-estimate} and the weak-compactness of $L^2(0,T)$ for any $T>0$.

\medskip

\textbf{Parts (A.i)-(A.iii).} Part (A.i) follows from Proposition \ref{prop:convergence-measure}. Part (A.ii) is obtained by applying the dominated convergence theorem and Part (A.i) to
\begin{align*}
\lim_{\varepsilon\to0} \int_0^{\infty}\int_\Omega \varphi \, d\mu_t^{\varepsilon} \, dt = \int_{\Omega \times [0, \infty)} \varphi \, d\mu \qquad \text{for any } \varphi \in C_c(\Omega \times [0, \infty)).
\end{align*}
Part (A.iii) follows from Theorem \ref{thm:integrality}. We prove Part (A.iv) after Part (B).

\medskip

\textbf{Part (B).} We let $\psi^{\varepsilon}:=K\circ\phi^{\varepsilon}$, where $K(s):=\sigma^{-1}\int_{-1}^s\sqrt{2W(u)}du$. By \eqref{eq:modica-mortola}, the function $\psi^{\varepsilon}$ satisfies
\begin{align*}
\int_{\Omega}|\nabla\psi^{\varepsilon}(\cdot,t)|dx\leq\mu^{\varepsilon}_t(\chi_{\Omega})\leq C.
\end{align*}
Similarly, due to Proposition \ref{prop:energy-dissipation},
\begin{align*}
\int_0^T\int_{\Omega}|\psi^{\varepsilon}_t|dxdt\leq\sigma^{-1}\int_0^T\int_{\Omega}\frac{\varepsilon(\phi^{\varepsilon}_t)^2}{2}+\frac{W(\phi^{\varepsilon})}{\varepsilon}dxdt\leq C_T\qquad\text{for }T>0.
\end{align*}
Therefore, $\psi^{\varepsilon}\in BV_{loc}(\Omega\times[0,\infty))$, and there exist a subsequence $\varepsilon\to0$ (with the same index by abuse of notations) and $\psi\in BV_{loc}(\Omega\times[0,\infty))$ such that $\psi^{\varepsilon}\to\psi$ strongly in $L^1_{loc}(\Omega\times[0,\infty))$ and a.e. pointwise. As $\phi^{\varepsilon}$ converges to $\pm1$ as $\varepsilon\to0$ a.e., we see that $\psi$ assumes only the two values $0$ and $1$, which implies
\begin{align*}
\psi=\lim_{\varepsilon\to0}\frac{1+K^{-1}\circ\psi^{\varepsilon}}{2}.
\end{align*}
Therefore, $\psi\in BV_{loc}(\Omega\times[0,\infty))$, and $\phi^{\varepsilon}$ converges to $2\psi-1$ strongly in $L^1_{loc}(\Omega\times[0,\infty))$ and a.e. pointwise. We further check that for almost every $0\leq s<t$, it holds
\begin{align*}
\int_{\Omega}|\psi(\cdot,t)-\psi(\cdot,s)|dx&\leq\liminf_{\varepsilon\to0}\int_{\Omega}|\psi^{\varepsilon}(\cdot,t)-\psi^{\varepsilon}(\cdot,s)|dx\\
&\leq\liminf_{\varepsilon\to0}\int_{\Omega}\int_s^t|\psi^{\varepsilon}_t|dtdx\\
&\leq\liminf_{\varepsilon\to0}\int_{\Omega}\int_s^t\frac{\varepsilon|\phi^{\varepsilon}_t|^2}{2}\sqrt{t-s}+\frac{W(\phi^{\varepsilon})}{\varepsilon\sqrt{t-s}}dtdx\\
&\leq C\sqrt{t-s}.
\end{align*}
Here, Fatou's lemma is used in the first line and Proposition \ref{prop:energy-dissipation} is used in the last line. Therefore, $\psi\in C^{0,1/2}([0,\infty);L^1(\Omega))$.

As $\psi$ assumes only the two values $0$ and $1$, it has support set, say $E(t)$ at time $t\geq0$. Then, we clearly have $E(0)=U_0$, $O_+\subset E(t)$ and $O_-\subset\Omega\setminus E(t)$ for $t\geq0$ from Proposition \ref{prop:barrier-functions}. Moreover, by \eqref{eq:VP-from-energy-dissipation}, we see that $\int_{\Omega\setminus O}\psi(\cdot,t)dx=\int_{\Omega\setminus O}\psi(\cdot,0)dx$, which implies $\mathcal{L}^d(E(t))=\mathcal{L}^d(E(0))$ for $t\geq0$ from the fact that $O_+\subset E(t)$ and $O_-\subset\Omega\setminus E(t)$. Hence, (B.i) and (B.ii) are proved. To prove (B.iii), let $\Phi\in C(\Omega;\mathbb{R}_{\geq0})$ and $t>0$. Then,
\begin{align*}
\int_{\Omega}\Phi d\|\nabla\psi(\cdot,t)\|\leq\liminf_{\varepsilon\to0}\int_{\Omega}\Phi|\nabla\psi^{\varepsilon}|dx\leq\liminf_{\varepsilon\to0}\sigma^{-1}\int_{\Omega}\Phi\left(\frac{\varepsilon|\nabla\phi^{\varepsilon}|^2}{2}+\frac{W(\phi^{\varepsilon})}{\varepsilon}\right)dx=\int_{\Omega}\Phi d\mu_t.
\end{align*}
Therefore, (B.iii) is proved.

We now prove (B.iv). We claim that $\mathrm{supp}(\mu_t)\cap O=\emptyset$ for all $t\geq0$, where $\mu_t$ is from (A). It suffices to prove $\mu^{\varepsilon}_t(\Phi)\to0$ as $\varepsilon\to0$ for $\Phi\in C_c(O_+)$ as the argument for $\Phi\in C_c(O_-)$ is similar. For the part of the energy responsible for $\frac{W(\phi^{\varepsilon})}{\varepsilon}$, it suffices to verify that $\int_{B_{r}(y)}\frac{W(\phi^{\varepsilon})}{\varepsilon}dx\to0$ as $\varepsilon\to0$ for every ball $B_{r}(y)$, $r\in(0,R_0)$ such that $B_{R_0+2\sqrt{\varepsilon}}(y)\subset O_+$. We recall the notations $\underline{r}_{y}(x)=\frac{1}{2R_0}\left(R_0^2-|x-y|^2\right)$, $\underline{\phi}^{\varepsilon}_y(x)=q^{\varepsilon}(\underline{r}_y(x))$. By Propositions \ref{prop:strong-max-principle}, \ref{prop:barrier-functions}, we have $\underline{\phi}^{\varepsilon}_y\leq\phi^{\varepsilon}\leq1$ everywhere, and therefore, for a fixed $r\in(0,R_0)$ we have
\begin{align*}
\int_{B_{r}(y)}\frac{W(\phi^{\varepsilon})}{\varepsilon}dx\leq\int_{B_{r}(y)}\frac{W(\underline{\phi}^{\varepsilon}_y)}{\varepsilon}dx&\leq\int_{B_{r}(y)}\varepsilon^{-1}W\left(\tanh{\left(\frac{R_0^2-r^2}{2R_0\varepsilon}\right)}\right)dx\\
&\leq2\mathcal{L}^d(B_r(y))\varepsilon^{-1}\left|1-\tanh{\left(\frac{R_0^2-r^2}{2R_0\varepsilon}\right)}\right|^2\to0\quad\text{as }\varepsilon\to0.
\end{align*}

For the other part of the energy responsible for $\frac{\varepsilon|\nabla\phi^{\varepsilon}|^2}{2}$, it is sufficient to prove that $\int_{B_{R_0}(y)}\Phi\cdot\frac{\varepsilon|\nabla\phi^{\varepsilon}|^2}{2}dx\to0$ as $\varepsilon\to0$ for $\Phi\in C_c(B_{R_0}(y))$ where $B_{R_0+2\sqrt{\varepsilon}}(y)\subset O_+$. The case $t=0$ is clear and we assume the other case $t>0$. Find $r\in(0,R_0)$ such that $\mathrm{supp}(\Phi)\subset\overline{B}_r(y)\subset B_{R_0}(y)$. Write
\begin{align}\label{eq:rewriting}
\int_{B_{r}(y)}\Phi\nabla(\phi^{\varepsilon}-1)\cdot\nabla\phi^{\varepsilon}dx=-\int_{B_{r}(y)}\Phi(\phi^{\varepsilon}-1)\Delta\phi^{\varepsilon}+(\phi^{\varepsilon}-1)\nabla\Phi\cdot\nabla\phi^{\varepsilon}dx
\end{align}
By Propositions \ref{prop:strong-max-principle}, \ref{prop:barrier-functions}, we have $\underline{\phi}^{\varepsilon}_y\leq\phi^{\varepsilon}\leq1$ everywhere, and thus,
\begin{align}\label{eq:fast-convergence}
\left|1-\phi^{\varepsilon}\right|\leq\left|1-\tanh{\left(\frac{R_0^2-r^2}{2R_0\varepsilon}\right)}\right|\qquad\text{on }\overline{B}_r(y).
\end{align}
From the parabolic scaling argument, we have (see \cite[Lemma 3.9]{T21})
\begin{align}\label{eq:slow-divergence}
\sup_{\Omega\times[\varepsilon^2,T)}\left\{\varepsilon|\nabla\phi^{\varepsilon}|+\varepsilon^2\|D^2\phi^{\varepsilon}\|\right\}\leq C_{T}\qquad\text{for each }T>0.
\end{align}
We can easily see that the claim that $\int_{B_{R_0}(y)}\Phi\cdot\frac{\varepsilon|\nabla\phi^{\varepsilon}|^2}{2}dx\to0$ as $\varepsilon\to0$ is obtained by combining \eqref{eq:rewriting}, \eqref{eq:fast-convergence}, \eqref{eq:slow-divergence}. Therefore, that $\mathrm{supp}(\mu_t)\cap O=\emptyset$ for all $t\geq0$ is proved and we have verified (B.iv).

\medskip

\textbf{Part (A.iv).} We first of all note that the map
\begin{align*}
t\mapsto E^{\varepsilon}(t)-\int_{\{x:s(x)>\frac{1}{2}\sqrt{\varepsilon}\}}g^{\varepsilon}(x,t) k(\phi^{\varepsilon}(x,t))dx,
\end{align*}
where $g^{\varepsilon}$ is independent of $t$ on the integration region, is nonincreasing in $t$ due to \eqref{eq:energy-dissipation-equality}. By Helley's selection theorem, there exist a subsequence of $\varepsilon\to0$, still denoted by $\varepsilon\to0$, and a nonincreasing function $E(t)$ such that $E^{\varepsilon}(t)\to E(t)$ as $\varepsilon\to0$. Moreover, by Proposition \ref{prop:L2-estimate} for $T>0$, we have
\begin{align*}
\int_0^T\frac{1}{2\varepsilon^{\alpha}}\left(\int_{\left\{x\in\Omega:s(x)\leq\frac{1}{2}\sqrt{\varepsilon}\right\}}\eta_{\varepsilon}(s(x))\left(k(\phi_0^{\varepsilon})-k(\phi^{\varepsilon})\right)dx\right)^2\,dt\leq\int_0^T\frac{\varepsilon^{\alpha}}{2}|\lambda^{\varepsilon}(t)|^2\,dt\to0\qquad\text{as }\varepsilon
\to0,
\end{align*}
and therefore,
\begin{align*}
\liminf_{\varepsilon\to0}\frac{1}{2\varepsilon^{\alpha}}\left(\int_{\left\{x\in\Omega:s(x)\leq\frac{1}{2}\sqrt{\varepsilon}\right\}}\eta_{\varepsilon}(s(x))\left(k(\phi_0^{\varepsilon})-k(\phi^{\varepsilon})\right)dx\right)^2=0\qquad\text{for a.e. }t\geq0
\end{align*}
by Fatou's lemma. Note that Part (B.iv) implies that
\begin{align*}
\int_{\{x:s(x)>\frac{1}{2}\sqrt{\varepsilon}\}}g^{\varepsilon}\cdot\left(k(\phi^{\varepsilon}(t))-k(\phi^{\varepsilon}(0)\right)dx\to0\qquad\text{as }\varepsilon\to0.
\end{align*}
Combining this with the convergences of $E^{\varepsilon}(t)$ and of $\mu^{\varepsilon}_t$, we have
\begin{align*}
\lim_{\varepsilon\to0}\frac{1}{2\varepsilon^{\alpha}}\left(\int_{\left\{x\in\Omega:s(x)\leq\frac{1}{2}\sqrt{\varepsilon}\right\}}\eta_{\varepsilon}(s(x))\left(k(\phi_0^{\varepsilon})-k(\phi^{\varepsilon})\right)dx\right)^2=0\qquad\text{and}\qquad E(t)=\sigma\mu_t(\Omega)
\end{align*}
for a.e. $t\geq0$. The continuity of $t\mapsto\mu_t(\Omega)$ on $[0,\infty)\setminus B$ finishes proof of Part (A.iv).

\medskip

\textbf{Part (D).} Set $d \overline{\mu}^{\varepsilon} := \frac{1}{\sigma} \sqrt{2 W(\phi^{\varepsilon})} |\nabla \phi^{\varepsilon}| \, dx \, dt$. Then, for $\varphi\in C_c\left(\Omega\times[0,\infty)\right)$, we have, by Theorem \ref{thm:vanishing-xi},
\begin{align*}
\left|\int_{\Omega\times(0,\infty)}\varphi d \overline{\mu}^{\varepsilon}-\int_{\Omega\times(0,\infty)}\varphi d \mu^{\varepsilon}\right|&\leq\frac{\|\varphi\|_{L^{\infty}}}{\sigma}\int_{\mathrm{supp}(\varphi)}\left|\frac{\varepsilon |\nabla \phi^\varepsilon|^2}{2} + \frac{W(\phi^\varepsilon)}{\varepsilon} - \sqrt{2W(\phi^\varepsilon)} |\nabla \phi^\varepsilon|\right|dxdt\\
&\leq\frac{\|\varphi\|_{L^{\infty}}}{\sigma}\int_{\mathrm{supp}(\varphi)}\left( \frac{\sqrt{\varepsilon} |\nabla \phi^\varepsilon|}{\sqrt{2}} - \frac{\sqrt{W(\phi^\varepsilon)}}{\sqrt{\varepsilon}} \right)^2dt\\
&\leq\frac{\|\varphi\|_{L^{\infty}}}{\sigma}|\xi^{\varepsilon}|(\mathrm{supp}(\varphi))\to0\qquad\text{as }\varepsilon\to0.
\end{align*}
Therefore, by Proposition \ref{prop:convergence-measure}, we have $\overline{\mu}^{\varepsilon}\to \mu$ as Radon measures as $\varepsilon\to0$. We note that, for $T>0$, by Proposition \ref{prop:L2-estimate}, there exists $C_T>0$ such that
\begin{align*}
\sup_{\varepsilon\in(0,\varepsilon_1)}\int_{\Omega\times(0,T)}|g^{\varepsilon}|^2d\overline{\mu}^{\varepsilon}\leq C\sup_{\varepsilon\in(0,\varepsilon_1)}\int_0^T\max\{|\lambda^{\varepsilon}(t)|^2,1\}dt\leq C_T.
\end{align*}
for $\varepsilon_1\in(0,1)$ as in Proposition \ref{prop:L2-estimate}. Here, we used the fact from Proposition \ref{prop:energy-dissipation} and Young's inequality that there exists $C>0$ such that for all $t\geq0$,
\begin{align*}
\overline{\mu}^{\varepsilon}(\Omega)\leq\frac{1}{\sigma}\int_{\Omega}\left(\frac{\varepsilon|\nabla\phi^{\varepsilon}|^2}{2}+\frac{W(\phi^{\varepsilon})}{\varepsilon}\right)dx\leq C.
\end{align*}
Consequently, by \cite[Theorem 4.4.2]{H86}, there exists $g\in L^2_{loc}\left(0,\infty;L^2(\mu_t)^d\right)$ such that
\begin{align}\label{eq:Hutchinson}
\int_{\Omega\times[0,\infty)}g\cdot\Phi d\mu=\lim_{\varepsilon\to0}\frac{1}{\sigma}\int_{\Omega\times[0,\infty)}-g^{\varepsilon}\sqrt{2W(\phi^{\varepsilon})}\nabla\phi^{\varepsilon}\cdot\Phi dxdt.
\end{align}
for any $\Phi\in C_c\left(\Omega\times[0,\infty);\mathbb{R}^d\right)$. As $g^{\varepsilon}=\lambda^{\varepsilon}$ on $\Omega\setminus \overline{O}$, we have that for $\Phi\in C_c\left(\left(\Omega\setminus\overline{O}\right)\times[0,\infty);\mathbb{R}^d\right)$,
\begin{align*}
\int_{\Omega\times[0,\infty)}g\cdot\Phi d\mu&=\lim_{\varepsilon\to0}\frac{1}{\sigma}\int_{\Omega\times[0,\infty)}-\lambda^{\varepsilon}\sqrt{2W(\phi^{\varepsilon})}\nabla\phi^{\varepsilon}\cdot\Phi dxdt\\
&=\lim_{\varepsilon\to0}\frac{1}{\sigma}\int_{\Omega\times[0,\infty)}-\lambda^{\varepsilon}\nabla(k(\phi^{\varepsilon}))\cdot\Phi dxdt\\
&=\lim_{\varepsilon\to0}\frac{1}{\sigma}\int_{\Omega\times[0,\infty)}\lambda^{\varepsilon}k(\phi^{\varepsilon})\mathrm{div}(\Phi)dxdt.
\end{align*}
As $K(s)=\sigma^{-1}\int_{-1}^s\sqrt{2W(u)}du=\frac{1}{2}+\sigma^{-1}k(s)$, we have $\psi^{\varepsilon}=K\circ\phi^{\varepsilon}=\frac{1}{2}+\sigma^{-1}k(\phi^{\varepsilon})$, which implies that $k(\phi^{\varepsilon})$ converges to $\sigma\left(\psi-\frac{1}{2}\right)$ as $\varepsilon\to0$. Therefore, for $\Phi\in C_c\left(\left(\Omega\setminus\overline{O}\right)\times[0,\infty);\mathbb{R}^d\right)$,
\begin{align}
\int_{\Omega\times[0,\infty)}g\cdot\Phi d\mu&=\int_0^{\infty}\lambda\int_{\Omega}\left(\psi-\frac{1}{2}\right)\mathrm{div}(\Phi)dxdt\notag\\
&=\int_0^{\infty}\lambda\int_{\Omega}\psi\mathrm{div}(\Phi)dxdt\notag\\
&=-\int_0^{\infty}\lambda\int_{\Omega}\nu\cdot\Phi d\|\nabla\psi(\cdot,t)\|dt\notag\\
&=-\int_{\Omega\times[0,\infty)}\lambda\frac{d\|\nabla\psi(\cdot,t)\|}{d\mu_t}\nu\cdot\Phi d\mu.\label{eq:forcing-term}
\end{align}
Here, $\nu(\cdot,t)$ is the inner unit normal vector of $E(t)$ on $\partial^{\ast}E(t)$. If we set $\theta:=\left(\frac{d\|\nabla\psi(\cdot,t)\|}{d\mu_t}\right)^{-1}$ on $\partial^{\ast}E(t)$, then $\theta$ is positive integer-valued $\mathcal{H}^{d-1}$-a.e. on $\partial^{\ast}E(t)$ by the integrality of $\mu_t$.

Now we obtain the asymptotic limit for the velocity vector. Let
\begin{align*}
v^{\varepsilon} = 
\begin{cases}
-\frac{\partial_t \phi^{\varepsilon}}{|\nabla \phi^{\varepsilon}|} \frac{\nabla \phi^{\varepsilon}}{|\nabla \phi^{\varepsilon}|}, & \text{if } |\nabla \phi^{\varepsilon}| \neq 0, \\
0, & \text{otherwise.}
\end{cases}
\end{align*}
With similar arguments to above, we see that $d\tilde{\mu}^{\varepsilon}:=\frac{\varepsilon}{\sigma}|\nabla\phi^{\varepsilon}|^2dxdt$ converges $d\mu$ as Radon measures as $\varepsilon\to0$ and that there exists $v\in L^2_{loc}\left(0,\infty;L^2(\mu_t)^d\right)$ such that
\begin{align*}
\int_{\Omega\times[0,\infty)}v\cdot\Phi d\mu=\lim_{\varepsilon\to0}\int_{\Omega\times[0,\infty)}v^{\varepsilon}\cdot\Phi d\widetilde{\mu}_t^{\varepsilon}\quad\text{for any }\Phi\in C_c\left(\Omega\times[0,\infty);\mathbb{R}^d\right).
\end{align*}
We moreover have, by \eqref{eq:estimate-from-energy-dissipation}, \eqref{eq:L2-approximate-mean-curvature}, Theorem \ref{thm:rectifiability}, and \cite[Lemma 7.1]{MR08} (whose proof works for all dimensions), that
\begin{align*}
\lim_{\varepsilon\to0} \frac{1}{\sigma}\int_{\Omega\times(0,\infty)} -\varepsilon \left(\Delta \phi^{\varepsilon} - \frac{W'(\phi^{\varepsilon})}{\varepsilon^2}  \right) \nabla \phi^{\varepsilon} \cdot \Phi \, dx dt =\int_{\Omega\times(0,\infty)} h \cdot \Phi \, d\mu
\end{align*}
for $\Phi\in C_c\left(\Omega\times[0,\infty);\mathbb{R}^d\right)$. Here, $h$ is the generalized mean curvature vector as in Theorem \ref{thm:rectifiability}.
By combining this with \eqref{eq:Hutchinson}, we obtain
\begin{align*}
\int_{\Omega\times(0,\infty)} v \cdot \Phi \, d\mu & = \lim_{\varepsilon\to0} \int_{\Omega\times(0,\infty)} v^{\varepsilon} \cdot \Phi \, d\tilde{\mu}^{\varepsilon} \\
&= \lim_{\varepsilon\to0} \text{$\frac{1}{\sigma}$}\int_{\Omega\times(0,\infty)} -\varepsilon \left(\Delta \phi^{\varepsilon} - \frac{W'(\phi^{\varepsilon})}{\varepsilon^2} + g^{\varepsilon} \frac{\sqrt{2W(\phi^{\varepsilon})}}{\varepsilon} \right) \nabla \phi^{\varepsilon} \cdot \Phi \, dx dt \\
&=\int_{\Omega\times(0,\infty)} (h+g) \cdot \Phi \, d\mu
\end{align*}
for $\Phi\in C_c\left(\Omega\times[0,\infty);\mathbb{R}^d\right)$. Therefore, $(\mu_t)_{t>0}$ is an $L^2$-flow with the velocity vector $v=h+g$ (see \cite[Proposition 4.3]{T17} and \cite[Lemma 6.3]{MR08}). In particular, by \eqref{eq:forcing-term}, we obtain
\[
v=h-\frac{\lambda}{\theta}\nu(\cdot,t)\quad\text{$\mathcal{H}^{d-1}$-a.e. on }\partial^{\ast}E(t)\setminus\overline{O}.
\]
We finish the proof.
\end{proof}

\appendix
\section*{Appendix}\label{sec:appendix-A}
We leave several statements as preparatory steps for the proof of Theorem \ref{thm:integrality}. We refer to \cite{HT00,T03,TT16} for references and \cite[Subsection 4.4]{T23} in particular. In this appendix, we adopt the settings and the facts that are developed in this paper.



\begin{proposition}\label{prop:appendix-3}(See \cite[Proposition 4.1]{T03} and \cite[Proposition 15]{T23})
For any $s,\,T>0$ and $a\in(0,T)$, there exist $b=b(a,s,T),\,\varepsilon_5=\varepsilon_5(a,s,T)>0$ such that
\begin{align*}
\int_{\{x \in B_1(0) | |\phi^{\varepsilon}(x, t)| \ge 1 - b\}} \frac{W(\phi^{\varepsilon}(x, t))}{\varepsilon} dx \le s
\end{align*}
for all $t\in(a,T)$ and $\varepsilon\in(0,\varepsilon_5)$.
\end{proposition}

The validity of Proposition \ref{prop:appendix-3} under the setting of this paper follows from the proof of \cite[Proposition 15]{T23} in the exactly same way based on \cite[Lemmas 9, 10]{T23}. The lemmas also follow from the identical proofs, with the only one adaptation made to \cite[(96)]{T23}. As the adaptation is identical to the one shown in the proof of Lemma \ref{lem:auxiliary-lemma2} of this paper, we skip the proof of Proposition \ref{prop:appendix-3}.

\begin{lemma}\label{lem:appendix-1}
There exists a constant $C>0$ (depending only on the problem) such that for each $\varepsilon\in(0,1)$, it holds
\begin{align*}
\sup_{\Omega \times [0, T)} \varepsilon |\nabla \phi^{\varepsilon}| + \sup_{x, y \in \Omega,\, t \in [0, T)} \frac{\varepsilon^{\frac{3}{2}} |\nabla \phi^{\varepsilon}(x, t) - \nabla \phi^{\varepsilon}(y, t)|}{|x - y|^{\frac{1}{2}}} \le C.
\end{align*}
\end{lemma}

In the following statements, the notation $\nu_d$ refers to the $x_d$-component of the vector $\frac{\nabla\phi^{\varepsilon}}{|\nabla\phi^{\varepsilon}|}$ when $|\nabla\phi^{\varepsilon}|\neq0$ and $0$ otherwise.

\begin{proposition}\label{prop:appendix-1}
For any $s,b,\lambda\in(0,1)$, and $C\in(1,\infty)$, there exist $\rho,\varepsilon_1\in(0,1)$, and $L\in(1,\infty)$ satisfying the following: Suppose that $\varepsilon\in(0,\varepsilon_1)$,
\begin{align*}
&\sup_{\Omega \times [0, T)} \varepsilon |\nabla \phi^{\varepsilon}| + \sup_{x, y \in \Omega,\, t \in [0, T)} \frac{\varepsilon^{\frac{3}{2}} |\nabla \phi^{\varepsilon}(x, t) - \nabla \phi^{\varepsilon}(y, t)|}{|x - y|^{\frac{1}{2}}} \le C,\\
&\int_{B_{4\varepsilon L}(0)} (|\xi_{\varepsilon}| + (1 - (\nu_d)^2)\varepsilon|\nabla \phi^{\varepsilon}|^2) dx \le \rho(4\varepsilon L)^{d-1},\qquad\text{and}\\
&\sup_{B_{4\varepsilon L}(0)} (\xi_{\varepsilon})^+ \le \varepsilon^{-\lambda}.
\end{align*}
Then,
\begin{align}
[-1 + b, 1 - b] \subset \phi^{\varepsilon}(J) \qquad \text{and} \qquad \inf_{x \in J} \partial_{x_d} \phi^{\varepsilon}(x) > 0 \quad \text{or} \quad \sup_{x \in J} \partial_{x_d} \phi^{\varepsilon}(x) < 0,\label{eq:appendix-1}\tag{A.1}
\end{align}
where $J := B_{3\varepsilon L}(0) \cap \{(0, \dots, 0, x_d) \in \mathbb{R}^d\, :\, x_d \in \mathbb{R}\}$ and
\begin{align*}
\left| \sigma - \frac{1}{\omega_{d-1}(L\varepsilon)^{d-1}} \int_{B_{\varepsilon L}(0)} \frac{\varepsilon |\nabla \phi^{\varepsilon}(x)|^2}{2} + \frac{W(\phi^{\varepsilon}(x))}{\varepsilon}dx \right| \le s.
\end{align*}
\end{proposition}

\begin{proposition}\label{prop:appendix-2}
For any $R,E_0\in(0,\infty),\,s\in(0,1)$, and $N\in\mathbb{N}$, there exists $\rho\in(0,1)$ satisfying the following: Suppose that a subset $Y\subset\mathbb{R}^d$ satisfies that $|Y|\leq N$, $Y\subset\{(0, \dots, 0, x_d) \in \mathbb{R}^d | x_d \in \mathbb{R}\}$, $\mathrm{diam}(Y)\leq\rho R$, and there exists $a\in(0,R)$ such that $|y-z|>3a$ for any $y,z\in Y$ with $y\neq z$. Moreover, we assume that
\begin{itemize}
    \item[(i)] For any $x\in Y$ and $r\in[a,R]$,
    \begin{align*}
    &\int_{B_{r}(x)} |\xi_{\varepsilon}| + (1 - (\nu_d)^2)\varepsilon|\nabla \phi^{\varepsilon}|^2 + \varepsilon|\nabla \phi^{\varepsilon}| \left|\Delta \phi^{\varepsilon} - \frac{W'(\phi^{\varepsilon})}{\varepsilon^2}\right| dy \le \rho r^{d-1}\\
    &\qquad\qquad\text{and}\qquad\int_{B_r(x)} \varepsilon |\nabla \phi^{\varepsilon}|^2 dy \le E_0 r^{d-1}.
    \end{align*}
    \item[(ii)] For any $x\in Y$,
    \begin{align*}
    \int_a^R \frac{1}{\tau^d} \int_{B_{\tau}(x)} (\xi_{\varepsilon})^+ dyd\tau \le \rho.
    \end{align*}
\end{itemize}
Then, it holds that
\begin{align*}
\sum_{x \in Y} \frac{1}{a^{d-1}} \int_{B_a(x)} \frac{\varepsilon |\nabla \phi^{\varepsilon}(y)|^2}{2} + \frac{W(\phi^{\varepsilon}(y))}{\varepsilon}dy \le s + \frac{1+s}{R^{d-1}} \int_{\{y | \mathrm{dist}(y, Y) < R\}} \frac{\varepsilon |\nabla \phi^{\varepsilon}(y)|^2}{2} + \frac{W(\phi^{\varepsilon}(y))}{\varepsilon}dy.
\end{align*}
\end{proposition}

\section*{Acknowledgments} The author expresses his gratitude to the referee for the comments on improving the presentation and the contents of the paper.



\bibliographystyle{plain}
\bibliography{Preprint}

\begin{thebibliography}{10}

\bibitem{AA14}
Matthieu Alfaro and Pierre Alifrangis.
\newblock Convergence of a mass conserving allen-cahn equation whose lagrange multiplier is nonlocal and local.
\newblock {\em Interfaces Free Bound.}, 16(2):243--268, 2014.

\bibitem{AC79}
Samuel Allen and John~W. Cahn.
\newblock A microscopic theory for antiphase boundary motion and its application to antiphase domain coarsening.
\newblock {\em Acta Metallurgh}, 27:2789--2796, 1979.

\bibitem{ACN12}
Luis Almeida, Antonin Chambolle, and Matteo Novaga.
\newblock Mean curvature flow with obstacles.
\newblock {\em H. Poincar\'e Anal. Non Lin\'eaire}, 29:667--681, 2012.

\bibitem{AFP00}
Luigi Ambrosio, Nicola Fusco, and Diego Pallara.
\newblock {\em Functions of bounded variation and free discontinuity problems}.
\newblock The Clarendon Press, Oxford University Press, New York, 2000.

\bibitem{BMR24}
Matteo Bonforte, Francesco Maggi, and Daniel Restrepo.
\newblock Asymptotic behavior of a diffused interface volume-preserving mean curvature flow.
\newblock {\em arXiv:2407.18868}, 2024.

\bibitem{B78}
Kenneth Brakke.
\newblock {\em The motion of a surface by its mean curvature}.
\newblock Princeton University Press, 1978.

\bibitem{BB11}
Elie Bretin and Morgan Brassel.
\newblock A modified phase field approximation for mean curvature flow with conservation of the volume.
\newblock {\em Math. Methods Appl. Sci.}, 34(10):1157--1180, 2011.

\bibitem{BP12}
Elie Bretin and Valerie Perrier.
\newblock Phase field method for mean curvature flow with boundary constraints.
\newblock {\em ESAIM Math. Model. Numer. Anal.}, 46:1509--1526, 2012.

\bibitem{BS97}
Lia Bronsard and Barbara Stoth.
\newblock Volume-preserving mean curvature flow as a limit of a nonlocal ginzburg-landau equation.
\newblock {\em SIAM J. Math. Anal.}, 28(4):769--807, 1997.

\bibitem{C96}
Xinfu Chen.
\newblock Global asymptotic limit of solutions of the cahn-hilliard equation.
\newblock {\em J. Differential Geometry}, 44:262--311, 1996.

\bibitem{ESV12}
Charles~M. Elliott, Bj\"orn Stinner, and Chandrasekhar Venkataraman.
\newblock Modelling cell motility and chemotaxis with evolving surface finite elements.
\newblock {\em J. R. Soc. Interface}, 9:3027--3044, 2012.

\bibitem{EI05}
Joachim Escher and Kazuo Ito.
\newblock Some dynamic properties of volume preserving curvature driven flows.
\newblock {\em Math. Ann.}, 333(1):213--230, 2005.

\bibitem{ES98}
Joachim Escher and Gieri Simonett.
\newblock The volume preserving mean curvature flow near spheres.
\newblock {\em Proc. Amer. Math. Soc.}, 126(9):2789--2796, 1998.

\bibitem{G86}
Michael Gage.
\newblock {\em On an area-preserving evolution equation for plane curves}, volume~51.
\newblock American Mathematical Society, 1986.

\bibitem{GTZ19}
Yoshikazu Giga, Hung~V. Tran, and Longjie Zhang.
\newblock On obstacle problem for mean curvature flow with driving force.
\newblock {\em Geom. Flows}, 4:9--29, 2019.

\bibitem{G84}
Enrico Giusti.
\newblock {\em Minimal Surfaces and Functions of Bounded Variation , Monographs in Mathematics}, volume~80.
\newblock Birkh\"auser Verlag, Basel, 1984.

\bibitem{G97}
Dmitry Golovaty.
\newblock The volume-preserving motion by mean curvature as an asymptotic limit of reaction-diffusion equations.
\newblock {\em Quart. Appl. Math.}, 55(2):243--298, 1997.

\bibitem{H87}
Gerhard Huisken.
\newblock The volume preserving mean curvature flow.
\newblock {\em J. reine angew. Math.}, 382:35--48, 1987.

\bibitem{H90}
Gerhard Huisken.
\newblock Asymptotic behavior for singularities of the mean curvature flow.
\newblock {\em J. Differential Geom.}, 31:285--299, 1990.

\bibitem{H86}
John~E. Hutchinson.
\newblock Second fundamental form for varifolds and the existence of surfaces minimising curvature.
\newblock {\em Indiana Univ. Math. J.}, 35:45--71, 1986.

\bibitem{HT00}
John~E. Hutchinson and Yoshihiro Tonegawa.
\newblock Convergence of phase interfaces in the van der waals-cahn-hilliard theory.
\newblock {\em Calc. Var. Partial Differential Equations}, 10:49--84, 2000.

\bibitem{I93}
Tom Ilmanen.
\newblock Convergence of the allen-cahn equation to brakke's motion by mean curvature.
\newblock {\em J. Differential Geom.}, 38(2):417--461, 1993.

\bibitem{KK20}
Inwon Kim and Dohyun Kwon.
\newblock Volume preserving mean curvature flow for star-shaped sets.
\newblock {\em Calc. Var. Partial Differential Equations}, 59(81), 2020.

\bibitem{KS06}
Robert Kohn and Sylvia Serfaty.
\newblock A deterministic-control-based approach to motion by curvature.
\newblock {\em Comm. Pure Appl. Math.}, 59:344--407, 2006.

\bibitem{KL24}
Milan Kroemer and Tim Laux.
\newblock Quantitative convergence of the nonlocal allen-cahn equation to volume-preserving mean curvature flow.
\newblock {\em Math. Ann.}, 2024.

\bibitem{L24}
Tim Laux.
\newblock Weak-strong uniqueness for volume-preserving mean curvature flow.
\newblock {\em Rev. Mat. Iberoam.}, 40(1):93--110, 2024.

\bibitem{LS18}
Tim Laux and Theresa~M. Simon.
\newblock Convergence of allen-cahn equation to multiphase mean curvature flow.
\newblock {\em Calc. Var. Partial Differential Equations}, 71(8):1597--1647, 2018.

\bibitem{LS17}
Tim Laux and Drew Swartz.
\newblock Convergence of thresholding schemes incorporating bulk effects.
\newblock {\em Interfaces Free Bound.}, 19(2):273--304, 2017.

\bibitem{LS95}
Stephan Luckhaus and Thomas Sturzenhecker.
\newblock Implicit time discretization for the mean curvature flow equation.
\newblock {\em Calc. Var. Partial Differential Equations}, 3(2):253--271, 1995.

\bibitem{M66}
Bernard Malgrange.
\newblock {\em Ideals of differentiable functions}.
\newblock Oxford Univ. Press, 1966.

\bibitem{MS00}
Uwe~F. Mayer and Gieri Simonett.
\newblock Self-intersections for the surface diffusion and the volume-preserving mean curvature flow.
\newblock {\em Differential Integral Equations}, 13(7-9):1189--1199, 2000.

\bibitem{MR24}
Antoine Mellet and Michael Rozowski.
\newblock Volume-preserving mean-curvature flow as a singular limit of a diffusion-aggregation equation.
\newblock {\em arXiv:2408.14309}, 2024.

\bibitem{M14}
Gwena\"el Mercier.
\newblock Mean curvature flow with obstacles: A viscosity approach.
\newblock {\em arXiv:1409.7657}, 2014.

\bibitem{MN15}
Gwena\"el Mercier and Matteo Novaga.
\newblock Mean curvature flow with obstacles: existence, uniqueness and regularity of solutions.
\newblock {\em Interfaces Free Bound.}, 17:399--426, 2015.

\bibitem{M24}
Kuniyasu Misu.
\newblock A game-theoretic approach to the asymptotic behavior of solutions to an obstacle problem for the mean curvature flow equation.
\newblock {\em arXiv:2404.02682}, 2024.

\bibitem{MBRZ16}
Matthew~S. Mizuhara, Leonid Beryland, Volodymyr Rybalko, and Lei Zhang.
\newblock On an evolution equation in a cell motility model.
\newblock {\em Phys. D.}, 318/319:12--25, 2016.

\bibitem{M85}
Luciano Modica.
\newblock A gradient bound and a liouville theorem for nonlinear poisson equations.
\newblock {\em Comm. Pure Appl. Math.}, 38:679--684, 1985.

\bibitem{MR08}
Luca Mugnai and Matthias R\"oger.
\newblock The allen-cahn action functional in higher dimensions.
\newblock {\em Interfaces and Free Boundaries}, 10(1):45--78, 2008.

\bibitem{MSS16}
Luca Mugnai, Christian Seis, and Emanuele Spadaro.
\newblock Global solutions to the volume-preserving mean-curvature flow.
\newblock {\em Calc. Var. Partial Differential Equations}, 55(18), 2016.

\bibitem{N14}
Giacomo Nardi.
\newblock Schauder estimate for solutions of poisson’s equation with neumann boundary condition.
\newblock {\em Enseign. Math.}, 60(3/4):421--435, 2014.

\bibitem{NT25}
Katerina Nik and Keisuke Takasao.
\newblock On an obstacle problem for the brakke flow with a generalized right-angle boundary condition.
\newblock {\em SIAM. J. Math. Anal.}, 57(1):452--494, 2025.

\bibitem{RS92}
Jacob Rubinstein and Peter Sternberg.
\newblock Nonlocal reaction-diffusion equations and nucleation.
\newblock {\em IMA J. Appl. Math.}, 48(3):249--264, 1992.

\bibitem{S83}
Leon Simon.
\newblock {\em Lectures on Geometric Measure Theory}, volume~3.
\newblock Proceedings of the Centre for Mathematics and its Applications, Australian National University, 1983.

\bibitem{T17}
Keisuke Takasao.
\newblock Existence of weak solution for volume preserving mean curvature flow via phase field method.
\newblock {\em Indiana Univ. Math. J.}, 66(6):2015--2035, 2017.

\bibitem{T21}
Keisuke Takasao.
\newblock On obstacle problem for brakke's mean curvature flow.
\newblock {\em SIAM. J. Math. Anal.}, 53(6):6355--6369, 2021.

\bibitem{T23}
Keisuke Takasao.
\newblock The existence of a weak solution to volume preserving mean curvature flow in higher dimensions.
\newblock {\em Arch. Rational Mech. Anal.}, 247(52), 2023.

\bibitem{TT16}
Keisuke Takasao and Yoshihiro Tonegawa.
\newblock Existence and regularity of mean curvature flow with transport term in higher dimensions.
\newblock {\em Math. Ann.}, 366:857--935, 2016.

\bibitem{T03}
Yoshihiro Tonegawa.
\newblock Integrality of varifolds in the singular limit of reaction-diffusion equations.
\newblock {\em Hiroshima Math. J.}, 33:323--341, 2003.

\bibitem{W34}
Hassler Whitney.
\newblock Analytic extensions of differentiable functions defined in closed sets.
\newblock {\em Trans. Amer. Math. Soc.}, 36(1):63--89, 1934.

\end{thebibliography}

\end{document}